\documentclass{article}
\usepackage{arxiv}

\usepackage{mathtools}
\usepackage{tikz}
\usepackage{adjustbox}

\usepackage{natbib}

\usepackage{amsmath, amsfonts, amssymb, amsthm, dsfont}

\usepackage{mathrsfs}

\usepackage{soul}

\usepackage{bm}

\usepackage{booktabs} 

\definecolor{mypurple}{RGB}{0,0,0}
\definecolor{myblue}{RGB}{0,87,120}
\definecolor{aqua}{RGB}{0,0,0}
\definecolor{myorange}{RGB}{0,0,0}
\definecolor{mygray}{RGB}{255,255,255}

\usepackage{ulem}

\usepackage{graphicx}
\usepackage{color}
\usepackage{nicefrac}

\usepackage{scrextend}

\usepackage{titlesec}
\titleformat{\section}[hang]{\large\center\scshape}{\thesection.}{1em}{}
\titleformat{\subsection}[hang]{\large}{\thesubsection.}{1em}{}
\titleformat{\subsubsection}[hang]{}{\thesubsubsection.}{1em}{}

\usepackage[colorlinks=true, linkcolor=red, urlcolor=blue, citecolor=blue, psdextra, pdfencoding=auto]{hyperref}
\usepackage{bookmark}

\usepackage[ruled,vlined]{algorithm2e}

\newtheoremstyle{mytheoremstyle} 
    {0.3cm}                      
    {0cm}                        
    {\itshape}                   
    {}                           
    {\scshape}                   
    {: }                          
    {0em}                       
    {}  

\theoremstyle{mytheoremstyle}
\newtheorem{Theorem}{Theorem}

\newtheorem{Lemma}{Lemma}
\newtheorem*{Lemma*}{Lemma}
\newtheorem{Corollary}{Corollary}
\newtheorem{Proposition}{Proposition}

\renewenvironment{proof}{{\noindent \sc Proof:}}{\qed}

\newtheoremstyle{myExampleRemarkstyle} 
    {0.3cm}                    
    {0cm}                           
    {\itshape}                   
    {}                           
    {\scshape}                   
    {: }                          
    {0em}                       
    {}  

\theoremstyle{myExampleRemarkstyle}

\newtheorem{Assumption}{Assumption}

\renewcommand{\theAssumption}{\Alph{Assumption}}

\newtheorem{innercustomgeneric}{\customgenericname}
\providecommand{\customgenericname}{}
\newcommand{\newcustomtheorem}[2]{%
  \newenvironment{#1}[1]
  {%
   \renewcommand\customgenericname{#2}%
   \renewcommand\theinnercustomgeneric{##1}%
   \innercustomgeneric
  }
  {\endinnercustomgeneric}
}

\newcustomtheorem{customthm}{Theorem}
\newcustomtheorem{customlemma}{Lemma}
\usepackage{mathalfa}
\usepackage{oubraces}

\newtheoremstyle{simuStyle}
{0.3cm} 
{0cm} 
{} 
{} 
{\bfseries} 
{.} 
{0em} 
{} 

\theoremstyle{simuStyle}

\newtheoremstyle{stratStyle}
{0.3cm} 
{0cm} 
{} 
{} 
{\scshape} 
{: } 
{0em} 
{} 

\theoremstyle{stratStyle}

\DeclareSymbolFont{lettersA}{U}{txmia}{m}{it}
\DeclareMathSymbol{\real}{\mathord}{lettersA}{"92}
\DeclareMathSymbol{\field}{\mathord}{lettersA}{"83}

\def\real{{\rm I\!R}}

\def\tt{^{\rm T}}

\DeclareMathOperator*{\tr}{tr}

\DeclareMathOperator*{\Int}{Int}

\DeclareMathOperator*{\cov}{cov}

\def\0{{\bf 0}}

\def\btheta{{\bm{\theta}}}

\def\bpi{{\bm{\pi}}}

\def\x{{\bf x}}

\def\B{{\bf B}}

\DeclareMathOperator*{\var}{var}

\DeclareMathOperator*{\argzero}{argzero}

\def\bt{\bm{\theta}}

\def\bT{\bm{\Theta}}

\def\boxit#1{\vbox{\hrule\hbox{\vrule\kern3pt
          \vbox{\kern3pt#1\kern3pt}\kern3pt\vrule}\hrule}}

\definecolor{pinegreen}{rgb}{0.0, 0.47, 0.44}

\title{A General Approach for Simulation-based Bias Correction in High Dimensional Settings}

\author{St\'ephane Guerrier \\\
        Geneva School of Economics and Management and Faculty of Science\\
        University of Geneva\\
        \And
        Mucyo Karemera \\\\
        Geneva School of Economics and Management\\
        University of Geneva\\
        \And
        Samuel Orso \\\\
        Geneva School of Economics and Management\\
        University of Geneva\\
        \And
        Maria-Pia Victoria-Feser
        \\\\
        Geneva School of Economics and Management\\
        University of Geneva\\
        \And
        Yuming Zhang\\\\
        Geneva School of Economics and Management\\
        University of Geneva
}

\begin{document}


\let\refBKP\ref
\renewcommand{\ref}[1]{{\upshape\refBKP{#1}}}

    
\maketitle
 
%

\begin{abstract}
An important challenge in statistical analysis lies in controlling the bias of estimators due to the ever-increasing data size and model complexity. Approximate numerical methods and data features like censoring and misclassification often result in analytical and/or computational challenges when implementing standard estimators. As a consequence, consistent estimators may be difficult to obtain, especially in complex and/or high dimensional settings. In this paper, we study the properties of a general simulation-based estimation framework that allows to construct bias corrected consistent estimators. We show that the considered approach leads, under more general conditions, to stronger bias correction properties compared to alternative methods. Besides its bias correction advantages, the considered method can be used as a simple strategy to construct consistent estimators in settings where alternative methods may be challenging to apply. Moreover, the considered framework can be easily implemented and is computationally efficient. These theoretical results are highlighted with simulation studies of various commonly used models, including the negative binomial regression (with and without censoring) and the logistic regression (with and without misclassification errors). Additional numerical illustrations are provided in the supplementary materials.

\vspace{0.2cm}

\end{abstract}%

\newpage


\section{Introduction}
\label{sec:intro}

Point estimates resulting from parametric estimation are widely used in subsequent analyses, in particular for inference based on asymptotic theory or in simulation methods like the bootstrap \citep{Efro:79}. They are also used to compute, for example, the sample sizes in planning experiments, prediction errors in model assessment, and mean squared error in extreme value estimation. However, classical (consistent) estimators like the Maximum Likelihood Estimator (MLE) often suffer from severe finite sample bias, especially in high dimensional settings where the number of parameters $p$ is relatively large compared to the sample size $n$ (see e.g. \citealp{sur2019modern}). Moreover, with the ever-increasing model complexity, 
approximate methods 
have been commonly used, which, however, often lead to biased or even inconsistent estimators. Furthermore, data often exhibit features such as truncation, censoring or misclassification, which can also lead to inconsistent estimators if these mechanisms are ignored in the estimation process. In this case, a common approach is to consider marginal likelihood functions, which lead to estimators that can possibly suffer from (significant) bias, in finite sample and/or asymptotically, especially in high dimensional settings. 

Different approaches have been proposed to remedy, often separately, these issues. For example, bias correction can be achieved by using analytical approximations to the likelihood or the score functions (see e.g. \citealp{BrDa:08, Kosm:14} and the references therein). While these approaches are often considerably numerically efficient as they do not rely on simulations, they are typically model dependent and have mainly been studied in low dimensional settings. On the other hand, simulation-based methods (see e.g. \citealp{Efro:79, gourieroux1993indirect, mackinnon1998approximate}) are naturally more computationally challenging,  they are however broadly applicable and model independent. Moreover, these methods enjoy from desirable bias correction properties in finite samples, and some of them additionally allow to achieve asymptotic bias correction. More details regarding the simulation-based methods are presented in Section \ref{sec:existing:methods}, where we illustrate that these methods are intrinsically connected.

Given their connections, some of these methods can be put into a general simulation-based estimation framework (applicable in high dimensional settings) that can be used to either correct the bias of consistent estimators, or construct consistent bias corrected estimators starting from inconsistent ones. In this paper, we study the properties of the resulting estimators. First, we show that these estimators provide better bias correction results compared to alternative methods. For example, under plausible conditions, they even achieve asymptotically optimal bias correction in the sense that they are unbiased for sufficiently large, but finite, $n$. While enjoying from these desirable properties, the considered estimators appear not to pay a considerable price in terms of variance and typically enjoy from a reduced Root Mean Square Error (RMSE) in finite samples. Second, while being generally applicable to a broad set of models, these estimators are simple to implement. Indeed, one can easily construct consistent estimators while avoiding the analytical and/or numerical challenges typically entailed in standard procedures. This is especially useful when the data show some features such as censoring or misclassification, and when no alternative (consistent) estimators are available. Third, the simulation-based estimators we consider can be computed through an algorithm which is shown to converge exponentially fast. Therefore, they can be obtained in a numerically efficient manner even in large and complex data settings. Lastly, our theoretical and numerical findings are developed in high dimensional settings, which broadens their applicability. 

The rest of the paper is organized as follows. In Section~\ref{sec:existing:methods}, we discuss various existing simulation-based bias correction methods as well as their relations and equivalence. Based on this observation, we consider a general simulation-based framework that provides consistent and bias corrected estimators based on possibly inconsistent ones. In Section~\ref{sec:bias}, we investigate the bias correction properties of the resulting estimators based on both consistent and inconsistent estimators. We study their asymptotic and computational properties in Sections~\ref{sec:consist:asympnorm} and \ref{sec:computation}, respectively. Moreover, in Sections~\ref{sec:app:consist}~and~\ref{sec:app:inconsist} (and in the supplementary materials), we apply our approach to obtain consistent and bias corrected estimators for various commonly used models in high dimensional settings, based on consistent and inconsistent ones (when the data show censoring or misclassification), respectively. We find that the theoretical results are in line with the simulation studies, highlighting the advantages and the wide applicability of the general framework we study.

\section{Simulation-based Bias Correction Methods and their Connections}
\label{sec:existing:methods}

Various methods have been proposed to correct finite sample and asymptotic biases, often separately. While preventive methods have been proposed (see e.g. \citealp{Firt:93}), a widely used corrective approach is to consider an initial  estimator which is readily available but typically biased, and define a bias corrected estimator as a (possibly implicit) function of the initial one \citep[see e.g.][and the references therein]{kosmidis2014bias}. Specifically, suppose that we observe a random sample of size $n$ generated under the true model $F_{\bt_0}$ (possibly conditional on a set of fixed covariates), where $\bt_0\in\bm\Theta\subset\real^p$ is the parameter vector we wish to estimate. We denote the initial estimator as $\widehat{\bpi} \vcentcolon= \widehat{\bpi}(\bt_0, n) \in \bT$. We assume that $\bpi(\bt, n) \vcentcolon = \mathbb{E}[\widehat{\bpi}(\bt, n)]$ exists, where $\mathbb{E}[\cdot]$ denotes the expectation under $F_{\bt}$. We define the (generally unknown) bias function as $\mathbf{d}(\bt, n) \vcentcolon = {\bm{\pi}}(\bt, n) - \bt$, which might also include the asymptotic bias when the initial estimator $\widehat{\bpi}$ is inconsistent.  Then, the resulting corrected estimator $\widehat{\bt}_{C}$ can be expressed as $\widehat{\bt}_{C} \vcentcolon= \mathbf{g}(\widehat{\bpi})$,
where the function $\mathbf{g}(\bt)$ varies depending on the approaches. For example, in the simple situation where $\widehat{\bpi}$ is consistent and the bias function $\mathbf{d}(\bt, n)$ is known, a natural choice is $\mathbf{g}(\widehat{\bpi}) = \widehat{\bpi} - \mathbf{d}(\widehat{\bpi}, n)$. When $\mathbf{d}(\bt, n)$ is unknown, approximations could be used instead, such as an asymptotic analytical approximation of $\mathbf{d}(\bt, n)$ (see e.g. \citealp{cox1968general}). Similarly, simulation-based approximation of $\mathbf{d}(\bt, n)$, say $\mathbf{d}^*(\bt, n)$, can also be used. For example, considering $\mathbf{d}^*(\bt, n) \vcentcolon = \bpi^*(\bt,n) - \bt$, the (parametric) Bootstrap Bias Corrected (BBC) estimator (see e.g. \citealp{efron1994introduction}) is given by
\begin{equation}
    \widetilde{\bt}\vcentcolon = \widehat{\bpi} - \mathbf{d}^*(\widehat{\bpi}, n) = 2\widehat{\bpi} - {\bm{\pi}}^*(\widehat{\bpi}, n),
    \label{eqn:def:BBC}
\end{equation}
where
\begin{equation}
 \bpi^*(\bt,n) \vcentcolon = \frac{1}{H} \sum_{h=1}^H \widehat{\bpi}_h(\bt, n),
    \label{eqn:pi-star}
\end{equation}
with $\widehat{\bpi}_h(\bt,n)$ denoting the value of the initial estimator computed on the $h^{th}$ simulated sample under $F_{\bt}$. Alternatively, and still considering $\widehat{\bpi}$ consistent, \cite{mackinnon1998approximate} put forward the Nonlinear Bias Correcting (NBC) estimator (see also \citealp{gourieroux2000} and \citealp{smith1997fractional}), which is defined as
\begin{equation}
    \widehat{\bt}_{\mbox{\tiny\text{NBC}}} \vcentcolon = \argzero_{\bt \in \bT}\; \bt - \widehat{\bpi} + \mathbf{d}(\bt, n). 
     \label{eqn:def:NBC-theoretical}
\end{equation}
This implicit bias correction method has been used, for example, in \cite{godwin2019analytic} where an analytical approximation of $\mathbf{d}(\bt, n)$ is considered to compute the MLE of a generalized Pareto distribution. Similarly, $\mathbf{d}(\bt, n)$ can also be approximated by means of simulations, leading to 
\begin{equation}
    \widehat{\bt}_{\mbox{\tiny\text{NBC*}}} \vcentcolon = \argzero_{\bt \in \bT}\; \bt - \widehat{\bpi} + \mathbf{d}^*(\bt, n). 
    \label{eqn:def:NBC}
\end{equation}
Interestingly, this simulation-based counterpart of the NBC estimator is closely related to the BBC and other plug-in estimators, as $\widehat{\bt}_{\mbox{\tiny\text{NBC*}}}$ satisfies $\widehat{\bt}_{\mbox{\tiny\text{NBC*}}} = \widehat{\bpi} - \mathbf{d}^*(\widehat{\bt}_{\mbox{\tiny\text{NBC*}}}, n)$, and therefore, it corresponds to the BBC estimator but with the arguably better plug-in value $\widehat{\bt}_{\mbox{\tiny\text{NBC*}}}$.

When the initial estimator $\widehat{\bpi}$ is inconsistent but the bias function $\mathbf{d}(\bt,n)$ (or $\bpi(\bt,n)$ equivalently) is known, a natural strategy is to consider an indirect approach to construct a consistent Classical Minimum Distance (CMD) estimator (see e.g. \citealp{newey1994large}), which, under plausible conditions, is defined as
\begin{equation}
    \widehat{\bt}_{\mbox{\tiny\text{CMD}}}  \vcentcolon = \argzero_{\bt \in \bT}\; \widehat{\bpi} - \bpi(\bt, n).
    \label{eqn:def:CMD}
\end{equation}
Similarly, when $\mathbf{d}(\bt,n)$ is unknown, a simulation-based counterpart of $\widehat{\bt}_{\mbox{\tiny\text{CMD}}}$ can be obtained by replacing $\bpi(\bt, n)$ with $\bpi^*(\bt, n)$ defined in \eqref{eqn:pi-star}, leading to
\begin{equation}
    \widehat{\bt} \vcentcolon = \argzero_{\bt \in \bT}\; \widehat{\bpi} - \bpi^*(\bt, n).
    \label{eqn:def:JINI}
\end{equation}
This estimator is a special case of an indirect inference estimator proposed in \cite{gourieroux1993indirect} and \cite{smith1993estimating} (see also \citealp{mcfadden1989method, gallant1996moments}), which we call the Just Identified iNdirect Inference (JINI) estimator. Interestingly, the simulation-based NBC estimator $\widehat{\bt}_{\mbox{\tiny\text{NBC*}}}$ in \eqref{eqn:def:NBC} is equivalent to the JINI  estimator as 
\begin{equation*}
    \widehat{\bt}  \vcentcolon= \argzero_{\bt \in \bT}\; \widehat{\bpi} - \bpi^*(\bt, n) = \argzero_{\bt \in \bT}\; \widehat{\bpi} - \bt -\mathbf{d}^*(\bt,n) =\vcentcolon \widehat{\bt}_{\mbox{\tiny\text{NBC*}}}. 
\end{equation*}
Naturally, this equivalence also holds between the CMD estimator in \eqref{eqn:def:CMD} and the NBC estimator in \eqref{eqn:def:NBC-theoretical} when $\mathbf{d}(\bt,n)$ is known.  \cite{kuk1995asymptotically} independently produced the same idea to correct asymptotic bias when estimating the Generalized Linear Mixed Model (GLMM) with approximate estimating equations. He proposed to iterate the BBC estimator by updating the plug-in value as follows:
\begin{equation}
    \bt^{(k+1)}  = \bt^{(k)} + \left\{\widehat{\bpi} - \bpi^\ast\left(\bt^{(k)},n\right)\right\},
\label{eqn:IB}
\end{equation}
with $\bt^{(0)} \vcentcolon = \widehat{\bpi}$. With one step of \eqref{eqn:IB}, the BBC estimator is obtained and therefore, such approach is named the Iterative Bootstrap (IB), which is different from the (more computationally intensive) iterated bootstrap \citep[see e.g.][and the references therein]{Hall:92,ChHa:15}. Under plausible conditions (discussed in Section \ref{sec:computation}), this sequence converges and its limit $\widehat{\bt}_{\mbox{\tiny\text{IB}}}\vcentcolon =\underset{k \to \infty}{\lim}\bt^{(k)}$ is actually the fixed point of
\begin{equation}
    \bm{T}\left(\bt, n\right)\vcentcolon = \bt + \left\{\widehat{\bpi} - \bpi^\ast\left(\bt,n\right)\right\}.
    \label{eqn:def:function:T}
\end{equation}
Hence, it is equivalent to the JINI estimator as 
\begin{equation*}
    \widehat{\bpi} = \bpi^*(\widehat{\bt}_{\mbox{\tiny\text{IB}}},n) \Longleftrightarrow \widehat{\bt}_{\mbox{\tiny\text{IB}}} = \argzero_{\bt \in \bT}\; \widehat{\bpi} - \bpi^*(\bt, n) =\vcentcolon \widehat{\bt},
\end{equation*}
and therefore,
\begin{equation}
\label{eqn:equivalence}
    \widehat{\bt} = \widehat{\bt}_{\mbox{\tiny\text{IB}}}  =  \widehat{\bt}_{\mbox{\tiny\text{NBC*}}}.
\end{equation}
%
%
Some of these connections have been preliminarily studied in \cite{guerrier2018simulation}, which showed the equivalence of the IB and the JINI estimators when considering consistent initial estimators in low dimensional settings. Under plausible conditions (discussed in Section~\ref{sec:computation}), the relation in \eqref{eqn:equivalence} holds and suggests that these estimators can be put into a computationally efficient unifying framework to construct consistent bias corrected estimators starting from a possibly inconsistent initial estimator $\widehat{\bpi}$. Indeed, the properties of $\widehat{\bt}$ (an indirect inference estimator) are well established as a method for constructing consistent estimators starting from inconsistent ones. On the other hand, the properties of $\widehat{\bt}_{\mbox{\tiny\text{NBC*}}}$, preliminarily studied in \cite{mackinnon1998approximate}, suggest that it can potentially achieve advantageous bias correction performance, while the IB approach is easy to implement and  appears to be computationally efficient. The  properties of the JINI estimator (bias correction, consistency,  asymptotic normality and computational efficiency) are formally studied in high dimensional settings in the remaining sections.

In addition, assuming $\bpi(\bt, n)$ is bijective in $\bt$, so that its inverse, which we denote as $\bpi^{-1}_n(\bt)$, exists, both estimators $\widehat{\bt}_{\mbox{\tiny\text{CMD}}}$ and $\widehat{\bt}_{\mbox{\tiny\text{NBC}}}$ can be written 
as $\widehat{\bt}_{\mbox{\tiny\text{CMD}}} = \widehat{\bt}_{\mbox{\tiny\text{NBC}}} = \mathbf{g} \left(\widehat{\bpi}\right) = \bpi^{-1}_n \left( \widehat{\bpi} \right)$.
Interestingly, both $\widehat{\bt}_{\mbox{\tiny\text{CMD}}}$ and $\widehat{\bt}_{\mbox{\tiny\text{NBC}}}$ possess the property that 
\begin{equation}
\label{eqn:target:unbias}
    \mathbf{g} \left(\mathbb{E}[\widehat{\bpi}]\right) = \bpi^{-1}_n \left( \mathbb{E}[\widehat{\bpi}]\right) = \bt_0,
\end{equation}
and therefore, when $\bpi^{-1}_n(\bt)$ is continuous, it implies that the realizations of $\widehat{\bpi}$ ``close'' to its expected value would lead to $\mathbf{g} \left(\widehat{\bpi}\right) \approx \bt_0$.  Naturally, this property is different from unbiasedness (i.e. $\mathbb{E}[\mathbf{g} \left(\widehat{\bpi}\right) ] =\bt_0$) and they do not imply each other as $\mathbb{E}[\mathbf{g} \left(\widehat{\bpi}\right) ] \neq \mathbf{g} \left(\mathbb{E}[\widehat{\bpi}]\right)$. When the bias function is linear, however, the two properties are equivalent, suggesting that the resulting estimator may enjoy advantageous bias correction properties, especially when the bias function is (approximately or locally) linear. 
Moreover, the simulation error can be made arbitrarily small by considering a sufficiently large number of simulated samples $H$. Therefore, the difference between $\widehat{\bt}$ and $\widehat{\bt}_{\mbox{\tiny\text{CMD}}}$ can be negligible, and under appropriate conditions, simulation-based estimators essentially possess the same properties as their analytical counterparts. 
On the other hand, the BBC estimator $\widetilde{\bt}$ (or other estimators based on explicit plug-in) does not satisfy \eqref{eqn:target:unbias} in general. Indeed, informally we have 
\begin{equation*}
    \mathbb{E}[\widehat{\bpi}] - \mathbf{d}^*\left(\mathbb{E}[\widehat{\bpi}], n\right) \underset{\text{as} \; H \to \infty}{\longrightarrow} \mathbb{E}[\widehat{\bpi}] - \mathbf{d}\left(\mathbb{E}[\widehat{\bpi}],n\right) \neq \mathbb{E}[\widehat{\bpi}] - \mathbf{d}\left(\bt_0,n\right) = \bt_0.
\end{equation*}
Therefore, advantageous bias correction properties are expected for the JINI estimator $\widehat{\bt}$ (and equivalently $\widehat{\bt}_{\mbox{\tiny\text{NBC*}}}$ and $\widehat{\bt}_{\mbox{\tiny\text{IB}}}$), possibly better than alternative methods like the BBC estimator.

\section{Bias Correction Properties}
\label{sec:bias}

As discussed in Section \ref{sec:existing:methods}, we recall that $\widehat{\bpi}(\bt_0, n)$ (or simply $\widehat{\bpi}$) is the initial estimator of $\bt_0$ computed on the observed sample, and that $\bpi(\bt_0, n) = \mathbb{E}[\widehat{\bpi}(\bt_0, n)]$. Moreover, we consider the decomposition of a generic bias function $\mathbf{d}(\bt, n) = \mathbf{a}(\bt) + \mathbf{b}(\bt, n).$ Denoting $a_i$ as the $i^{th}$ entry of a generic vector $\mathbf{a} \in \real^p$, the vector $\mathbf{a}(\bt)\in\real^p$ denotes the asymptotic bias in the sense that $a_{i}(\bt)\vcentcolon = \underset{n \to \infty}{\lim}\, d_{i}(\bt, n)$, for all $i=1,\ldots p$, and $\mathbf{b}(\bt, n)$ denotes the finite sample bias given by $\mathbf{b}(\bt, n) \vcentcolon = \mathbf{d}(\bt, n) - \mathbf{a}(\bt)$. We also denote $\bpi(\bt) \vcentcolon = \bt + \mathbf{a}(\bt)$, then we can write 
\begin{equation}
\label{eqn:bias-decomp}
\widehat{\bpi}(\bt,n) = \underbrace{\overunderbraces{&&\br{2}{\mathbf{d}(\bt,n)}}%
{&\bt + &\mathbf{a}(\bt)& + \mathbf{b}(\bt,n)}%
{&\br{2}{\bpi(\bt)}&}}_{\bpi(\bt,n)}  + \mathbf{v}(\bt,n),
\end{equation} 
where $\mathbf{v}(\bt, n) \vcentcolon = \widehat{\bpi}(\bt,n) - \bpi(\bt,n)$ is a zero mean random vector.  Our assumption framework discussed later in this section implies that $\widehat{\bpi}(\bt,n)$ converges (in probability) to $\bpi(\bt)$. Therefore, the decomposition given in \eqref{eqn:bias-decomp} allows to study the bias properties of the JINI estimator defined in \eqref{eqn:def:JINI}, when $\widehat{\bpi}(\bt_0, n)$ is a consistent estimator of $\bt_0$ (i.e. $\mathbf{a}(\bt_0)= \0$, see Section \ref{sec:bias:consist}), and when $\widehat{\bpi}(\bt_0, n)$ is inconsistent (i.e. $\mathbf{a}(\bt_0)\neq \0$, see Section \ref{sec:bias:inconsist}). 


\subsection{Consistent Initial Estimators}
\label{sec:bias:consist}

In this section, we analyze the bias correction properties of the JINI estimator when the initial estimator $\widehat{\bpi}(\bt_0,n)$ is consistent. The results are based on the following assumptions.

\setcounter{Assumption}{0}
\renewcommand{\theHAssumption}{otherAssumption\theAssumption}
\renewcommand\theAssumption{\Alph{Assumption}}
\begin{Assumption}
\label{assum:A}
The parameter space $\bm\Theta$ is a compact convex subset of $\real^p$ and $\bt_0\in\Int(\bm\Theta)$.
\end{Assumption}

\setcounter{Assumption}{1}
\renewcommand\theAssumption{\Alph{Assumption}}
\begin{Assumption}
\label{assum:B}
The variance of $\mathbf{v} \left(\bt,n \right)$ exists and is finite for any $\bt \in \bm\Theta$. Moreover, there exists some $\alpha >0$ such that $\lVert\mathbf{v} \left(\bt,n\right)\rVert_\infty = \mathcal{O}_{\rm p}(n^{-\alpha})$ and $p = o(n^{2\alpha})$.
\end{Assumption}
\vspace{0.3cm}

Assumption \ref{assum:A} is mild and is typically used in most settings where estimators have no closed-form solutions. Indeed, the compactness of $\bT$ and $\bt_0 \in \Int(\bm\Theta)$ are common regularity conditions for the consistency and asymptotic normality of  standard estimators such as the MLE, respectively (see e.g. \citealp{newey1994large}). The convexity condition ensures that expansions can be made between $\bt_0$ and an arbitrary point in $\bT$, and it can always be relaxed in low dimensions as well as in some high dimensional settings. Hence, Assumption~\ref{assum:A} may already be satisfied depending on the model settings and the requirements already brought in by taking $\widehat{\bpi}(\bt_0,n)$ to be consistent (and possibly asymptotic normally distributed). Assumption \ref{assum:B} imposes a (common) restriction on the variance of the initial estimator. It is also frequently employed and typically very mild. In the common situation where $\alpha = 1/2$, Assumption~\ref{assum:B} would require that $p/n \to 0$.

\setcounter{Assumption}{2}
\renewcommand\theAssumption{\Alph{Assumption}}
\begin{Assumption}
\label{assum:C}
The finite sample bias function $\mathbf{b}(\bt, n)$ is once continuously differentiable in $\bt \in \bm\Theta$. Let $\mathbf{B}(\bt,n) \vcentcolon = \partial \mathbf{b}(\bt, n) / \partial \bt\tt$, there exists some $\beta > \alpha$, with $\alpha$ stated in Assumption \ref{assum:B}, such that $\|\mathbf{b}(\bt, n)\|_\infty = \mathcal{O}(n^{-\beta})$ and $\|\mathbf{B}(\bt, n)\|_\infty = \mathcal{O}(n^{-\beta})$. 
\end{Assumption}
\vspace{0.3cm}

Assumption \ref{assum:C} is mainly used to ensure that the rate of the bias, as well as the rate of its derivative, is dominated by the standard error rate. The form of the finite sample bias $\mathbf{b}(\bt, n)$ is typically unknown, especially in complex models. In fact, depending on different specifications on the form of $\mathbf{b}(\bt, n)$, we can obtain different bias correction properties of the JINI estimator (and the BBC estimator). Therefore, we further specify Assumption \ref{assum:C} into the following three increasingly general versions.

\setcounter{Assumption}{2}
\renewcommand\theAssumption{\Alph{Assumption}.1}
\begin{Assumption}
\label{assum:C1}
The function $\mathbf{b}(\bt, n)$ is such that $\mathbf{b}(\bt,n)=\mathbf{B}(n)\bt + \mathbf{c}(n)$, where $\mathbf{B}(n) \in \real^{p \times p}$ and $\mathbf{c}(n) \in \real^p$. Moreover, there exist some $\beta_1$ and $\beta_2$ such that $\min(\beta_1, \beta_2) = \beta$ where $\beta$ is stated in Assumption \ref{assum:C}, $\|\mathbf{B}(n)\|_\infty = \mathcal{O}(n^{-\beta_1})$ and $\|\mathbf{c}(n)\|_\infty = \mathcal{O}(n^{-\beta_2})$.
\end{Assumption}
\setcounter{Assumption}{2}
\renewcommand\theAssumption{\Alph{Assumption}.2}
\begin{Assumption}
\label{assum:C2}
The function $\mathbf{b}(\bt, n)$ is such that $\mathbf{b}(\bt, n) = \mathbf{B}(n)\bt + \mathbf{c}(n) + \mathbf{r}(\bt, n)$, where $\mathbf{B}(n) \in \real^{p \times p}$, $\mathbf{c}(n) \in \real^p$ and $\mathbf{r}(\bt, n) \in \real^p$. There exist some $\beta_1$, $\beta_2$ and $\beta_3$ such that $\min(\beta_1, \beta_2, \beta_3) = \beta$ where $\beta$ is stated in Assumption \ref{assum:C}, $\|\mathbf{B}(n)\|_\infty = \mathcal{O}(n^{-\beta_1})$ and $\|\mathbf{c}(n)\|_\infty = \mathcal{O}(n^{-\beta_2})$. Moreover, there exist some $\beta_4 > \beta_3$ such that $\mathbf{r}(\bt, n)$ can be expressed as $\mathbf{r}(\bt, n) = \left(\sum_{j=1}^p \sum_{k=1}^p r_{ijk} \theta_j \theta_k n^{-\beta_3}\right)_{i=1,\ldots,p} + \mathbf{e}(\bt, n),$
where $\|\mathbf{e}(\bt, n)\|_\infty = \mathcal{O}(n^{-\beta_4})$ and $\underset{i=1,\ldots,p}{\max} \sum_{j=1}^p \sum_{k=1}^p |r_{ijk}| = \mathcal{O}(1)$.
\end{Assumption}
\setcounter{Assumption}{2}
\renewcommand\theAssumption{\Alph{Assumption}.3}
\begin{Assumption}
\label{assum:C3}
The function $\mathbf{b}(\bt, n)$ is such that $\mathbf{b}(\bt, n) = \mathbf{B}(n)\bt + \mathbf{c}(n) + \mathbf{r}(\bt, n)$, where $\mathbf{B}(n) \in \real^{p \times p}$, $\mathbf{c}(n) \in \real^p$ and $\mathbf{r}(\bt, n) \in \real^p$. There exist some $\beta_1$, $\beta_2$ and $\beta_3$ such that $\min(\beta_1, \beta_2, \beta_3) = \beta$ where $\beta$ is stated in Assumption \ref{assum:C}, $\|\mathbf{B}(n)\|_\infty = \mathcal{O}(n^{-\beta_1})$, $\|\mathbf{c}(n)\|_\infty = \mathcal{O}(n^{-\beta_2})$ and $\|\mathbf{r}(\bt, n)\|_\infty = \mathcal{O}(n^{-\beta_3})$.
\end{Assumption}
\vspace{0.3cm}

Assumptions \ref{assum:C1}, \ref{assum:C2} and \ref{assum:C3} 
essentially consider that the finite sample bias $\mathbf{b}(\bt, n)$ is the sum of a linear term $\mathbf{B}(n)\bt$, a non-linear term $\mathbf{r}(\bt, n)$ and a constant term $\mathbf{c}(n)$, with increasingly general restrictions imposed on the non-linear term $\mathbf{r}(\bt, n)$. Indeed, Assumption~\ref{assum:C1} applies to the cases where $\mathbf{b}(\bt,n)$ is linear, i.e. $\mathbf{r}(\bt, n) = \0$. While this assumption may appear restrictive, it is often an accurate (local) approximation \citep[see e.g.][]{mackinnon1998approximate, sur2019modern}. On the other hand, Assumption \ref{assum:C2} states a quadratic leading order form of $\mathbf{b}(\bt,n)$, which is more general
than in Assumption~\ref{assum:C1}. This assumption is usually satisfied when $\mathbf{b}(\bt,n)$ is a sufficiently smooth function of $\bt/n$. Like Assumption \ref{assum:C1}, Assumption \ref{assum:C2} is arguably not easy to verify but it often appears to be a suitable approximation of $\mathbf{b}(\bt,n)$. Indeed, in Supplementary Material~\ref{supp:bias}, we numerically illustrate that these assumptions are plausible in some practical situations. Similar approximations are also commonly used to assess the bias of simulation-based methods (see e.g. \citealp{tibshirani1993introduction}). Assumption~\ref{assum:C3} further relaxes the restriction imposed on the non-linear term $\mathbf{r}(\bt,n)$ and states that it is known only up to the order. Finally, Assumption \ref{assum:C} is our most general assumption on the finite sample bias. Indeed, it only assumes a certain degree of smoothness and a specific order for $\mathbf{b}(\bt,n)$, without further information on its decomposition. Therefore, these assumptions have the following implications: 
\begin{center}
    \ref{assum:C1} $\Rightarrow$ \ref{assum:C2} $\Rightarrow$ \ref{assum:C3} $\Rightarrow$ \ref{assum:C}.
\end{center}

Our assumption framework is for the most part mild and likely to be satisfied (at least approximately) in most practical situations. However, our conditions are not necessarily the weakest possible in theory and may be further relaxed, or transformed into equivalent conditions. For example, the infinite norm conditions in these assumptions can possibly be replaced with other suitable conditions (such as with pointwise orders), which in turn will alter the restriction on $p/n$. We however do not attempt to pursue the weakest possible conditions to avoid overly technical treatments in establishing the theoretical results.

In Theorem \ref{thm:bias:finite} below, we provide the bias orders of the JINI estimator under different bias conditions, and in Proposition \ref{prop:bias:finite}, we provide the bias orders for the BBC estimator. The proofs are provided in Appendices~\ref{app:proof:JINI:bias:finite} and \ref{app:proof:BBC:bias:finite}, respectively. 

\begin{Theorem}
\label{thm:bias:finite}
Under Assumptions \ref{assum:A}, \ref{assum:B} and
\begin{enumerate}
    \item Assumption \ref{assum:C1}, the JINI estimator $\widehat{\bt}$ satisfies $\left\|\mathbb{E}[\widehat{\bt}] - \bt_0\right\|_2 = 0,$ for sufficiently large (but finite) $n$. 
\item Assumption  \ref{assum:C2}, the JINI estimator $\widehat{\bt}$ satisfies $\left\|\mathbb{E}[\widehat{\bt}] - \bt_0\right\|_2 = \mathcal{O}\left(p^{1/2}n^{-\gamma_2}\right),$ where $\gamma_2 \vcentcolon = \max\left\{\min(2\alpha+\beta_3, 2\beta_3, \beta_4), \alpha+\beta \right\}.$
\item Assumption  \ref{assum:C3}, the JINI estimator $\widehat{\bt}$ satisfies $\left\|\mathbb{E}[\widehat{\bt}] - \bt_0\right\|_2 = \mathcal{O}\left(p^{1/2}n^{-\gamma_3}\right),$ where $\gamma_3 \vcentcolon = \max(\beta_3, \alpha + \beta). $
\item Assumption  \ref{assum:C}, the JINI estimator $\widehat{\bt}$ satisfies $\left\|\mathbb{E}[\widehat{\bt}] - \bt_0\right\|_2 = \mathcal{O}\left(p^{1/2}n^{-(\alpha+\beta)}\right).$
\end{enumerate}
Finally, we have $\gamma_2 \geq \gamma_3 \geq \alpha+\beta$.
\end{Theorem}

\begin{Proposition}
\label{prop:bias:finite}
Under Assumptions \ref{assum:A}, \ref{assum:B}, and \begin{enumerate}
    \item Assumption \ref{assum:C1}, the BBC estimator $\widetilde{\bt}$ satisfies $\left\|\mathbb{E}[\widetilde{\bt}] - \bt_0\right\|_2 = \mathcal{O}\left(p^{1/2}n^{-\upsilon_1}\right),$ where $\upsilon_1 \vcentcolon = \beta + \beta_1.$
\item Assumption \ref{assum:C2}, the BBC estimator $\widetilde{\bt}$ satisfies $\left\|\mathbb{E}[\widetilde{\bt}] - \bt_0\right\|_2 = \mathcal{O}\left(p^{1/2}n^{-\upsilon_2}\right),$ 
where $\upsilon_2 \vcentcolon = \max\left\{\min(\beta+\beta_1, \beta+\beta_3, 2\alpha+\beta_3, \beta_4), \alpha+\beta \right\}$.
\item Assumption \ref{assum:C3}, the BBC estimator $\widetilde{\bt}$ satisfies $\left\|\mathbb{E}[\widetilde{\bt}] - \bt_0\right\|_2 = \mathcal{O}\left(p^{1/2}n^{-\upsilon_3}\right),$ where $\upsilon_3 \vcentcolon = \max\left\{\min(\beta+\beta_1, \beta_3), \alpha+\beta\right\}.$
\item Assumption \ref{assum:C}, the BBC estimator $\widetilde{\bt}$ satisfies $\left\|\mathbb{E}[\widetilde{\bt}] - \bt_0\right\|_2 = \mathcal{O}\left(p^{1/2}n^{-(\alpha+\beta)}\right). $
\end{enumerate}
Finally, we have $\upsilon_1 \geq \upsilon_2 \geq \upsilon_3 \geq \alpha+\beta$.
\end{Proposition}

\vspace{0.4cm}
Theorem \ref{thm:bias:finite} and Proposition \ref{prop:bias:finite} provide the bias correction properties of the JINI and BBC estimators, respectively, under assumptions of increasing generality. In Corollary~\ref{cor:JINI:BBC}, we compare the bias orders of the JINI and the BBC estimators under the same assumption framework. We conclude that the JINI estimator provides stronger guarantees in terms of bias correction compared to the BBC estimator, as the bias order of the JINI estimator is always smaller than (or equal to) the one of the BBC estimator under the same assumptions. The proof of Corollary \ref{cor:JINI:BBC} is provided in Appendix~\ref{app:bias:order:compare}.

\begin{Corollary}
\label{cor:JINI:BBC} 
Under Assumptions \ref{assum:A} and \ref{assum:B}, and any one of Assumptions \ref{assum:C}, \ref{assum:C1}, \ref{assum:C2} and \ref{assum:C3}, we have 
$
\left\|\mathbb{E}[\widehat{\bt}] - \bt_0\right\|_2 + \left\|\mathbb{E}[\widetilde{\bt}] - \bt_0\right\|_2 = \mathcal{O}\left(\left\|\mathbb{E}[\widetilde{\bt}] - \bt_0\right\|_2\right).
$
\end{Corollary}
\vspace{0.4cm}

As a further illustration of the bias correction properties of these estimators, we present the results of Theorem \ref{thm:bias:finite} and Proposition \ref{prop:bias:finite} in Table \ref{tab:bias:finite} based on the (arguably) typical values $\alpha=1/2$, $\beta_1=1$, $\beta_2 > \beta_1$ and $\beta_3=2$, at least in low dimensional settings. Therefore, our results lead to the following conclusions. \textit{(i)} The bias order of the JINI estimator changes significantly depending on the form of $\mathbf{b}(\bt,n)$. As expected, stronger requirements imposed on $\mathbf{b}(\bt,n)$ lead to better bias correction properties of the JINI estimator. In the case where $\mathbf{b}(\bt,n)$ is linear, the JINI estimator achieves an asymptotically optimal bias correction (i.e unbiasedness for sufficiently large, but finite, $n$). \textit{(ii)} The bias order of the BBC estimator also depends on the form of $\mathbf{b}(\bt,n)$ but up to a certain level, and asymptotically optimal bias correction is not plausible. \textit{(iii)} Overall, the JINI estimator leads to better bias correction properties than the BBC estimator. 

\begin{table}[!tb]
\caption{Bias orders of the JINI estimator $\widehat{\bt}$ and the BBC estimator $\widetilde{\bt}$, with $\alpha=1/2$, $\beta_1=1$, $\beta_2 > \beta_1$ and $\beta_3=2$, according to the form of the finite sample bias considered in Theorem \ref{thm:bias:finite} and Proposition \ref{prop:bias:finite}. The value of 0 for the bias is for sufficiently large (but finite) $n$. $^\star$Under additional requirements, we have $\|\mathbb{E}[\widehat{\bt}] - \bt_0\|_2 = \mathcal{O}\left(p^{1/2}n^{-2}\right)$ and $\|\mathbb{E}[\widetilde{\bt}] - \bt_0\|_2  = \mathcal{O}\left(p^{1/2}n^{-2}\right)$ (see Appendix \ref{app:proof:BBC:bias:finite} for more details).}

\setlength{\tabcolsep}{30pt}
\renewcommand{\arraystretch}{1.25} 
\centering
\begin{tabular}{lcc}
\toprule
Under assumptions & $\left\|\mathbb{E}[\widehat{\bt}] - \bt_0\right\|_2$ & $\left\|\mathbb{E}[\widetilde{\bt}] - \bt_0\right\|_2$ \\
\midrule
\ref{assum:A}, \ref{assum:B} and \ref{assum:C1} & $0$  & $\mathcal{O}\left(p^{1/2}n^{-2}\right)$ \phantom{$^\star$} \\
\ref{assum:A}, \ref{assum:B} and \ref{assum:C2} & $\mathcal{O}\left(p^{1/2}n^{-3}\right)$ \phantom{$^\star$} & $\mathcal{O}\left(p^{1/2}n^{-2}\right)$ \phantom{$^\star$} \\
\ref{assum:A}, \ref{assum:B} and \ref{assum:C3} & $\mathcal{O}\left(p^{1/2}n^{-2}\right)$ \phantom{$^\star$} & $\mathcal{O}\left(p^{1/2}n^{-2}\right)$ \phantom{$^\star$} \\
\ref{assum:A}, \ref{assum:B} and \ref{assum:C} & \;\;$\mathcal{O}\left(p^{1/2}n^{-1.5}\right)$ $^\star$ & \; $\mathcal{O}\left(p^{1/2}n^{-1.5}\right)$ $^\star$ \\
\bottomrule
\end{tabular}
\label{tab:bias:finite}    
\end{table}

\subsection{Inconsistent Initial Estimators}
\label{sec:bias:inconsist}

In general, most of the bias correction methods, including the BBC estimator, are not valid when the initial estimator $\widehat{\bpi}(\bt_0, n)$ is inconsistent. However, we prove that starting from inconsistent initial estimators, the JINI estimator can both achieve the consistency (see Section \ref{sec:consist:asympnorm}) and enjoy from desirable bias correction properties, possibly achieving asymptotically optimal bias correction. Before presenting the results, we impose the following assumption on the asymptotic bias function $\mathbf{a}(\bt)$.

\setcounter{Assumption}{3}
\renewcommand\theAssumption{\Alph{Assumption}}
\begin{Assumption}
\label{assum:D}
The asymptotic bias function $\mathbf{a}(\bt)$ satisfies the followings:
    \begin{enumerate}
        \item $\mathbf{a}(\bt)$ is a contraction map in that for all $\bt_1, \bt_2 \in \bT$ with $\bt_1 \neq \bt_2$, there exists some $0 < M < 1$ such that $\lVert\mathbf{a}(\bt_2)-\mathbf{a}(\bt_1)\rVert_2 \leq M \|\bt_2 - \bt_1\|_2$.
        \item $\mathbf{a}(\bt)$ is twice continuously differentiable. Moreover, let $\mathbf{A}(\bt) \vcentcolon = \partial \mathbf{a}(\bt) / \partial \bt\tt$, then $\mathbf{A}(\bt)$ is such that $\left\|\mathbf{A}(\bt)\right\|_\infty<1$. Let $\mathbf{G}_i(\bt) \vcentcolon = \partial^2 a_i(\bt) / \partial \bt\tt \partial \bt$, then $$\underset{i=1,\ldots,p}{\max} \sum_{j=1}^p \sum_{k=1}^p |G_{ijk}(\bt)| = \mathcal{O}(1).$$
    \end{enumerate}
\end{Assumption} 
\vspace{-0.2cm}
Assumption \ref{assum:D} is a plausible requirement when considering situations where the asymptotic bias is relatively ``small''. This is typically the case when the data exhibit small departures from the assumed model, such as censoring or misclassification (as illustrated in Section \ref{sec:app:inconsist} and Supplementary Material~\ref{supp:cens-pois}). This may also be the case, for example, when neglecting some term in the likelihood functions or the estimating equations of complex models. Moreover, this condition guarantees the identifiability of the JINI estimator\footnote{Indeed, since $\mathbf{a}(\bm{\theta})$ is a contraction map, it is continuous. Moreover, given that $\bpi(\bt)=\bt+\mathbf{a}(\bt)$, taking $\bm{\theta}_1,\bm{\theta}_2\in\bm{\Theta}$ with
	${\bm{\pi}}(\bm{\theta}_1) = {\bm{\pi}}(\bm{\theta}_2)$, then $
	    \big\lVert {\mathbf{a}}(\bm{\theta}_1) - {\mathbf{a}}(\bm{\theta}_2) \big\rVert_2 = \big\lVert \bm{\theta}_1 - \bm{\theta}_2 \big\rVert_2$,
which is only possible if $\btheta_1 = \btheta_2$. Thus, $\bm{\pi}(\bm{\theta})$ is injective, which satisfies the identifiability condition.}. Similar to Assumption~\ref{assum:C} for the finite sample bias, we can also restrict Assumption \ref{assum:D} to the case when $\mathbf{a}(\bt)$ is linear, as stated in Assumption \ref{assum:D1} below.

\setcounter{Assumption}{3}
\renewcommand\theAssumption{\Alph{Assumption}.1}
\begin{Assumption}
\label{assum:D1}
    The asymptotic bias function $\mathbf{a}(\bt)$ can be expressed as $\mathbf{a}(\bt) = \mathbf{A}\bt + \mathbf{z}$, where $\mathbf{A} \in \real^{p \times p}$ and $\mathbf{z} \in \real^p$. Moreover, the matrix $\mathbf{A}$ is such that $0 < \|\mathbf{A}\|_2 < 1$, $(\mathbf{I}+\mathbf{A})$ is nonsingular and $\|(\mathbf{I}+\mathbf{A})^{-1}\|_\infty = \mathcal{O}(1)$. 
\end{Assumption} 
\vspace{0.4cm}

Like for the finite sample bias $\mathbf{b}(\bt,n)$ discussed in Section \ref{sec:bias:consist}, the asymptotic bias $\mathbf{a}(\bt)$ is also typically difficult to obtain, especially when considering complex models in possibly high dimensional settings. However, Assumption \ref{assum:D1} holds for some models (see e.g. \citealp{sur2019modern}). As previously mentioned, our assumptions can be thought of as local approximations and may be numerically assessed via simulations as discussed in Supplementary Material~\ref{supp:bias}.

We provide, in Theorem \ref{thm:bias:asymp} below, the bias orders of the JINI estimator under different bias conditions, the proof of which is provided in Appendix~\ref{app:proof:JINI:bias:asymp} together with a discussion on Assumptions \ref{assum:D} and \ref{assum:D1}. 

\begin{Theorem}
\label{thm:bias:asymp}
Under Assumptions \ref{assum:A}, \ref{assum:B}, and
\begin{enumerate}
    \item Assumptions \ref{assum:C1} and \ref{assum:D1}, the JINI estimator $\widehat{\bt}$ satisfies $\left\|\mathbb{E}[\widehat{\bt}] - \bt_0\right\|_2 = 0$, for sufficiently large (but finite) $n$. 
    
    \item Assumptions \ref{assum:C} and \ref{assum:D}, the JINI estimator $\widehat{\bt}$ satisfies $\left\|\mathbb{E}[\widehat{\bt}] - \bt_0\right\|_2 = \mathcal{O}(p^{1/2}n^{-2\alpha}).$
\end{enumerate}
\end{Theorem}

Theorem \ref{thm:bias:asymp} states that the JINI estimator can achieve asymptotically optimal bias correction when both the finite sample and asymptotic biases are linear. Moreover, the comparison of Theorems \ref{thm:bias:finite} and \ref{thm:bias:asymp} indicates that the bias correction properties of the JINI estimator may deteriorate when considering an inconsistent initial estimator. Nevertheless, the JINI estimator (possibly based on an inconsistent initial estimator) provides superior guarantees than the MLE in terms of bias correction. Indeed, in low dimensional settings, the bias order of the MLE is typically $\mathcal{O}(n^{-1})$ (see e.g. \citealp{cox1968general} and \citealp{kosmidis2014bias}). Our results show that the JINI estimator (when the asymptotic bias of the initial estimator is not nil) possesses similar bias properties as the MLE.

Moreover, the JINI is a consistent estimator as discussed in Section \ref{sec:consist:asympnorm}. One of the notable advantages of considering inconsistent initial estimators is that they can be chosen for computational efficiency, as for example using approximations of integrals or ignoring (small) departures from the models such as the ones induced by censoring or misclassification. This is of particular interest in high dimensional settings as well as for complex models. The advantageous properties of the JINI estimator are further illustrated in Section~\ref{sec:app:inconsist} and Supplementary Material~\ref{supp:cens-pois}.

\section{Consistency and Asymptotic Normality}
\label{sec:consist:asympnorm}

In this section we provide the standard statistical properties of the JINI estimator, namely consistency and asymptotic normality. While these properties have been studied in low dimensional settings (see e.g. \citealp{gourieroux1993indirect}), these results may not be easily applicable in high dimensional settings. Under Assumptions \ref{assum:A} to \ref{assum:D}, the JINI estimator is shown to be consistent in Theorem \ref{thm:consist:asympnorm} using standard asymptotic techniques (see e.g. \citealp{mcfadden1989method}). However, the asymptotic normality result is more challenging. In order to derive this result, we consider Assumption \ref{assum:E} below.
\setcounter{Assumption}{4}
\renewcommand\theAssumption{\Alph{Assumption}}
\begin{Assumption}
\label{assum:E}
    For any $\mathbf{s} \in \real^p$ such that $\|\mathbf{s}\|_2=1$, we have 
    \begin{equation*}
        \sqrt{n} \mathbf{s}\tt \mathbf{\Sigma}(\bt_0)^{-1/2}\left\{\widehat{\bpi}(\bt_0, n) - \bpi(\bt_0) \right\} \overset{d}{\to} \mathcal{N}(0,1),
    \end{equation*}
    where $\mathbf{\Sigma}(\bt)$ is nonsingular and continuous in $\bt \in \bT$. Moreover, for all $i=1,\ldots,p$, $\partial\mathbf{\Sigma}(\bt_0) /\partial \theta_i $ exists and let 
    \begin{equation*}
        w_i(\bt_0) \vcentcolon = \lambda_{\max}\left[\{\mathbf{I}+\mathbf{A}(\bt_0)\}^{-1} \frac{\partial \mathbf{\Sigma}(\bt_0)}{\partial \theta_i} \{\mathbf{I}+\mathbf{A}(\bt_0)\}^{-\rm T}\right],
    \end{equation*}
    where $\mathbf{A}^{-\rm T} \vcentcolon = \left(\mathbf{A}^{-1}\right)^{\rm T}$, $\lambda_{\max}(\mathbf{U})$ denotes the largest eigenvalue of the matrix $\mathbf{U}$ and $\mathbf{A}(\bt)$ is stated in Assumption \ref{assum:D}, then we have $\|\mathbf{w}(\bt_0)\|_2 = \mathcal{O}(1)$ where $\mathbf{w}(\bt_0) \vcentcolon= \left\{w_1(\bt_0) \ldots w_p(\bt_0)\right\}\tt$.
\end{Assumption} 
\vspace{0.4cm}

When $\widehat{\bpi}(\bt_0, n)$ is consistent, we have $\bpi(\bt_0) = \bt_0$ and consequently, the first condition of Assumption \ref{assum:E} is equivalent to the asymptotic normality of the initial estimator. When $\widehat{\bpi}(\bt_0, n)$ is inconsistent, as we have $\bpi(\bt_0) = \bt_0 + \mathbf{a}(\bt_0)$, it implies the asymptotic normality of the initial estimator with respect to a ``shifted'' target $\bt_0 + \mathbf{a}(\bt_0)$. Moreover, for the second condition of Assumption \ref{assum:E}, when $\widehat{\bpi}(\bt_0, n)$ is consistent, $w_i(\bt_0)$ is simplified to be $w_i(\bt_0) = \lambda_{\max}\left\{\partial \mathbf{\Sigma}(\bt_0) / \partial \theta_i \right\}$. In this case, the quantity $\left\|\mathbf{w}(\bt_0)\right\|_2$ essentially measures the sensitivity of the asymptotic variance $\bf{\Sigma}(\bt_0)$ with respect to the parameters, and the second condition of Assumption \ref{assum:E} indicates that this measure should be bounded, which is the case, for example, for (generalized) linear models. Nevertheless, this requirement may not be easy to verify in general. 
We provide in Theorem \ref{thm:consist:asympnorm} below the results on the consistency and asymptotic normality of the JINI estimator, the proof of which is given in Appendix~\ref{app:proof:consist:asympnorm}.

\begin{Theorem}
\label{thm:consist:asympnorm}
    Under Assumptions \ref{assum:A}, \ref{assum:B}, \ref{assum:C} and \ref{assum:D}, the JINI estimator $\widehat{\bt}$ is such that $\left\|\widehat{\bt} - \bt_0\right\|_2 = o_{\rm p}(1).$
    Moreover, with the addition of Assumption \ref{assum:E}, the JINI estimator $\widehat{\bt}$ satisfies
    \begin{equation*}
        \left(1+\frac{1}{H}\right)^{-1/2} \sqrt{n} \mathbf{s}\tt \left[\{\mathbf{I}+\mathbf{A}(\bt_0)\}^{-1} \mathbf{\Sigma}(\bt_0) \{\mathbf{I}+\mathbf{A}(\bt_0)\}^{{-\rm T}}\right]^{-1/2} (\widehat{\bt} - \bt_0) \overset{d}{\to} \mathcal{N}(0,1).
    \end{equation*}
\end{Theorem}

This result shows that the JINI estimator is consistent under plausible conditions and asymptotically normally distributed under additional requirements. Moreover, the asymptotic variance of the estimator differs from the one of the initial estimator $\widehat{\bpi}(\bt_0, n)$, especially when $\widehat{\bpi}(\bt_0, n)$ is inconsistent. Indeed, when $\widehat{\bpi}(\bt_0, n)$ is consistent, we have (in low dimensional settings)
\begin{equation}
\label{eqn:asymp:norm:consistent}
    \sqrt{n} (\widehat{\bt} - \bt_0) \overset{d}{\to}  \mathcal{N} \big(\0, \left(1+1/H\right) \mathbf{\Sigma}(\bt_0)\big), 
\end{equation}
where $\mathbf{\Sigma}(\bt_0)$ corresponds to the asymptotic variance of $\widehat{\bpi}(\bt_0, n)$. In this case, the asymptotic variance of $\widehat{\bt}$ is slightly inflated by the term $1+1/H$ compared to $\mathbf{\Sigma}(\bt_0)$. Therefore, this bias correction approach comes at the price of a typically marginally increased asymptotic variance. However, in practice, the JINI estimator can not only correct the bias of its initial estimator but also possibly reduce its finite sample variance as illustrated, for example, in Section~\ref{sec:app:consist} (see also Supplementary Material~\ref{supp:logistic}). When $\widehat{\bpi}(\bt_0, n)$ is an inconsistent initial estimator, this procedure reduces the bias of $\widehat{\bpi}(\bt_0, n)$ but its (asymptotic) variance may considerably differ from $\bm{\Sigma}(\bt_0)$. Indeed, the result in Theorem \ref{thm:consist:asympnorm} can be rewritten (in low dimensional settings) as
\begin{equation}
\label{eqn:asymp:norm:inconsistent}
    \sqrt{n}(\widehat{\bt} - \bt_0) \overset{d}{\to} \mathcal{N}\left(\0, \left(1+1/H\right)\left\{\mathbf{I}+\mathbf{A}(\bt_0)\right\}^{-1} \mathbf{\Sigma}(\bt_0) \{\mathbf{I}+\mathbf{A}(\bt_0)\}^{-\rm T} \right),
\end{equation}
which is in line with the indirect inference results presented, for example, in \cite{gourieroux1993indirect}. The results in \eqref{eqn:asymp:norm:consistent} and \eqref{eqn:asymp:norm:inconsistent} suggest that, depending on the asymptotic bias function $\mathbf{a}(\bt)$, the asymptotic variance of $\widehat{\bt}$ is not necessarily larger when considering an inconsistent initial estimator, compared to the result obtained with a consistent initial estimator. Overall, both the theoretical and practical results highlight the fact that the JINI estimator can correct finite sample and/or asymptotic biases without necessarily significantly inflating the variance. 


\section{Computational Aspects}
\label{sec:computation}

As mentioned before, the JINI estimator can be computed efficiently using the IB algorithm provided in \eqref{eqn:IB}. In this section, we provide the convergence properties of the IB sequence in Theorem~\ref{thm:ib} with the proof in Appendix~\ref{app:proof:ib}. Detailed implementation instructions for the IB algorithm are provided in Table \ref{tab:ib:algo} in Appendix \ref{app:imple:ib}.

\begin{Theorem}
\label{thm:ib}
    Under Assumptions \ref{assum:A}, \ref{assum:B}, \ref{assum:C} and \ref{assum:D}, the IB sequence in \eqref{eqn:IB} satisfies
    \begin{align*}
        \left\|\bt^{(k)} - \widehat{\bt}\right\|_2 = \mathcal{O}_{\rm p}\left(p^{1/2} \exp(-ck)\right), 
    \end{align*}
    where $c$ is some positive real number.
\end{Theorem}
\vspace{0.4cm}

Theorem \ref{thm:ib} shows that the IB algorithm converges to the JINI estimator (in norm) at an exponential rate. However, this convergence can be slower when $p$ is large. In practice, the number of iterations needed appears to be relatively small. For example, we present in Figure \ref{fig:ib-num-iter} in Supplementary Material~\ref{supp:ib-num-iter} the number of iterations needed for the IB algorithm to compute the JINI estimators considered in Sections~\ref{sec:app:consist}~and~\ref{sec:app:inconsist} (as well as in Supplementary Material~\ref{supp:logistic}), and we find that the computations typically converge in less than 15 iterations. 

Interestingly, from Theorem~\ref{thm:ib}, we can easily define a sequence $\{k_n\}\in\mathbb{N}^+$ such that $k_n = \mathcal{O}\big(\log(pn)\big)$ and $\left\|\bt^{(k_n)} - \widehat{\bt}\right\|_2 = o_{\rm p}(1)$. Additionally with the consistency result of $\widehat{\bt}$ as presented in Theorem~\ref{thm:consist:asympnorm}, we have 
\begin{equation*}
    \left\|\bt^{(k_n)} - \bt_0\right\|_2 \leq \left\|\bt^{(k_n)} - \widehat{\bt}\right\|_2 + \left\|\widehat{\bt} - \bt_0\right\|_2 = o_{\rm p}(1),
\end{equation*}
which suggests that $\bt^{(k_n)}$, without reaching the limit $\widehat{\bt}$, is nevertheless consistent.

Therefore, the IB algorithm provides a  computationally efficient approach to obtain the JINI estimator, which could otherwise be difficult to compute. This algorithm alleviates considerably the limitations of some simulation-based estimators, especially in high dimensional settings.

\section{Applications with Consistent Initial Estimators: An Example with the Negative Binomial Regression}
\label{sec:app:consist}

In this section, we apply the methodology developed in Sections \ref{sec:existing:methods} and \ref{sec:bias} to estimate the parameters of the negative binomial regression. We additionally consider the logistic regression in Supplementary Material~\ref{supp:logistic} and the logistic regression with random intercept in Supplementary Material~\ref{supp:glmm}. For each model, we analyze a simulation-based approximation of the bias function of the MLE. We then  construct both the BBC and the JINI estimators based on the MLE, and compare their finite sample properties, particularly the bias and RMSE, in high dimensional settings where $p$ is relatively large compared to $n$. For all models, we use the IB algorithm to compute the JINI estimator, as it is easy and straightforward to implement (see Table \ref{tab:ib:algo} in Appendix~\ref{app:imple:ib}) and does not require model specific analytical derivations. Overall, our results show that the JINI estimator allows to significantly reduce (or even apparently eliminate) the finite sample bias of the initial estimator, with comparable (or even lower) RMSE to alternative bias corrected estimators. In particular, for the logistic regression, the bias function of the MLE appears to be (approximately locally) linear (see Figure~\ref{fig:logistic-bias} in Supplementary Material~\ref{supp:logistic}), which is in line with the result of \cite{sur2019modern}. In this application, we additionally consider a bias reduced estimator proposed by \cite{KoFi:09} which is based on an analytical modification of the score function. We observe that this estimator and the JINI estimator both achieve apparent unbiasedness with comparable RMSE, which, for the JINI estimator, is suggested by the first result of Theorem~\ref{thm:bias:finite} in Section~\ref{sec:bias}. Therefore, if computational time is not an important limitation, the JINI estimator is an interesting option when considering consistent estimators that may suffer from finite sample biases.

We now discuss in more details the negative binomial regression, which is typically used to analyze count response data which exhibit overdispersion, conditional on a set of covariates. Specifically, it links the (independent) responses $y_i \in \{0, 1, \ldots, m\}$, $m \in \mathbb{N}^+$ and $i=1,\ldots,n$, with the linear predictor $\mathbf{x}_i\tt\boldsymbol{\beta}$ where $\mathbf{x}_i$ is a vector of fixed covariates, through the exponential link $\mu_i \vcentcolon = \mathbb{E}[Y_i\vert \x_i] = \exp(\mathbf{x}_i\tt \bm{\beta})$. The conditional density, with overdispersion parameter $\alpha > 0$, can be written as 
\begin{equation*}
    f(y_i | \mathbf{x}_i ; \bm{\beta}, \alpha) = \frac{\Gamma(y_i + \alpha^{-1})}{\Gamma(y_i+1)\Gamma(\alpha^{-1})} \left(\frac{\alpha^{-1}}{\alpha^{-1}+\mu_i}\right)^{\alpha^{-1}} \left(\frac{\mu_i}{\alpha^{-1}+\mu_i}\right)^{y_i},
\end{equation*}
where $\Gamma(\cdot)$ is the gamma function. The (conditional) variance function is defined as $\var(Y_i\vert\x_i) \vcentcolon = \mu_i + \alpha \mu_i^2$, and therefore, increasing the value of $\alpha$ corresponds to increasing the overdispersion level. The limiting case when $\alpha$ goes to $0$ corresponds to the Poisson distribution. 

In general, it is known that the MLE, especially for the overdispersion parameter, can be seriously biased in finite sample (see e.g. \citealp{SaPa:05,heller2019beyond} and the reference therein). Since the overdispersion parameter is widely used as an important measure in medical, biological and environmental studies, we consider the BBC and the JINI estimators based on the MLE for bias correction, and compare their finite sample performance to the MLE. The MLE is computed using the \texttt{glm.nb} function (with default parameters) in the \texttt{MASS} R package. The covariates are simulated in the following way \citep[see e.g.][]{AeCaHe:14}: $x_{i1}$ is the intercept, $x_{i2}$ is simulated from a $\mathcal{N}(0, 1)$, $x_{i3}$ is a dummy variable with the first half of the values being zeros, and $x_{i4}, \ldots, x_{ip}$ are independently simulated from a $\mathcal{N}(0, 4^2/n)$. The simulation settings are presented in Table~\ref{tab:sim-negbin}.

Figure \ref{fig:sim-neg_bin-smr} presents the (absolute) bias and RMSE of the estimators. The boxplots of the finite sample distributions and the simulation-based approximation of the bias function of the initial estimator are presented in Figures \ref{fig:sim-neg_bin-bxp} and  \ref{fig:nb-bias1} respectively in Supplementary Materials~\ref{supp:boxplots} and \ref{supp:bias}. The bias function of the MLE appears to be locally linear, especially for the overdispersion parameter~$\alpha$, which explains the apparent unbiasedness and relatively low RMSE for all parameters of the JINI estimator in Figure~\ref{fig:sim-neg_bin-smr}. On the other hand, as expected, the MLE is significantly biased, especially on the overdispersion parameter. The BBC estimator also has a reduced bias, but the JINI estimator is (apparently) unbiased, hence better than the BBC estimator, without loss in terms of RMSE. 
\begin{table}[!tb]
     \centering
     \caption{Simulation setting for the negative binomial regression}
     \begin{tabular}{lr}
 \toprule
 Parameters &  Values\\
 \midrule
 $n=$ & $100$ \\
 $\beta_1=$ & $1.5$ \\
 $\beta_2=$ & $2.5$ \\
 $\beta_3=$ & $-2.5$ \\
 $\beta_4=\ldots=\beta_{20}=$ & $0$ \\
 $\alpha=$ & $0.6$ \\
 $H=$ & $200$ \\
 Number of simulations $=$ & $1000$ \\
 \bottomrule 
     \end{tabular}
     \label{tab:sim-negbin}
\end{table}
\begin{figure}[!tb]
     \centering
     \includegraphics[width=12cm]{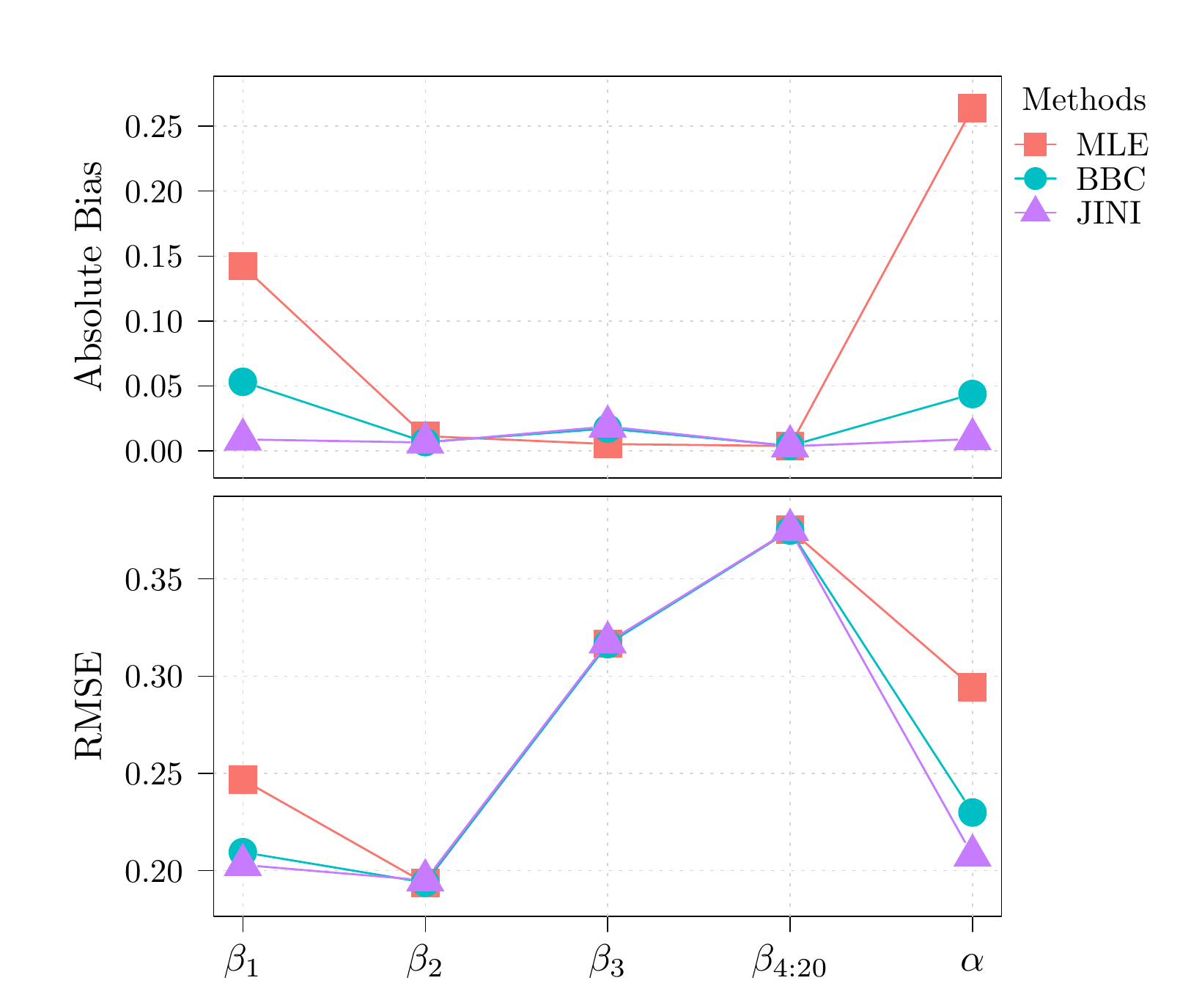}
     \caption{Finite sample absolute bias and RMSE of estimators for the negative binomial regression using the simulation setting presented in Table \ref{tab:sim-negbin}. The estimators are the MLE, the BBC and the JINI estimators based on the MLE. The value $\beta_{4:20}$ corresponds to the estimates of $\beta_4,\ldots,\beta_{20}$, whose true values are all zero.}
     \label{fig:sim-neg_bin-smr}
\end{figure}

\section{Applications with Inconsistent Initial Estimators}
\label{sec:app:inconsist}

Constructing consistent estimators can be challenging for various reasons. For example, data features such as censoring or misclassification errors can lead to analytical and computational challenges when implementing standard estimators such as the MLE. These difficulties are typically magnified when considering complex models and/or high dimensional settings. We propose to use the JINI estimator to alleviate these limitations by constructing consistent and bias corrected estimators based on an available initial estimator, such as the MLE that ignores these data features. The theoretical results presented in Sections~\ref{sec:bias} and \ref{sec:consist:asympnorm} suggest that such approach is applicable in these situations (as well as in other ones) where the asymptotic bias of the initial estimator is ``small'' (as discussed in Section~\ref{sec:bias:inconsist}). 

To illustrate the simplicity and generality of this approach, we consider different applications, including the censored negative binomial regression (Section~\ref{sec:NB-censoring}) and the logistic regression with misclassfied responses (Section~\ref{sec:logistic-misscla}). Additionally, we consider the censored Poisson regression in Supplementary Material~\ref{supp:cens-pois}. For each model, we construct the JINI estimator based on the MLE that ignores these data features, which is inconsistent but readily available. Our results suggest that the proposed JINI estimator outperforms the consistent MLE (when available) in terms of both bias and RMSE (see Section~\ref{sec:NB-censoring} and Supplementary Material~\ref{supp:cens-pois}). When no alternative consistent estimators are available, we observe that the proposed JINI estimator appears to enjoy from suitable finite sample properties (see Section~\ref{sec:logistic-misscla}). These results suggest that the JINI estimator allows to consider a broad set of initial estimators, possibly chosen for their computational advantages, to construct consistent and bias corrected estimators. Moreover, it is based on a simple approach that can be applied to cases where no alternative estimators are available due to the numerical and/or analytical challenges that they typically entail.

\subsection{Censored Negative Binomial Regression} 
\label{sec:NB-censoring}

An important extension of the negative binomial regression (see Section \ref{sec:app:consist}) is the censored negative binomial regression that accounts for censored responses.  It was first proposed in \cite{caudill1995modeling} and an alternative model in terms of discrete survival models was proposed in \cite{Hilb:11}. In this section, we take the case of right censoring above a known threshold $C$, and propose to construct the JINI estimator based on the classical MLE for negative binomial regression that ignores the censoring mechanism, which is inconsistent. Then, we compare the finite sample performance to the MLE for censored responses that is 
computed using the \texttt{gamlss} function in the \texttt{gamlss.cens} R package (see Chapter 12 of \citealp{Hilb:11} for more details). In this simulation study, we generate the covariates in the same manner as in Section \ref{sec:app:consist}, and the simulation settings are presented in Table~\ref{tab:sim-NB-censoring}.

\begin{table}[!tb]
     \centering
     \caption{Simulation setting for the censored negative binomial regression}
     \begin{tabular}{lr}
 \toprule
 Parameters &  Values\\
 \midrule
 $n=$ & $100$ \\
 $\beta_1=$ & $1.5$ \\
 $\beta_2=$ & $2.5$ \\
 $\beta_3=$ & $-2.5$ \\
 $\beta_4=\ldots=\beta_{20}=$ & $0$ \\
 $\alpha=$ & $0.6$ \\
 $C=$ & $30$ \\
 Censoring percentage $\approx$ & 11\%\\
 $H=$ & $200$ \\
 Number of simulations $=$ & $1000$ \\
 \bottomrule 
     \end{tabular}
     \label{tab:sim-NB-censoring}
\end{table}

We present in Figure~\ref{fig:NB-censoring-smr} the (absolute) bias and RMSE of the JINI estimator as well as the MLE for censored responses. We also provide the boxplots of their finite sample distributions and the simulation-based approximation of the bias function of the initial estimator (i.e. the classical MLE that ignores the censoring mechanism) in Figures \ref{fig:NB-censoring-bxp} and \ref{fig:nb-bias2} respectively in Supplementary Material~\ref{supp:boxplots} and \ref{supp:bias}. We observe from Figure~\ref{fig:nb-bias2} that the bias function of the initial estimator appears generally smooth  and locally linear, especially for the overdispersion parameter~$\alpha$, which explains the apparent unbiasedness of the JINI estimator for all parameters in Figure~\ref{fig:NB-censoring-smr}. On the other hand, the MLE shows significant bias (except for the parameters with true value of zero), especially for the overdispersion parameter, due to the relatively large $p/n$ ratio. In addition, the resulting RMSE of the JINI estimator appears to be significantly lower compared to the one of the MLE for all parameters. Finally, the JINI estimator avoids the numerical and/or analytical challenges typically entailed in the computation of the MLE and hence is simple to implement and readily available.

\begin{figure}[!tb]
    \centering
     \includegraphics[width=12cm]{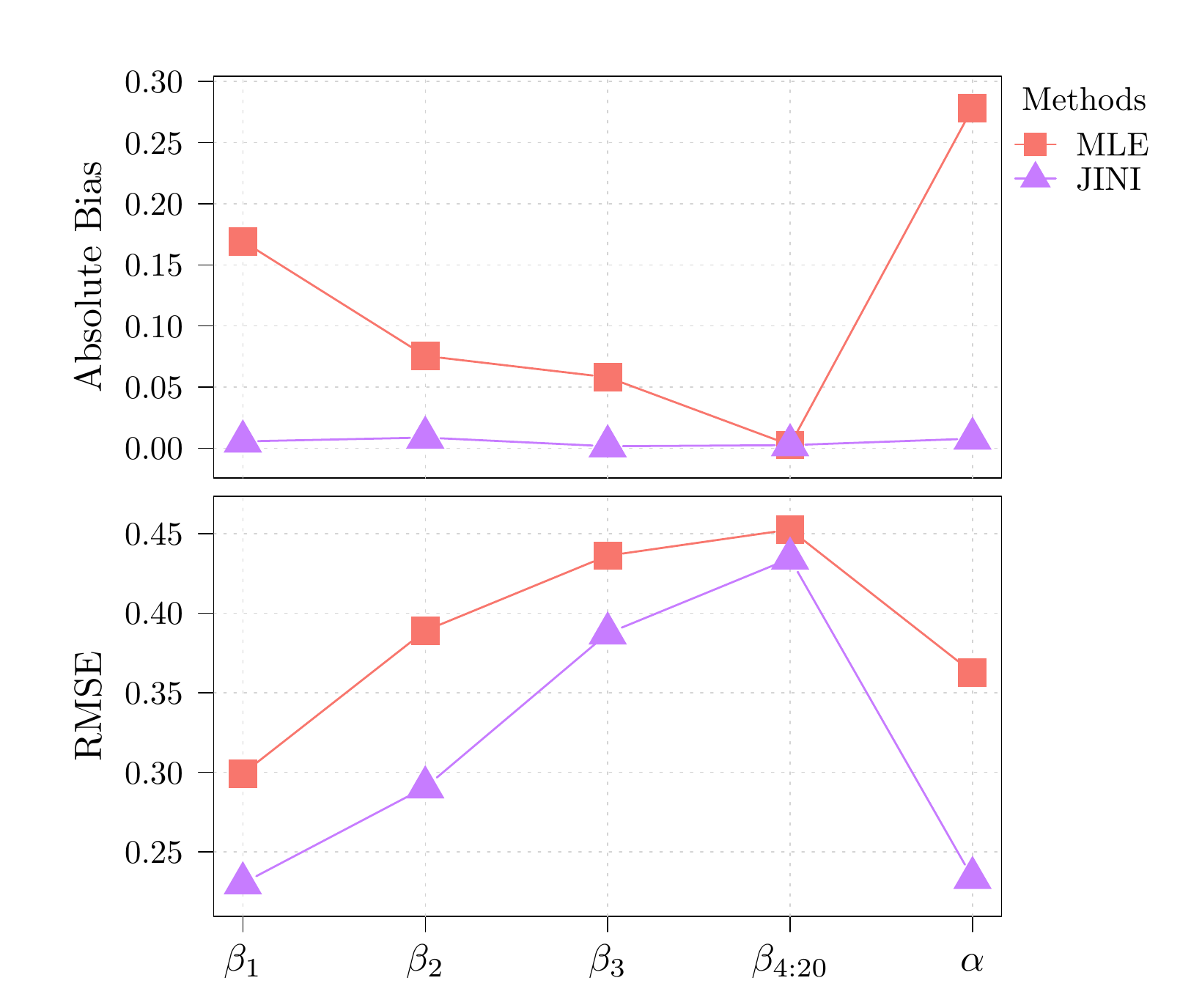}
     \caption{Finite sample absolute bias and RMSE of estimators for the censored negative binomial regression using the simulation settings presented in Table \ref{tab:sim-NB-censoring}. The estimators are the MLE for the censored negative binomial regression model, and the JINI estimator with, as initial estimator, the classical MLE for the negative binomial regression model that ignores the censoring. The estimators are grouped according to their parameter values.}
     \label{fig:NB-censoring-smr}
\end{figure}

\subsection{Logistic Regression with Misclassified Responses}
\label{sec:logistic-misscla}

Misclassification of binary response variables are frequently encountered, for example, in epidemiology studies \citep[see e.g.][and the references therein]{NiDaKaTaLe:19}, and in economics \citep[see e.g.][and the references therein]{MEYER2017}. Specifically, in these settings with a sample of $n$ independent subjects, the response is misclassified in the sense that the actual response variable $Y_i$ of the $i^{th}$ subject is not observed, but instead, we observe another binary variable $Z_i$ which is such that
\begin{equation*}
    P(Z_i = 0 | Y_i = 0) = 1-\tau_1 \quad \text{and} \quad P(Z_i = 1 | Y_i = 1) = 1-\tau_2.
\end{equation*}
The probabilities $\tau_1, \tau_2 \in [0,1]$ denote, respectively, the False Positive (FP) rate and the False Negative (FN) rate. Such misclassification introduces necessary modification on a classical logistic regression model (where misclassification is not considered). Indeed,  the probability of success for the $i^{th}$ response is given by 
\begin{equation}
\begin{aligned}
\label{eqn:logistic-misscla-fixed}
    \mu_i^*(\bm{\beta}) \vcentcolon &= \mathbb{E}[Z_i] = P(Z_i = 1 | Y_i = 0) P(Y_i = 0) + P(Z_i = 1 | Y_i = 1) P(Y_i = 1) \\
    &= \tau_1 \left\{1-\mu_i(\bm{\beta})\right\} + (1-\tau_2) \mu_i(\bm{\beta}),
\end{aligned}
\end{equation}
where ${\mu}_i(\bm{\beta})\vcentcolon = \mathbb{E}[Y_i\vert \x_i]  =\exp(\x_i\tt\bm{\beta})/\left\{1+\exp(\x_i\tt\boldsymbol{\beta})\right\}$, with $\x_i\in\real^p$ as a vector of fixed covariates for the $i^{th}$ subject and $\bm{\beta}$ as the regression coefficients. When $\tau_1 = \tau_2 = 0$, we have $\mu_i^*(\bm{\beta}) = \mu_i(\bm{\beta})$, that is, the classical logistic regression model is recovered. Often $\tau_1$ and $\tau_2$ are known or assumed to be known, for example, when the response is the measure of some medical testing device with known specificity $(1-\tau_1)$ and sensitivity $(1-\tau_2)$ \citep[see e.g.][]{Neuh:99,CarrollBook:2006}. \cite{Bouman2020} also considers FP and FN rates distributions derived from the Receiver Operating Characteristic (ROC) curve \citep{Swets:88} of the medical testing devices. Overall, these (empirical) distributions for quantitative serological test measurements are usually available. Nevertheless, the MLE based on the modified likelihood function is typically challenging to derive analytically and/or to compute numerically, especially in high dimensional settings.

In this section, we extend the results of \cite{ZwCa:93} for the logistic regression with misclassified responses, for which, to the best of our knowledge, no consistent estimators are available. We propose to construct the JINI estimator using the MLE of the classical logistic regression as the initial estimator, which is inconsistent in this case. More precisely, we extend the probability of success given in \eqref{eqn:logistic-misscla-fixed} to 
\begin{equation}
    \mu_i^*(\bm{\beta}) \vcentcolon = \mathbb{E}[Z_i | U_{\tau_1, i}, U_{\tau_2, i}] = U_{\tau_1, i}\left\{1-\mu_i(\bm{\beta})\right\} + (1-U_{\tau_2, i}) \mu_i(\bm{\beta}),
    \label{eqn:logistic-missclass}
\end{equation}
where $U_{\tau_1, i}$ and $U_{\tau_2, i}$ are two latent variables generated around the ROC curve. To compute the JINI estimator,  $U_{\tau_1, i}$ and $U_{\tau_2, i}$ are simulated only once and then held fixed throughout the steps of the IB algorithm (see Table~\ref{tab:ib:algo} in Appendix \ref{app:imple:ib}). We consider a ROC curve with an area under the curve of $99\%$, a FP rate of $\tau_1 \approx 2\%$ and a sensitivity of $1-\tau_2 \approx 90\%$. We then generate $U_{\tau_1, i}$ and $U_{\tau_2, i}$ centered at $(\tau_1,1-\tau_2)$ around the ROC curve, which is presented in Figure \ref{fig:roc} in Supplementary Material~\ref{supp:ROC}. 

\begin{table}[!tb]
     \centering
     \caption{Simulation settings for the logistic regression with misclassified responses}
     \begin{tabular}{lrr}
 \toprule
 Parameters &  Setting I & Setting II \\
 \midrule
 $n=$ & $2000$ & $3000$ \\
 $\sum_{i=1}^n y_i\approx$ & $1000$ & $750$ \\
  $\sum_{i=1}^n z_i\approx$ & $800$ &  $350$\\
 EPV $\approx$ & $4$ & $1.75$ \\
 $\beta_1=\beta_2=$ & $5$ & $5$ \\
 $\beta_3=\beta_4=$ & $-7$ & $-7$ \\
 $\beta_5=\ldots=\beta_{200}=$ & $0$ & $0$ \\
 $U_{\tau_1, i} \overset{iid}{\sim}$ & $\text{Beta}(2, 50)$ & $\text{Beta}(2, 50)$ \\
 $U_{\tau_2, i} \overset{iid}{\sim}$ & $\text{Beta}(2, 10)$ & $\text{Beta}(2, 10)$ \\
 FP rate $\approx$ & 2\% & 2\% \\
 FN rate $\approx$ & 10\% & 10\% \\
 $H=$ & $200$ & $200$ \\
 Number of simulations $=$ & $1000$ & $1000$ \\
 \bottomrule
     \end{tabular}
     \label{tab:sim-logistic-misscla}
\end{table}

In addition, for the arguably simpler model of the classical logistic regression (with no misclassification errors), a rule of thumb is typically to consider at least $10$ Events Per Variable (EPV) in order to ensure reliable finite sample performance of the MLE for the classical logistic regression (see e.g. \citealp{PeCoKeHoFe:96}). However, with an EPV smaller than $10$, a situation that is quite frequently encountered in practice, the MLE becomes significantly biased and the resulting inference becomes unreliable. More detailed discussions can be found in Supplementary Material~\ref{supp:logistic}. 

In this simulation study, we consider two high dimensional settings with balanced (Setting I) and unbalanced (Setting II) responses, both of which can possibly occur in practice. We consider large models with $p$ relatively large compared to $n$, and provide an EPV of approximately $4$ and $1.75$ respectively. We also generate the covariates independently from a $\mathcal{N}(0,4^2/n)$ for Setting I and from a $\mathcal{N}(0.6,4^2/n)$ for Setting II. The detailed simulation settings can be found in Table~\ref{tab:sim-logistic-misscla}. These settings would already be challenging, in terms of estimation accuracy, for the classical logistic regression without misclassification errors (see Supplementary Material~\ref{supp:logistic}), given the high dimensions and the relatively low EPV of the settings. We expect it to be more challenging when considering misclassification responses.

Since there is no alternative (consistent) available estimator, we present the finite sample distribution (boxplots) of  the JINI estimator in Figure~\ref{fig:logistic-misscla-bxp}, as well as the simulation-based approximation of the bias function of the initial estimator (i.e. the classical MLE for the logistic regression that ignores the misclassification errors) in Figure~\ref{fig:sero-bias} in Supplementary Material~\ref{supp:bias}. We can see from Figure~\ref{fig:sero-bias} that the bias function of the initial estimator appears linear in Setting I and locally linear in Setting~{II}. This explains the observations in Figure~\ref{fig:logistic-misscla-bxp} that the JINI estimator shows small (or apparently negligible) bias in both settings as suggested by the first result of Theorem~\ref{thm:bias:asymp}. Although there is no benchmark in this case, we can, for example, compare the finite sample performance of the JINI estimator with the consistent MLE (assuming no misclassification) in the same settings but without misclassified responses as done in Supplementary Material~\ref{supp:logistic}. By comparing Figures~\ref{fig:logistic-misscla-bxp} and \ref{fig:sim-logistic-bxp} (in Supplementary Material~\ref{supp:logistic}), it can be observed that the JINI estimator (in Figure~\ref{fig:logistic-misscla-bxp}) has comparable finite sample performance as the MLE (in Figure~\ref{fig:sim-logistic-bxp}). Therefore, it is reasonable to conclude that the JINI estimator would have, at least, similar performance to the MLE for the model with misclasification errors, while avoiding its computational challenges. These results highlight the advantage that the JINI estimator is an easily constructed consistent and bias corrected estimator which does not require analytical derivations and complex numerical implementations.  
\begin{figure}[!tb]
    \centering
     \includegraphics[width=12cm]{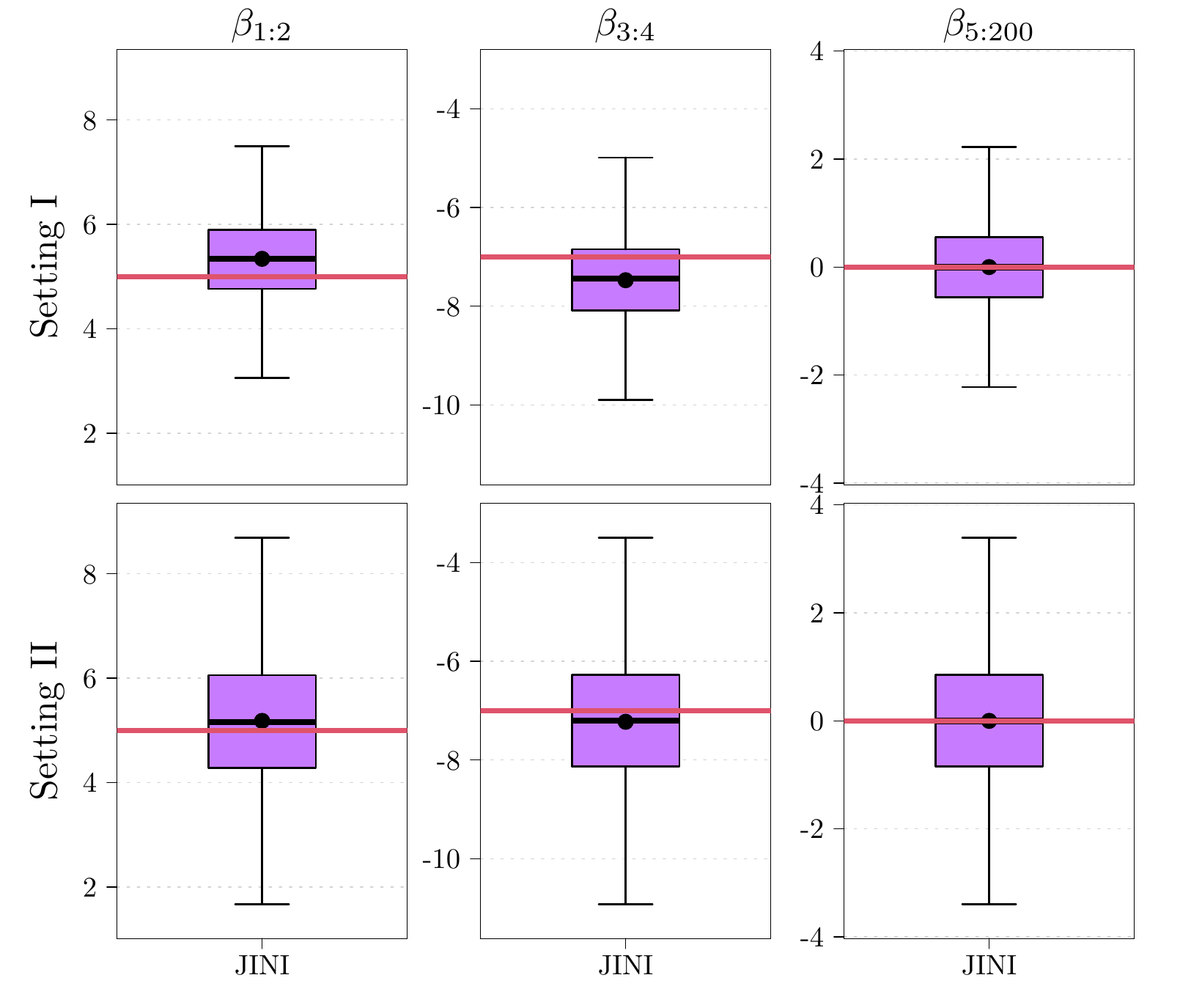}
     \caption{Finite sample distributions of the JINI estimator for the logistic regression with misclassified responses using the simulation settings presented in Table \ref{tab:sim-logistic-misscla}. The JINI estimator is computed based on the classical MLE for the  logistic regression that ignores the misclassification errors. The lines indicate the true parameter values and the dots locate the mean simulation values. The estimators are grouped according to their parameter values.} 
     \label{fig:logistic-misscla-bxp}
\end{figure}

\clearpage
\bibliographystyle{chicago}
\newpage

\bibliography{biblio.bib}


\newpage
\appendix
\centerline{\Large\sc Appendices}

\titleformat{\section}[hang]{\large\scshape}{\Alph{section}}{1em}{}

\section{Proof of Theorem \ref{thm:bias:finite}}
\label{app:proof:JINI:bias:finite}

In order to prove Theorem \ref{thm:bias:finite}, we introduce the following lemma.

\begin{Lemma}
\label{lem:1}
Let $\|\cdot\|$ be any sub-multiplicative matrix norm on $\real^{p \times p}$.
\vspace{-0.5cm}
\begin{enumerate}
    \item If $\mathbf{B}(n) \in \real^{p \times p}$ is such that $\|\mathbf{B}(n)\| = \mathcal{O}(n^{-\beta})$ with $\beta>0$, then we have $\{\mathbf{I} - \mathbf{B}(n)\}^{-1} = \sum_{k=0}^\infty \B(n)^k$ and $\|\{\mathbf{I} - \mathbf{B}(n)\}^{-1} \| = \mathcal{O}(1)$ for sufficiently large $n$. 
    \item If $\mathbf{A} \in \real^{p \times p}$ is such that $0 < \|\mathbf{A}\| < 1$, then $\left(\mathbf{I} - \mathbf{A}\right)^{-1} = \sum_{k=0}^\infty \mathbf{A}^k$ and
$\|(\mathbf{I} - \mathbf{A})^{-1} \| = \mathcal{O}(1)$. 
\end{enumerate} 
\end{Lemma}

\begin{proof}
We start by proving the first part of the lemma. Since $\|\mathbf{B}(n)\| = \mathcal{O}(n^{-\beta})$, 
there exist $N>0$ and $C>0$ such that $\|\mathbf{B}(n)\| \leq Cn^{-\beta}$
for all $n \geq N$. Without loss of generality, we consider $CN^{-\beta}<1$. Therefore, $\| \mathbf{B}(n)^k \| \leq \| \mathbf{B}(n)\| ^k\leq(Cn^{-\beta})^k\to 0$ when $k\to\infty$ for any $n\geq N$. Hence for any $n\geq N$, we have
$\mathbf{B}(n)^k\to \0$ when $k\to\infty$ and 
\begin{equation}
\label{eq:appA:5}
   \left\|\sum_{k=0}^\infty\mathbf{B}(n)^k\right\| \le \sum_{k=0}^\infty\|\mathbf{B}(n)\|^k \le \sum_{k=0}^\infty(Cn^{-\beta})^k=(1-Cn^{-\beta})^{-1}=\mathcal{O}(1).
\end{equation}
Now we want to prove that $\mathbf{I}-\mathbf{B}(n)$ is nonsingular for sufficiently large $n$. Indeed, note that for all $m \geq 0$, we have 
\begin{equation*}
   \{\mathbf{I}-\mathbf{B}(n)\} \left\{\sum_{k=0}^m\mathbf{B}(n)^k\right\}= \left\{\sum_{k=0}^m\mathbf{B}(n)^k\right\}\{\mathbf{I}-\mathbf{B}(n)\}= \mathbf{I} - \mathbf{B}(n)^{m+1},
\end{equation*}
and therefore, let $m \to \infty$ and for sufficiently large $n$, we have 
\begin{equation*}
   \{\mathbf{I}-\mathbf{B}(n)\} \left\{\sum_{k=0}^\infty\mathbf{B}(n)^k\right\}=  \left\{\sum_{k=0}^\infty\mathbf{B}(n)^k\right\}\{\mathbf{I}-\mathbf{B}(n)\}= \mathbf{I},
\end{equation*}
which defines that 
\begin{equation*}
    \{\mathbf{I} - \mathbf{B}(n)\}^{-1} = \sum_{k=0}^\infty \mathbf{B}(n)^k.
\end{equation*}
Finally, we have $\left\|\{\mathbf{I} - \mathbf{B}(n)\}^{-1}\right\| = 
   \Big\|\sum_{k=0}^\infty\mathbf{B}(n)^k\Big\| = \mathcal{O}(1)$ by \eqref{eq:appA:5}. The second part of the lemma can be derived using the similar arguments.
\end{proof}
\vspace{0.5cm}

We now start to prove Theorem \ref{thm:bias:finite}. When the initial estimator is consistent (i.e. $\mathbf{a}(\bt)=\0$), by definition in \eqref{eqn:def:JINI}, the JINI estimator satisfies 
\begin{equation}
\begin{aligned}
\label{eq:appA:1}
    \widehat{\bpi}(\bt_0, n) &= \bt_0 + \mathbf{b}(\bt_0,n) + \mathbf{v}(\bt_0, n) = \bpi^*(\widehat{\bt}, n) = \frac{1}{H}\sum_{h=1}^H \widehat{\bpi}_h(\widehat{\bt},n) \\
    &= \widehat{\bt} + \mathbf{b}(\widehat{\bt}, n) + \frac{1}{H}\sum_{h=1}^H \mathbf{v}_h(\widehat{\bt}, n),
\end{aligned}
\end{equation}
with $\mathbf{v}_h(\widehat{\bt}, n)$ corresponding to the noise of the $h^{th}$ simulated sample. By rearranging the terms and taking expectations on both sides of \eqref{eq:appA:1}, we obtain
\begin{equation}
\label{eq:appA:2}
    \mathbb{E}[\widehat{\bt}] - \bt_0 = \mathbf{b}(\bt_0, n) - \mathbb{E} \left[\mathbf{b}(\widehat{\bt}, n)\right].
\end{equation}
In the following, we prove Theorem \ref{thm:bias:finite} from the fourth result to the first result. 

\subsection{Proof of the Fourth Result}
\label{app:proof:JINI:bias:finite:4}

By the mean value theorem and under Assumption \ref{assum:C}, we have, from \eqref{eq:appA:2}, that
%
%
\begin{equation*}
\begin{aligned}
    \left\|\mathbb{E}[\widehat{\bt}] - \bt_0\right\|_\infty = \left\| \mathbb{E}\left[\mathbf{b}(\bt_0,n) - \mathbf{b}(\widehat{\bt}, n)\right] \right\|_\infty \leq \mathbb{E} \left[\left\| \mathbf{B}(\bt^*,n) \right\|_\infty \| \widehat{\bt}-\bt_0 \|_\infty\right],
\end{aligned}
\end{equation*}
where $\bt^* \in \bT$ is on the line connecting $\widehat{\bt}$ and $\bt_0$. By Assumption \ref{assum:C}, we have $\left\| \mathbf{B}(\bt^*,n) \right\|_\infty = \mathcal{O}_{\rm p}(n^{-\beta})$. Moreover, by \eqref{eq:appA:1}, we have
\begin{equation}
\label{eq:appA:9}
     \left\|\widehat{\bt} - \bt_0 \right\|_\infty \leq  \left\|\mathbf{b}(\bt_0, n) - \mathbf{b}(\widehat{\bt},n)\right\|_\infty + \left\|\mathbf{v}(\bt_0, n) - \frac{1}{H} \sum_{h=1}^H \mathbf{v}_h(\widehat{\bt},n) \right\|_\infty = \mathcal{O}_{\rm p}(n^{-\alpha}),
\end{equation}
hence $\left\| \mathbf{B}(\bt^*,n) \right\|_\infty \| \widehat{\bt}-\bt_0 \|_\infty = \mathcal{O}_{\rm p}(n^{-(\alpha+\beta)})$ . Therefore, we have 
\begin{equation*}
    \left\|\mathbb{E}[\widehat{\bt}] - \bt_0\right\|_\infty \leq \mathbb{E} \left[\left\| \mathbf{B}(\bt^*,n) \right\|_\infty \| \widehat{\bt}-\bt_0 \|_\infty\right] = \mathcal{O}(n^{-(\alpha+\beta)}),
\end{equation*}
which leads to
\begin{equation}
\label{eq:appA:14}
    \left\|\mathbb{E}[\widehat{\bt}] - \bt_0\right\|_2 = \mathcal{O}(p^{1/2}n^{-(\alpha + \beta)}). 
\end{equation}
\hfill $\square$

\subsection{Proof of the Third Result}
\label{app:proof:JINI:bias:finite:3}

Under Assumption \ref{assum:C3}, we obtain from \eqref{eq:appA:2} that 
\begin{equation*}
    \mathbb{E}[\widehat{\bt}] -\bt_0 = -\mathbf{B}(n) \left\{\mathbb{E}[\widehat{\bt}] -\bt_0\right\} + \mathbf{r}(\bt_0, n) - \mathbb{E}\left[\mathbf{r}(\widehat{\bt}, n)\right].
\end{equation*}
By rearranging the terms, we have
\begin{equation*}
    \mathbb{E}[\widehat{\bt}] - \bt_0 = \left\{\mathbf{I} + \mathbf{B}(n)\right\}^{-1} \left\{\mathbf{r}(\bt_0,n) - \mathbb{E}\left[\mathbf{r}(\widehat{\bt},n) \right] \right\}.
\end{equation*}
Moreover, by Lemma \ref{lem:1} and under Assumption \ref{assum:C3}, for sufficiently large $n$, we have 
\begin{equation*}
    \left\|\mathbb{E}[\widehat{\bt}] - \bt_0\right\|_\infty \leq \left\|\left\{\mathbf{I} + \mathbf{B}(n)\right\}^{-1}\right\|_\infty \left\|\mathbf{r}(\bt_0,n) - \mathbb{E}\left[\mathbf{r}(\widehat{\bt},n) \right] \right\|_\infty = \mathcal{O}(n^{-\beta_3}),
\end{equation*}
and therefore,
\begin{equation}
\label{eq:appA:13}
    \left\|\mathbb{E}[\widehat{\bt}] - \bt_0\right\|_2 = \mathcal{O}(p^{1/2}n^{-\beta_3}).
\end{equation}
Since Assumption \ref{assum:C3} implies Assumption \ref{assum:C}, we can incorporate the result of \eqref{eq:appA:14} under Assumption \ref{assum:C} to the current result of \eqref{eq:appA:13} under Assumption \ref{assum:C3}, i.e. we have 
\begin{equation}
\label{eq:appA:15}
    \left\|\mathbb{E}[\widehat{\bt}] - \bt_0\right\|_2 = \mathcal{O}\left(p^{1/2} n^{-\gamma_3}\right), 
\end{equation}
where $\gamma_3 \vcentcolon = \max(\beta_3, \alpha+\beta)$.

\hfill $\square$

\subsection{Proof of the Second Result}
\label{app:proof:JINI:bias:finite:2}

Similarly to Section \ref{app:proof:JINI:bias:finite:3}, under Assumption \ref{assum:C2}, we obtain from \eqref{eq:appA:2} that
\begin{equation*}
    \mathbb{E}[\widehat{\bt}] -\bt_0 = -\mathbf{B}(n) \left\{\mathbb{E}[\widehat{\bt}] -\bt_0\right\} + \mathbf{r}(\bt_0, n) - \mathbb{E}\left[\mathbf{r}(\widehat{\bt}, n)\right],
\end{equation*}
and therefore,
\begin{equation}
\label{eq:appA:6}
    \mathbb{E}[\widehat{\bt}] - \bt_0 = \left\{\mathbf{I} + \mathbf{B}(n)\right\}^{-1} \left\{\mathbf{r}(\bt_0,n) - \mathbb{E}\left[\mathbf{r}(\widehat{\bt},n) \right] \right\}.
\end{equation}
Under Assumptions \ref{assum:A} and \ref{assum:C2}, we have
\begin{equation*}
    \left\| \mathbb{E}\left[\mathbf{r}(\widehat{\bt},n)\right] \right\|_\infty \leq \underset{i=1,\ldots,p}{\max} \sum_{j=1}^p \sum_{k=1}^p |r_{ijk}| \cdot n^{-\beta_3} \cdot  \left|\mathbb{E}[\widehat{\theta}_j \widehat{\theta}_k]\right| + \left\|\mathbb{E}\left[\mathbf{e}(\widehat{\bt}, n)\right]\right\|_\infty = \mathcal{O}(n^{-\beta_3}),
\end{equation*}
where $\widehat{\theta}_j$ denotes the $j^{th}$ element of $\widehat{\bt}$. Similarly, $\left\|\mathbf{r}(\bt_0, n)\right\|_\infty = \mathcal{O}(n^{-\beta_3})$. Plugging the result to \eqref{eq:appA:6}, by Lemma \ref{lem:1} for sufficiently large $n$, we obtain
\begin{equation}
\label{eq:appA:7}
    \left\|\mathbb{E}[\widehat{\bt}] - \bt_0\right\|_\infty \leq \left\|\left\{\mathbf{I} + \mathbf{B}(n)\right\}^{-1}\right\|_\infty \left\|\mathbf{r}(\bt_0,n) - \mathbb{E}\left[\mathbf{r}(\widehat{\bt},n) \right] \right\|_\infty = \mathcal{O}(n^{-\beta_3}).
\end{equation}
We now further refine the order of $\mathbf{r}(\bt_0,n) - \mathbb{E}\left[\mathbf{r}(\widehat{\bt},n) \right]$. Using \eqref{eq:appA:7}, we obtain
\begin{equation}
\label{eq:appA:8}
    \mathbb{E}[\widehat{\theta}_j] \mathbb{E}[\widehat{\theta}_k] = \{\theta_j + \mathcal{O}(n^{-\beta_3})\} \{\theta_k + \mathcal{O}(n^{-\beta_3})\} = \theta_j\theta_k + \mathcal{O}(n^{-\beta_3}),
\end{equation}
for all $j,k=1,\ldots,p$, where $\theta_j$ denotes the $j^{th}$ element of $\bt_0$. Moreover, recall from \eqref{eq:appA:9} that we have $\left\|\widehat{\bt} - \bt_0 \right\|_\infty = \mathcal{O}_{\rm p}(n^{-\alpha})$, and thus
\begin{equation}
\begin{aligned}
\label{eq:appA:10}
    \var(\widehat{\theta_i}) &= \mathbb{E}\left[\left\{\widehat{\theta}_i - \mathbb{E}[\widehat{\theta}_i]\right\}^2\right] = \mathbb{E}\left[\left\{\widehat{\theta}_i - \theta_i - \left(\mathbb{E}[\widehat{\theta}_i] - \theta_i\right)\right\}^2\right] \\
    &= \mathbb{E}\left[(\widehat{\theta}_i - \theta_i)^2 - 2(\widehat{\theta}_i - \theta_i)\left\{\mathbb{E}[\widehat{\theta}_i]-\theta_i\right\} + \left\{\mathbb{E}[\widehat{\theta}_i]-\theta_i\right\}^2\right] \\
    &= \mathbb{E}\left[(\widehat{\theta}_i - \theta_i)^2\right] - \left\{\mathbb{E}[\widehat{\theta}_i]-\theta_i\right\}^2 \\
    &= \mathcal{O}(n^{-2\alpha}) + \mathcal{O}(n^{-2\beta_3}) \\
    &= \mathcal{O}(n^{-2\alpha}),
\end{aligned}
\end{equation}
for all $i=1,\ldots,p$. Using the results in \eqref{eq:appA:8} and \eqref{eq:appA:10}, we have
\begin{equation}
\begin{aligned}
\label{eq:appA:11}
    \left|\theta_j\theta_k - \mathbb{E}[\widehat{\theta}_j \widehat{\theta}_k]\right| &= \left| \theta_j \theta_k - \mathbb{E}[\widehat{\theta}_j] \mathbb{E}[\widehat{\theta}_k] -\cov(\widehat{\theta}_j, \widehat{\theta}_k)\right| \\
    &\leq \left| \theta_j \theta_k - \mathbb{E}[\widehat{\theta}_j] \mathbb{E}[\widehat{\theta}_k] \right| + \left| \cov(\widehat{\theta}_j, \widehat{\theta}_k)\right|
    &= \mathcal{O}(n^{-\min(2\alpha, \beta_3)}),
\end{aligned}
\end{equation}
for all $j,k=1,\ldots,p$, since $\left| \cov(\widehat{\theta}_j, \widehat{\theta}_k)\right| \leq \sqrt{ \var(\widehat{\theta}_j)\var(\widehat{\theta}_k) } = \mathcal{O}(n^{-2\alpha})$. Define $\mathbf{u} \vcentcolon = [u_1 \ldots u_p]\tt$, where
\begin{equation*}
    u_i \vcentcolon = \sum_{j=1}^p \sum_{k=1}^p r_{ijk} n^{-\beta_3} \left(\theta_j \theta_k - \mathbb{E}[\widehat{\theta}_j \widehat{\theta}_k]\right),
\end{equation*}
then under Assumption \ref{assum:C2} and using the result in \eqref{eq:appA:11}, we obtain
\begin{equation*}
\begin{aligned}
    \left\| \mathbf{r}(\bt_0,n) - \mathbb{E}\left[\mathbf{r}(\widehat{\bt},n)\right]\right\|_\infty &= \left\|\mathbf{u} + \mathbf{e}(\bt_0, n) - \mathbb{E}\left[\mathbf{e}(\widehat{\bt},n)\right]\right\|_\infty \\
    &\leq \|\mathbf{u}\|_\infty + \left\|\mathbf{e}(\bt_0, n) - \mathbb{E}\left[\mathbf{e}(\widehat{\bt},n)\right]\right\|_\infty \\
    &\leq \underset{i=1,\ldots,p}{\max}\sum_{j=1}^p \sum_{k=1}^p |r_{ijk}| \cdot n^{-\beta_3} \cdot \left|\theta_j \theta_k - \mathbb{E}[\widehat{\theta}_j \widehat{\theta}_k]\right| +  \mathcal{O}(n^{-\beta_4}), \\
    &= \mathcal{O}(n^{-\min(2\alpha+\beta_3, 2\beta_3, \beta_4)}). \\
\end{aligned}
\end{equation*}
Re-evaluating \eqref{eq:appA:7}, for sufficiently large $n$, we have 
\begin{equation*}
    \left\|\mathbb{E}[\widehat{\bt}] - \bt_0\right\|_\infty \leq \left\|\left\{\mathbf{I} + \mathbf{B}(n)\right\}^{-1}\right\|_\infty \left\|\mathbf{r}(\bt_0,n) - \mathbb{E}\left[\mathbf{r}(\widehat{\bt},n) \right] \right\|_\infty = \mathcal{O}(n^{-\min(2\alpha+\beta_3, 2\beta_3, \beta_4)}),
\end{equation*}
and therefore,
\begin{equation}
\label{eq:appA:12}
    \left\|\mathbb{E}[\widehat{\bt}] - \bt_0\right\|_2 = \mathcal{O}(p^{1/2} n^{-\min(2\alpha+\beta_3, 2\beta_3, \beta_4)}).
\end{equation}

Since Assumption \ref{assum:C2} implies Assumption \ref{assum:C3}, we can incorporate the result of \eqref{eq:appA:15} under Assumption \ref{assum:C3} to the current result of \eqref{eq:appA:12} under Assumption \ref{assum:C2}, i.e. we have 
\begin{equation*}
    \left\|\mathbb{E}[\widehat{\bt}] - \bt_0\right\|_2 = \mathcal{O}\left(p^{1/2} n^{-\gamma_2}\right), 
\end{equation*}
where $\gamma_2 \vcentcolon = \max\left\{\min(2\alpha+\beta_3, 2\beta_3, \beta_4), \gamma_3 \right\} = \max\left\{\min(2\alpha+\beta_3, 2\beta_3, \beta_4), \beta_3, \alpha + \beta \right\}$.

Moreover, under Assumption \ref{assum:C2}, we know that $\beta_4 > \beta_3$ and thus, 
\begin{equation*}
    \gamma_2 = \max\{\min(2\alpha+\beta_3, 2\beta_3, \beta_4), \alpha+\beta\}.
\end{equation*}

\hfill $\square$

\subsection{Proof of the First Result}
\label{app:proof:JINI:bias:finite:1}

Under Assumption \ref{assum:C1}, \eqref{eq:appA:2} can be re-evaluated as 
\begin{equation*}
    \mathbb{E}[\widehat{\bt}] - \bt_0 = -\mathbf{B}(n)\left\{\mathbb{E}[\widehat{\bt}] - \bt_0\right\},
\end{equation*}
or equivalently,
\begin{equation*}
    \left\{\mathbf{I}+\mathbf{B}(n)\right\} \left\{\mathbb{E}[\widehat{\bt}] - \bt_0\right\} = \0.
\end{equation*}
Since under Assumption \ref{assum:C1} that $\|\mathbf{B}(n)\|_\infty = \mathcal{O}(n^{-\beta_1})$, from Lemma \ref{lem:1} we have $\left\{\mathbf{I}+\mathbf{B}(n)\right\}$ is nonsingular for sufficiently large $n$, and therefore, $\mathbb{E}[\widehat{\bt}] - \bt_0 = \0$ for sufficiently large $n$.

\hfill $\square$

This completes the proof of Theorem \ref{thm:bias:finite}, and obviously, we have $\gamma_2 \geq \gamma_3 \geq \alpha + \beta$.

\hfill $\square$

\section{Proof of Proposition \ref{prop:bias:finite}}
\label{app:proof:BBC:bias:finite}

By definition in \eqref{eqn:def:BBC}, the BBC satisfies
\begin{equation}
\begin{aligned}
\label{eq:appB:1}
    \widetilde{\bt} &= 2\widehat{\bpi}(\bt_0, n) - \frac{1}{H}\sum_{h=1}^H \widehat{\bpi}_h\left\{\widehat{\bpi}(\bt_0,n), n\right\} \\
    &= \widehat{\bpi}(\bt_0, n) + \bt_0 + \mathbf{b}(\bt_0,n) + \mathbf{v}(\bt_0,n) 
    \\
    & \hspace{0.5cm} - \left[\widehat{\bpi}(\bt_0, n) + \mathbf{b}\left\{\widehat{\bpi}(\bt_0, n),n\right\} + \frac{1}{H}\sum_{h=1}^H \mathbf{v}_h\left\{\widehat{\bpi}(\bt_0, n), n\right\}\right] \\
    &= \bt_0 + \mathbf{b}(\bt_0,n) - \mathbf{b}\left\{\widehat{\bpi}(\bt_0, n),n\right\} + \mathbf{v}(\bt_0,n) - \frac{1}{H}\sum_{h=1}^H \mathbf{v}_h\left\{\widehat{\bpi}(\bt_0, n), n\right\}.
\end{aligned}
\end{equation}
By rearranging the terms and taking expectations on both sides of \eqref{eq:appB:1}, we obtain
\begin{equation}
\label{eq:appB:2}
    \mathbb{E}[\widetilde{\bt}] - \bt_0 = \mathbf{b}(\bt_0,n) - \mathbb{E}\left[\mathbf{b}\left\{\widehat{\bpi}(\bt_0, n),n\right\}\right].
\end{equation}

\subsection{Proof of the Fourth Result}
\label{app:proof:BBC:bias:finite:4}

By the mean value theorem and under Assumption \ref{assum:C}, we have, from \eqref{eq:appB:2}, that
\begin{equation*}
   \left\|\mathbb{E}[\widetilde{\bt}] - \bt_0\right\|_\infty = \left\|\mathbb{E}\left[\mathbf{b}(\bt_0,n) - \mathbf{b}\left\{\widehat{\bpi}(\bt_0, n),n\right\}\right]\right\|_\infty \leq \mathbb{E}\left[\left\|\mathbf{B}(\bpi^*,n)\right\|_\infty \left\|\widehat{\bpi}(\bt_0, n) - \bt_0\right\|_\infty\right],
\end{equation*}
where $\bpi^* \in \bT$ is on the line connecting $\widehat{\bpi}(\bt_0, n)$ and $\bt_0$. By Assumption \ref{assum:C}, we have $\left\|\mathbf{B}(\bpi^*,n)\right\|_\infty = \mathcal{O}_{\rm p}(n^{-\beta})$. Moreover, we have 
\begin{equation}
\label{eq:appB:5}
    \left\|\widehat{\bpi}(\bt_0, n) - \bt_0\right\|_\infty = \left\|\mathbf{b}(\bt_0, n) + \mathbf{v}(\bt_0,n)\right\|_\infty \leq \left\|\mathbf{b}(\bt_0, n)\right\|_\infty + \left\| \mathbf{v}(\bt_0,n)\right\|_\infty = \mathcal{O}_{\rm p}(n^{-\alpha}),
\end{equation}
and thus, $\left\|\mathbf{B}(\bpi^*,n)\right\|_\infty \left\|\widehat{\bpi}(\bt_0, n) - \bt_0\right\|_\infty = \mathcal{O}_{\rm p}(n^{-(\alpha+\beta)})$. Therefore, we have 
\begin{equation*}
    \left\|\mathbb{E}[\widetilde{\bt}] - \bt_0\right\|_\infty \leq \mathbb{E}\left[\left\|\mathbf{B}(\bpi^*,n)\right\|_\infty \left\|\widehat{\bpi}(\bt_0, n) - \bt_0\right\|_\infty\right] = \mathcal{O}(n^{-(\alpha+\beta)}),
\end{equation*}
which leads to 
\begin{equation}
\label{eq:appB:fourth}
    \left\|\mathbb{E}[\widetilde{\bt}] - \bt_0\right\|_2 = \mathcal{O}(p^{1/2}n^{-(\alpha + \beta)}).
\end{equation}

\hfill $\square$

\subsection{Proof of the Third Result}
\label{app:proof:BBC:bias:finite:3}
Under Assumption \ref{assum:C3}, we obtain from \eqref{eq:appB:2} that 
\begin{equation*}
\begin{aligned}
    \mathbb{E}[\widetilde{\bt}] - \bt_0 &= \mathbf{B}(n)\bt_0+\mathbf{c}(n)+\mathbf{r}(\bt_0,n) - \mathbb{E}\left[\mathbf{B}(n)\widehat{\bpi}(\bt_0, n) + \mathbf{c}(n) + \mathbf{r}\left\{\widehat{\bpi}(\bt_0,n),n\right\}\right]\\
    &= -\mathbf{B}(n)\mathbf{b}(\bt_0,n) + \mathbf{r}(\bt_0,n) - \mathbb{E}\left[\mathbf{r}\left\{\widehat{\bpi}(\bt_0,n),n\right\}\right],
\end{aligned}
\end{equation*}
and therefore, 
\begin{equation*}
\begin{aligned}
     \left\|\mathbb{E}[\widetilde{\bt}] - \bt_0\right\|_\infty &\leq \left\|\mathbf{B}(n)\right\|_\infty \left\|\mathbf{b}(\bt_0,n)\right\|_\infty + \left\|\mathbf{r}(\bt_0,n) - \mathbb{E}\left[\mathbf{r}\left\{\widehat{\bpi}(\bt_0,n),n\right\}\right]\right\|_\infty 
     \\
        &= \mathcal{O}(n^{-\min(\beta+\beta_1, \beta_3)}),
\end{aligned}
\end{equation*}
which leads to 
\begin{equation}
\label{eq:appB:third}
    \left\|\mathbb{E}[\widetilde{\bt}] - \bt_0\right\|_2 = \mathcal{O}(p^{1/2}n^{-\min(\beta+\beta_1, \beta_3)}).
\end{equation}

Since Assumption \ref{assum:C3} implies Assumption \ref{assum:C}, we can incorporate the result of \eqref{eq:appB:fourth} under Assumption \ref{assum:C} to the current result of \eqref{eq:appB:third} under Assumption \ref{assum:C3}, i.e. we have 
\begin{equation}
\label{eq:appB:third:final}
    \left\|\mathbb{E}[\widetilde{\bt}] - \bt_0\right\|_2 = \mathcal{O}(p^{1/2}n^{-\upsilon_3}), 
\end{equation}
where $\upsilon_3 \vcentcolon = \max \left\{\min(\beta+\beta_1, \beta_3), \alpha+\beta \right\}$. 

\hfill $\square$

\subsection{Proof of the Second Result}
\label{app:proof:BBC:bias:finite:2}

Similarly to Section \ref{app:proof:BBC:bias:finite:3}, under Assumption \ref{assum:C2}, we obtain from \eqref{eq:appB:2} that 
\begin{equation}
\label{eq:appB:3}
    \mathbb{E}[\widetilde{\bt}] - \bt_0 = -\mathbf{B}(n)\mathbf{b}(\bt_0,n) + \mathbf{r}(\bt_0,n) - \mathbb{E}\left[\mathbf{r}\left\{\widehat{\bpi}(\bt_0,n),n\right\}\right].
\end{equation}
We now want to refine the order of $\mathbf{r}(\bt_0,n) - \mathbb{E}\left[\mathbf{r}\left\{\widehat{\bpi}(\bt_0,n),n\right\}\right]$. Under Assumption~\ref{assum:C2}, we have
\begin{equation}
\label{eq:appB:4}
    \mathbb{E}[\widehat{\pi}_j]\mathbb{E}[\widehat{\pi}_k] = \left\{\theta_j + \mathcal{O}(n^{-\beta})\right\} \left\{\theta_k + \mathcal{O}(n^{-\beta})\right\} = \theta_j \theta_k + \mathcal{O}(n^{-\beta}),
\end{equation}
for all $j,k=1,\ldots,p$, where $\widehat{\pi}_j$ denotes the $j^{th}$ element of $\widehat{\bpi}(\bt_0,n)$. Moreover, recall from \eqref{eq:appB:5} that we have $\left\|\widehat{\bpi}(\bt_0, n) - \bt_0\right\|_\infty = \mathcal{O}_{\rm p}(n^{-\alpha})$, and thus, 
\begin{equation}
\begin{aligned}
\label{eq:appB:6}
    \var(\widehat{\pi}_i) &= \mathbb{E}\left[\left\{\widehat{\pi}_i - \mathbb{E}\left[\widehat{\pi}_i\right]\right\}^2\right] = \mathbb{E}\left[\left\{\widehat{\pi}_i - \theta_i - b_i(\bt_0,n)\right\}^2\right] \\
    &= \mathbb{E}\left[\left(\widehat{\pi}_i - \theta_i\right)^2 - 2\left(\widehat{\pi}_i - \theta_i\right)b_i(\bt_0,n) + b_i(\bt_0,n)^2\right] \\
    &= \mathbb{E}\left[\left(\widehat{\pi}_i - \theta_i\right)^2\right] -  b_i(\bt_0,n)^2 \\
    &= \mathcal{O}(n^{-2\alpha}) + \mathcal{O}(n^{-2\beta}) \\
    &= \mathcal{O}(n^{-2\alpha}),
\end{aligned}
\end{equation}
for all $i=1,\ldots,p$. Using the results in \eqref{eq:appB:4} and \eqref{eq:appB:6}, we have
\begin{equation}
\begin{aligned}
\label{eq:appB:7}
    \left|\theta_j\theta_k - \mathbb{E}[\widehat{\pi}_j\widehat{\pi}_k]\right| &= \left|\theta_j\theta_k - \mathbb{E}[\widehat{\pi}_j]\mathbb{E}[\widehat{\pi}_k] - \cov(\widehat{\pi}_j, \widehat{\pi}_k)\right| \\
    &\leq \left|\theta_j\theta_k - \mathbb{E}[\widehat{\pi}_j]\mathbb{E}[\widehat{\pi}_k]\right| + \left|\cov(\widehat{\pi}_j, \widehat{\pi}_k)\right| = \mathcal{O}(n^{-\min(\beta, 2\alpha)}),
\end{aligned}
\end{equation}
for all $j,k=1,\ldots,p$, since $\left| \cov(\widehat{\pi}_j, \widehat{\pi}_k)\right| \leq \sqrt{ \var(\widehat{\pi}_j)\var(\widehat{\pi}_k)} = \mathcal{O}(n^{-2\alpha})$. Define $\mathbf{u} \vcentcolon = [u_1 \ldots u_p]\tt$, where
\begin{equation*}
    u_i \vcentcolon = \sum_{j=1}^p \sum_{k=1}^p r_{ijk} n^{-\beta_3} \left(\theta_j \theta_k - \mathbb{E}[\widehat{\pi}_j\widehat{\pi}_k]\right),
\end{equation*}
then under Assumption \ref{assum:C2} and using the result in \eqref{eq:appB:7}, we obtain 
\begin{equation*}
\begin{aligned}
    \left\|\mathbf{r}(\bt_0,n) - \mathbb{E}\left[\mathbf{r}\left\{\widehat{\bpi}(\bt_0,n),n\right\}\right]\right\|_\infty &= \left\|\mathbf{u} + \mathbf{e}(\bt_0,n) - \mathbf{e}\left\{\widehat{\bpi}(\bt_0,n),n\right\}\right\|_\infty \\
    &\leq \left\|\mathbf{u}\right\|_\infty + \left\|\mathbf{e}(\bt_0,n) - \mathbf{e}\left\{\widehat{\bpi}(\bt_0,n),n\right\}\right\|_\infty \\
    &\leq \underset{i=1,\ldots,p}{\max} \sum_{j=1}^p \sum_{k=1}^p |r_{ijk}|\cdot n^{-\beta_3} \cdot \left|\theta_j\theta_k - \mathbb{E}[\widehat{\pi}_j\widehat{\pi}_k]\right| + \mathcal{O}(n^{-\beta_4}) \\
    &= \mathcal{O}\left(n^{-\min(\beta+\beta_3, 2\alpha+\beta_3, \beta_4)}\right).
\end{aligned}
\end{equation*}
By \eqref{eq:appB:3} we have
\begin{equation*}
\begin{aligned}
    \left\|\mathbb{E}[\widetilde{\bt}] - \bt_0\right\|_\infty &\leq \left\|\mathbf{B}(n)\mathbf{b}(\bt_0,n)\right\|_\infty + \left\|\mathbf{r}(\bt_0,n) - \mathbb{E}\left[\mathbf{r}\left(\widehat{\bpi}(\bt_0,n),n\right)\right]\right\|_\infty \\
    &\leq \left\|\mathbf{B}(n)\right\|_\infty \left\|\mathbf{b}(\bt_0,n)\right\|_\infty + \left\|\mathbf{r}(\bt_0,n) - \mathbb{E}\left[\mathbf{r}\left(\widehat{\bpi}(\bt_0,n),n\right)\right]\right\|_\infty \\
    &= \mathcal{O}\left(n^{-(\beta+\beta_1)}\right) + \mathcal{O}\left(n^{-\min(\beta+\beta_3, 2\alpha+\beta_3, \beta_4)}\right) \\
    &= \mathcal{O}\left(n^{-\min(\beta+\beta_1, \beta+\beta_3, 2\alpha+\beta_3, \beta_4)}\right),
\end{aligned}
\end{equation*}
which leads to 
\begin{equation}
\label{eq:appB:second}
    \left\|\mathbb{E}[\widetilde{\bt}] - \bt_0\right\|_2 = \mathcal{O}(p^{1/2}n^{-\min(\beta+\beta_1, \beta+\beta_3, 2\alpha+\beta_3, \beta_4)}).
\end{equation}

Since Assumption \ref{assum:C2} implies Assumption \ref{assum:C3}, we can incorporate the result of \eqref{eq:appB:third:final} under Assumption \ref{assum:C3} to the current result of \eqref{eq:appB:second} under Assumption \ref{assum:C2}, i.e. we have 
\begin{equation}
\label{eq:appB:second:final}
    \left\|\mathbb{E}[\widetilde{\bt}] - \bt_0\right\|_2 = \mathcal{O}(p^{1/2}n^{-\upsilon_2}), 
\end{equation}
where 
\begin{equation}
\begin{aligned}
\label{eq:appB:upsilon2:tmp}
    \upsilon_2 & \vcentcolon = \max\left\{ \min(\beta+\beta_1, \beta+\beta_3, 2\alpha+\beta_3, \beta_4), \upsilon_3 \right\} \\
    &= \max\left\{ \min(\beta+\beta_1, \beta+\beta_3, 2\alpha+\beta_3, \beta_4), \min(\beta+\beta_1, \beta_3), \alpha+\beta \right\}.
\end{aligned}
\end{equation}
Moreover, recall that under Assumption \ref{assum:C2}, we have $\beta_4 > \beta_3 > \alpha$. If $\beta+\beta_1 \leq \beta_3$, then $\min(\beta+\beta_1, \beta+\beta_3, 2\alpha+\beta_3, \beta_4) = \beta + \beta_1$, which leads to 
\begin{equation*}
\begin{aligned}
    &\max\left\{ \min(\beta+\beta_1, \beta+\beta_3, 2\alpha+\beta_3, \beta_4), \min(\beta+\beta_1, \beta_3)\right\} \\
    = &\min(\beta+\beta_1, \beta+\beta_3, 2\alpha+\beta_3, \beta_4) = \beta + \beta_1.
\end{aligned}
\end{equation*}
If $\beta+\beta_1 > \beta_3$, then 
\begin{equation}
\begin{aligned}
    &\max\left\{ \min(\beta+\beta_1, \beta+\beta_3, 2\alpha+\beta_3, \beta_4), \min(\beta+\beta_1, \beta_3)\right\} \\
    = &\max\left\{ \min(\beta+\beta_1, \beta+\beta_3, 2\alpha+\beta_3, \beta_4), \beta_3 \right\} \\
    = & \min(\beta+\beta_1, \beta+\beta_3, 2\alpha+\beta_3, \beta_4). 
\end{aligned}
\end{equation}
Therefore, the result in \eqref{eq:appB:upsilon2:tmp} can be further simplified to 
\begin{equation*}
\begin{aligned}
    \upsilon_2 &= \max\left\{ \min(\beta+\beta_1, \beta+\beta_3, 2\alpha+\beta_3, \beta_4), \min(\beta+\beta_1, \beta_3), \alpha+\beta \right\} \\
    &= \max\left\{ \min(\beta+\beta_1, \beta+\beta_3, 2\alpha+\beta_3, \beta_4), \alpha+\beta \right\}.
\end{aligned}
\end{equation*}
\hfill $\square$

\subsection{Proof of the First Result}
\label{app:proof:BBC:bias:finite:1}

Under Assumption \ref{assum:C1}, \eqref{eq:appB:2} can be re-evaluated as
\begin{equation*}
\begin{aligned}
    \mathbb{E}[\widetilde{\bt}] - \bt_0 &= \mathbf{B}(n)\bt_0 + \mathbf{c}(n) - \mathbb{E}\left[\mathbf{B}(n)\widehat{\bpi}(\bt_0,n) + \mathbf{c}(n)\right] \\
    &= -\mathbf{B}(n) \mathbb{E}\left[\widehat{\bpi}(\bt_0,n) - \bt_0\right] \\
    &= -\mathbf{B}(n) \mathbf{b}(\bt_0, n),
\end{aligned}
\end{equation*}
and thus,
\begin{equation*}
    \left\| \mathbb{E}[\widetilde{\bt}] - \bt_0 \right\|_\infty \leq \left\|\mathbf{B}(n)\right\|_\infty \left\|\mathbf{b}(\bt_0,n)\right\|_\infty = \mathcal{O}(n^{-(\beta + \beta_1)}),
\end{equation*}
which leads to 
\begin{equation}
\label{eq:appB:first}
    \left\| \mathbb{E}[\widetilde{\bt}] - \bt_0 \right\|_2 = \mathcal{O}(p^{1/2}n^{-\upsilon_1}),
\end{equation}
where $\upsilon_1 \vcentcolon = \beta+\beta_1$. 

\hfill $\square$

Obviously, we have $\upsilon_2 \geq \upsilon_3 \geq \alpha + \beta$. To complete the proof of Proposition \ref{prop:bias:finite}, we remain to show that $\upsilon_1 \geq \upsilon_2$. Recall that $\upsilon_2 = \max\left\{ \min(\beta+\beta_1, \beta+\beta_3, 2\alpha+\beta_3, \beta_4), \alpha+\beta \right\}$. Since $\upsilon_1 = \beta+\beta_1$, $\beta+\beta_1 \geq \min(\beta+\beta_1, \beta+\beta_3, 2\alpha+\beta_3, \beta_4)$ and $\beta+\beta_1 > \alpha+\beta$, we have $\upsilon_1 \geq \upsilon_2$. 

\hfill $\square$

\subsection{Additional Requirements for the Standard Bias Order of the BBC Estimator}

The bias of the BBC estimator has been shown to be $\mathcal{O}(n^{-2})$ in low dimensional settings and under different assumptions (see e.g. \citealp{hall1988exact}). Under our assumption framework, the same result can be obtained, for example, by adding additional restrictions on  Assumption \ref{assum:C}. Indeed, if $\mathbf{b}(\bt,n)$ is twice continuously differentiable in $\bt \in \bT$, and if
\begin{equation*}
    \underset{i=1,\ldots,p}{\max}\sum_{j=1}^p \sum_{k=1}^p \left|\frac{\partial^2 b_i(\bt,n)}{\partial \theta_j \partial \theta_k}\right| = \mathcal{O}(n^{-\beta}),
\end{equation*}
then using the Taylor's theorem, we can obtain that the BBC estimator $\widetilde{\bt}$ enjoys the bias order of $\mathcal{O}\left(p^{1/2}n^{-\min(2\beta, 2\alpha+\beta)}\right)$, which remains valid in high dimensional settings. With the typical values of $\alpha=1/2$ and $\beta = 1$, we have 
$\| \mathbb{E}[\widetilde{\bt}] - \bt_0 \|_2 = \mathcal{O}(p^{1/2}n^{-2})$. In addition, the JINI estimator enjoys the same bias correction properties.

\section{Proof of Corollary \ref{cor:JINI:BBC}}
\label{app:bias:order:compare}

By Theorem \ref{thm:bias:finite}, under Assumptions \ref{assum:A}, \ref{assum:B} and \ref{assum:C1}, we have $\left\|\mathbb{E}[\widehat{\bt}] -\bt_0 \right\|_2 = 0$ for sufficiently large (but finite) $n$, and $\left\|\mathbb{E}[\widetilde{\bt}] -\bt_0 \right\|_2 = \mathcal{O}\left(p^{1/2}n^{-\upsilon_1}\right)$ where $\upsilon_1 = \beta+\beta_1$. So in this case, we clearly have $\left\|\mathbb{E}[\widehat{\bt}] -\bt_0 \right\|_2 + \left\|\mathbb{E}[\widetilde{\bt}] -\bt_0 \right\|_2 = \mathcal{O}\left(\left\|\mathbb{E}[\widetilde{\bt}] -\bt_0 \right\|_2\right)$. Under Assumption \ref{assum:C}, we have $\left\|\mathbb{E}[\widehat{\bt}] -\bt_0 \right\|_2 = \mathcal{O}\left(p^{1/2}n^{-(\alpha + \beta)}\right)$ and $\left\|\mathbb{E}[\widetilde{\bt}] -\bt_0 \right\|_2 = \mathcal{O}\left(p^{1/2}n^{-(\alpha + \beta)}\right)$, and therefore, we obtain the same conclusion. So to prove Corollary~\ref{cor:JINI:BBC}, we essentially need to prove that $\gamma_i \geq \upsilon_i$ for $i=2,3$, under Assumptions \ref{assum:C2} and \ref{assum:C3}, respectively. 

\subsection{Proof of \texorpdfstring{$\gamma_2\geq\upsilon_2$}{\textgamma2\textgeq\textupsilon2}}
For simplicity, we denote that $\gamma \vcentcolon = \min(2\alpha+\beta_3, 2\beta_3, \beta_4)$, then $\gamma_2 = \max(\gamma, \alpha+\beta)$. As $2\beta_3 \geq \beta+\beta_3$, we have 
\begin{equation*}
\begin{aligned}
    \upsilon_2 &= \max\left\{ \min(\beta+\beta_1, \beta+\beta_3, 2\alpha+\beta_3, \beta_4), \alpha+\beta \right\} \\
    &= \max\left\{ \min(\beta+\beta_1, \beta+\beta_3, 2\alpha+\beta_3, 2\beta_3, \beta_4), \alpha+\beta \right\} \\ 
    &= \max\left\{ \min(\beta+\beta_1, \beta+\beta_3, \gamma), \alpha+\beta \right\}. 
\end{aligned}
\end{equation*}
If $\gamma \geq \alpha + \beta$, then $\gamma_2 = \gamma$. By Assumption \ref{assum:C2}, we also have $\beta_1 > \alpha$ and $\beta_3 > \alpha$, so $\upsilon_2 = \min(\beta+\beta_1, \beta+\beta_3, \gamma) \leq \gamma = \gamma_2$. Therefore, in this case, we have $\gamma_2 \geq \upsilon_2$. If $\gamma < \alpha + \beta$, then $\gamma_2 = \alpha + \beta$ and 
\begin{equation*}
\begin{aligned}
    \upsilon_2 &= \max\left\{ \min(\beta+\beta_1, \beta+\beta_3, \gamma), \alpha+\beta \right\} \\
    &= \max\left(\gamma, \alpha+\beta \right) = \alpha+\beta, 
\end{aligned}
\end{equation*}
%
hence $\upsilon_2 = \gamma_2$ in this case. So overall we have $\gamma_2 \geq \upsilon_2$. 

\hfill $\square$

\subsection{Proof of \texorpdfstring{$\gamma_3\geq\upsilon_3$}{\textgamma3\textgeq\textupsilon3}}
We recall that $\gamma_3 = \max(\beta_3, \alpha+\beta)$ and $\upsilon_3 = \max\{\min(\beta+\beta_1, \beta_3), \alpha+\beta\}$. If $\beta_3 \leq \alpha+\beta$, then $\gamma_3 = \alpha+\beta$. Under Assumption \ref{assum:C3}, we have $\beta_1 > \alpha$ so $\alpha+\beta < \beta+\beta_1$, leading to $\upsilon_3 = \max(\beta_3, \alpha+\beta) = \alpha+\beta = \gamma_3$. Therefore, in this case, we have $\gamma_3 = \upsilon_3$. If $\beta_3 > \alpha+\beta$, then $\gamma_3 = \beta_3$ and $\upsilon_3 = \min(\beta+\beta_1, \beta_3) \leq \beta_3 = \gamma_3$. Hence, in this case, we have $\gamma_3 \geq \upsilon_3$. So overall we have $\gamma_3 \geq \upsilon_3$.  

\hfill $\square$

\section{Proof of Theorem \ref{thm:bias:asymp}}
\label{app:proof:JINI:bias:asymp}

By definition in \eqref{eqn:def:JINI}, the JINI estimator satisfies
\begin{equation}
\begin{aligned}
\label{eq:appC:1}
    \widehat{\bpi}(\bt_0, n) &= \bt_0 + \mathbf{a}(\bt_0) + \mathbf{b}(\bt_0,n) + \mathbf{v}(\bt_0, n) =  \frac{1}{H}\sum_{h=1}^H \widehat{\bpi}_h(\widehat{\bt},n) \\
    &= \widehat{\bt} + \mathbf{a}(\widehat{\bt}) + \mathbf{b}(\widehat{\bt}, n) + \frac{1}{H}\sum_{h=1}^H \mathbf{v}_h(\widehat{\bt}, n).
\end{aligned}
\end{equation}
By rearranging the terms and taking expectations on both sides, we obtain
\begin{equation}
\label{eq:appC:2}
    \mathbb{E}[\widehat{\bt}] - \bt_0 = \mathbf{a}(\bt_0) - \mathbb{E}\left[\mathbf{a}(\widehat{\bt})\right] + \mathbf{b}(\bt_0,n) - \mathbb{E}\left[\mathbf{b}(\widehat{\bt},n)\right].
\end{equation}

\subsection{Proof of the First Result}
\label{app:proof:JINI:bias:asymp:1}

Under Assumptions \ref{assum:C1} and \ref{assum:D1}, \eqref{eq:appC:2} can be evaluated as
\begin{equation*}
    \mathbb{E}[\widehat{\bt}] - \bt_0 = -\mathbf{A} \left\{\mathbb{E}[\widehat{\bt}] - \bt_0\right\} - \mathbf{B}(n)\left\{\mathbb{E}[\widehat{\bt}] - \bt_0\right\},
\end{equation*}
or equivalently,
\begin{equation}
\label{eq:appC:inconsit:Ozero}
    \left\{\mathbf{I}+\mathbf{A}+\mathbf{B}(n)\right\}\left\{\mathbb{E}[\widehat{\bt}] - \bt_0\right\} = \0.
\end{equation}
Moreover, we have 
\begin{equation*}
   (\mathbf{I} + \mathbf{A})^{-1} \left\{\mathbf{I} + \mathbf{A} + \mathbf{B}(n)\right\} = \mathbf{I} + (\mathbf{I} + \mathbf{A})^{-1} \mathbf{B}(n),
\end{equation*}
and, under Assumptions \ref{assum:C1} and \ref{assum:D1}, that
\begin{equation*}
    \left\|(\mathbf{I}+\mathbf{A})^{-1} \mathbf{B}(n)\right\|_\infty \leq \left\|(\mathbf{I}+\mathbf{A})^{-1}\right\|_\infty \left\| \mathbf{B}(n)\right\|_\infty = \mathcal{O}(n^{-\beta_1}).
\end{equation*}
Since $(\mathbf{I} + \mathbf{A})^{-1} \mathbf{B}(n)$ satisfies the condition of Lemma \ref{lem:1}, we have that for sufficiently large~$n$, $\mathbf{I} + (\mathbf{I} + \mathbf{A})^{-1} \mathbf{B}(n)$ is nonsingular. Since $\mathbf{I}+\mathbf{A}+\mathbf{B}(n) = (\mathbf{I}+\mathbf{A}) \left\{\mathbf{I} + (\mathbf{I}+\mathbf{A})^{-1} \mathbf{B}(n)\right\}$, and both $(\mathbf{I}+\mathbf{A})$ and $\mathbf{I} + (\mathbf{I}+\mathbf{A})^{-1} \mathbf{B}(n)$ are nonsingular for sufficiently large $n$, we obtain that $\mathbf{I}+\mathbf{A}+\mathbf{B}(n)$ is nonsingular as well for sufficiently large~$n$. Therefore, by \eqref{eq:appC:inconsit:Ozero}, we have $\mathbb{E}[\widehat{\bt}] - \bt_0 = \0$ for sufficiently large $n$.

$\hfill\square$

\subsection{Proof of the Second Result}
\label{app:proof:JINI:bias:asymp:2}

From \eqref{eq:appC:1} and by the mean value theorem, we have 
\begin{equation}
\begin{aligned}
\label{eq:appC:3-pre}
   \left\| \widehat{\bt} -\bt_0\right\|_\infty &= \left\|\mathbf{a}(\bt_0) - \mathbf{a}(\widehat{\bt}) + \mathbf{b}(\bt_0, n) - \mathbf{b}(\widehat{\bt},n) + \mathbf{v}(\bt_0,n) - \frac{1}{H}\sum_{h=1}^H \mathbf{v}_h(\widehat{\bt},n)\right\|_\infty \\
    &\leq \left\|\mathbf{A}(\bt^*)\right\|_\infty\left\| \widehat{\bt} -\bt_0\right\|_\infty + \left\|\mathbf{b}(\bt_0, n) - \mathbf{b}(\widehat{\bt},n) +\mathbf{v}(\bt_0,n) - \frac{1}{H}\sum_{h=1}^H \mathbf{v}_h(\widehat{\bt},n)\right\|_\infty, 
\end{aligned}
\end{equation}
where $\bt^* \in \bT$ is on the line connecting $\widehat{\bt}$ and $\bt_0$.
Therefore, under Assumptions \ref{assum:B}, \ref{assum:C} and \ref{assum:D}, by Lemma~\ref{lem:1} and by rearranging \eqref{eq:appC:3-pre} we have
\begin{equation}
\begin{aligned}
\label{eq:appC:3}    
    \left\|\widehat{\bt} -\bt_0\right\|_\infty &\leq \left\{1-\left\|\mathbf{A}(\bt^*)\right\|_\infty\right\}^{-1} \left\|\mathbf{b}(\bt_0, n) - \mathbf{b}(\widehat{\bt},n) + \mathbf{v}(\bt_0,n) - \frac{1}{H}\sum_{h=1}^H \mathbf{v}_h(\widehat{\bt},n)\right\|_\infty \\
    &\leq \left\{1-\left\|\mathbf{A}(\bt^*)\right\|_\infty\right\}^{-1} \left\|\mathbf{b}(\bt_0, n) - \mathbf{b}(\widehat{\bt},n)\right\|_\infty \\
    &\quad + \left\{1-\left\|\mathbf{A}(\bt^*)\right\|_\infty\right\}^{-1} \left\|\mathbf{v}(\bt_0,n) - \frac{1}{H}\sum_{h=1}^H \mathbf{v}_h(\widehat{\bt},n)\right\|_\infty \\ 
    &= \mathcal{O}_{\rm p}(n^{-\beta}) + \mathcal{O}_{\rm p}(n^{-\alpha}) = \mathcal{O}_{\rm p}(n^{-\alpha}).
\end{aligned}    
\end{equation}
In addition, by Taylor's theorem, we have 
\begin{equation}
\label{eq:appC:4}
    \mathbf{a}(\widehat{\bt}) - \mathbf{a}(\bt_0) = \mathbf{A}(\bt_0)(\widehat{\bt} - \bt_0) + \mathbf{u},
\end{equation}
with $\mathbf{u} = [u_1 \ldots u_p]\tt$ and 
\begin{equation}
\label{eq:appC:5}
    u_i \vcentcolon = \sum_{j=1}^p \sum_{k=1}^p g_{ijk}(\widehat{\bt}) \Delta_j\Delta_k,
\end{equation}
where $\Delta_j$ denotes the $j^{th}$ element of $(\widehat{\bt} - \bt_0)$ and $g_{ijk}(\bt): \bT \to \real$ is some function such that, under Assumptions \ref{assum:A} and \ref{assum:D}, we have
\begin{equation}
\label{eq:appC:6}
    \left|g_{ijk}(\bt)\right| \leq \frac{1}{2} \underset{\bt \in \bT}{\max}\left|\frac{\partial^2 a_i(\bt)}{\partial \theta_j \partial \theta_k}\right| = \frac{1}{2} \underset{\bt \in \bT}{\max} \left|G_{ijk}(\bt)\right|,
\end{equation}
for all $\bt \in \bT$. 
%
From Assumption \ref{assum:D} and Lemma \ref{lem:1}, we have $\left\|\left\{\mathbf{I} + \mathbf{A}(\bt_0)\right\}^{-1}\right\|_\infty = \mathcal{O}(1)$. Moreover, from \eqref{eq:appC:2}, we have
\begin{equation*}
    \mathbb{E}[\widehat{\bt}] - \bt_0 = -\mathbf{A}(\bt_0) \left\{\mathbb{E}[\widehat{\bt}] - \bt_0\right\} - \mathbb{E}[\mathbf{u}] + \mathbf{b}(\bt_0,n) - \mathbb{E}\left[\mathbf{b}(\widehat{\bt}, n)\right],
\end{equation*}
or equivalently, 
\begin{equation*}
    \mathbb{E}[\widehat{\bt}] - \bt_0 = -\left\{\mathbf{I} + \mathbf{A}(\bt_0)\right\}^{-1} \left\{\mathbb{E}[\mathbf{u}] + \mathbb{E}\left[\mathbf{b}(\widehat{\bt}, n)\right] - \mathbf{b}(\bt_0,n)\right\},
\end{equation*}
and therefore, 
\begin{equation}
\begin{aligned}
\label{eq:appC:8}
    \left\|\mathbb{E}[\widehat{\bt}] - \bt_0\right\|_\infty &\leq \left\|\left\{\mathbf{I} + \mathbf{A}(\bt_0)\right\}^{-1}\right\|_\infty \left\|\mathbb{E}[\mathbf{u}] + \mathbb{E}\left[\mathbf{b}(\widehat{\bt}, n)\right] - \mathbf{b}(\bt_0,n)\right\|_\infty \\
    &\leq \left\|\left\{\mathbf{I} + \mathbf{A}(\bt_0)\right\}^{-1}\right\|_\infty \left\{\left\|\mathbb{E}[\mathbf{u}]\right\|_\infty + \left\|\mathbb{E}\left[\mathbf{b}(\widehat{\bt}, n)\right] - \mathbf{b}(\bt_0,n)\right\|_\infty\right\}. 
\end{aligned} 
\end{equation}
We now want to obtain the order of $\mathbb{E}[\mathbf{u}]$ and $\mathbb{E}\left[\mathbf{b}(\widehat{\bt}, n)\right] - \mathbf{b}(\bt_0,n)$ respectively. Firstly, for the order of $\mathbb{E}[\mathbf{u}]$, we have from \eqref{eq:appC:6} that 
\begin{equation*}
\begin{aligned}
    \left\|\mathbb{E}[\mathbf{u}]\right\|_\infty &\leq \underset{i=1,\ldots,p}{\max} \sum_{j=1}^p \sum_{k=1}^p \mathbb{E}\left|g_{ijk}(\widehat{\bt})\Delta_j \Delta_k\right| \leq \underset{i=1,\ldots,p}{\max} \sum_{j=1}^p \sum_{k=1}^p \mathbb{E}\left\{\underset{\bt \in \bT}{\max} |g_{ijk}(\bt)|\cdot \left|\Delta_j \Delta_k\right|\right\} \\
    &\leq \frac{1}{2} \underset{i=1,\ldots,p}{\max} \sum_{j=1}^p \sum_{k=1}^p \left\{\underset{\bt \in \bT}{\max} |G_{ijk}(\bt)|\right\} \mathbb{E}|\Delta_j \Delta_k|.
\end{aligned}
\end{equation*}
Recall from \eqref{eq:appC:3} that $\Delta_j = \mathcal{O}_{\rm p}(n^{-\alpha})$ for all $j=1,\ldots,p$, hence $\Delta_j\Delta_k = \mathcal{O}_{\rm p}(n^{-2\alpha})$ and thus $\mathbb{E}|\Delta_j \Delta_k| = \mathcal{O}(n^{-2\alpha})$ for all $j,k=1,\ldots,p$. Therefore, under Assumption \ref{assum:D}, we have
\begin{equation}
\begin{aligned}
\label{eq:appC:9}
    \left\|\mathbb{E}[\mathbf{u}]\right\|_\infty 
    &\leq \frac{1}{2} \underset{i=1,\ldots,p}{\max} \sum_{j=1}^p \sum_{k=1}^p \left\{\underset{\bt \in \bT}{\max} |G_{ijk}(\bt)|\right\} \mathbb{E}|\Delta_j \Delta_k| = \mathcal{O}(n^{-2\alpha}). 
\end{aligned}
\end{equation}
On the other hand, for the order of $\mathbb{E}\left[\mathbf{b}(\widehat{\bt}, n)\right] - \mathbf{b}(\bt_0,n)$, by mean value theorem and under Assumption \ref{assum:C}, we have
\begin{equation*}
\begin{aligned}
    \left\|\mathbb{E}\left[\mathbf{b}(\widehat{\bt}, n)\right] - \mathbf{b}(\bt_0,n)\right\|_\infty &= \left\|\mathbb{E}\left[\mathbf{B}(\bt^*,n)(\widehat{\bt} - \bt_0)\right]\right\|_\infty \leq \mathbb{E}\left[\left\|\mathbf{B}(\bt^*,n)(\widehat{\bt} - \bt_0)\right\|_\infty\right] \\
    &\leq \mathbb{E}\left[\left\|\mathbf{B}(\bt^*,n)\right\|_\infty \left\|\widehat{\bt} - \bt_0\right\|_\infty \right],
\end{aligned}    
\end{equation*}
where $\bt^{*}$ is on the line connecting $\widehat{\bt}$ and $\bt_0$. By Assumption \ref{assum:C}, we have $\left\|\mathbf{B}(\bt^*,n)\right\|_\infty = \mathcal{O}_{\rm p}(n^{-\beta})$. Moreover, recall from \eqref{eq:appC:3} that $\left\|\widehat{\bt} - \bt_0\right\|_\infty = \mathcal{O}_{\rm p}(n^{-\alpha})$, hence we obtain $\left\|\mathbf{B}(\bt^*,n)\right\|_\infty \left\|\widehat{\bt} - \bt_0\right\|_\infty = \mathcal{O}_{\rm p}(n^{-(\alpha+\beta)})$. Thus,
\begin{equation}
\label{eq:appC:10}
    \left\|\mathbb{E}\left[\mathbf{b}(\widehat{\bt}, n)\right] - \mathbf{b}(\bt_0,n)\right\|_\infty \leq \mathbb{E}\left[\left\|\mathbf{B}(\bt^*,n)\right\|_\infty \left\|\widehat{\bt} - \bt_0\right\|_\infty \right] =  \mathcal{O}(n^{-(\alpha+\beta)}).
\end{equation}
Therefore, from \eqref{eq:appC:8}, \eqref{eq:appC:9} and \eqref{eq:appC:10}, under Assumption \ref{assum:D} we have 
\begin{equation*}
\begin{aligned}
    \left\|\mathbb{E}[\widehat{\bt}] - \bt_0\right\|_\infty &\leq \left\|\left\{\mathbf{I} + \mathbf{A}(\bt_0)\right\}^{-1}\right\|_\infty \left\{\left\|\mathbb{E}[\mathbf{u}]\right\|_\infty + \left\|\mathbb{E}\left[\mathbf{b}(\widehat{\bt}, n)\right] - \mathbf{b}(\bt_0,n)\right\|_\infty\right\} \\
    &= \mathcal{O}(n^{-2\alpha}) + \mathcal{O}(n^{-(\alpha+\beta)}) = \mathcal{O}(n^{-2\alpha}),
\end{aligned}    
\end{equation*}
which leads to 
\begin{equation*}
    \left\|\mathbb{E}[\widehat{\bt}] - \bt_0\right\|_2 = \mathcal{O}(p^{1/2}n^{-2\alpha}).
\end{equation*}
$\hfill\square$

\subsection{A Discussion on Assumptions \ref{assum:D} and \ref{assum:D1}}

In the second part of Assumption \ref{assum:D}, the condition $\left\|\mathbf{A}(\bt)\right\|_\infty < 1$ implies that $\mathbf{a}(\bt)$ is a contraction map with respect to the norm $\|\cdot\|_\infty$, however, not necessarily with respect to the norm $\|\cdot\|_2$ as required in the first part of Assumption \ref{assum:D}. In addition, by Lemma~\ref{lem:1} in Appendix~\ref{app:proof:JINI:bias:finite}, $\left\|\mathbf{A}(\bt)\right\|_\infty < 1$ implies that  $\|\left\{\mathbf{I}+\mathbf{A}(\bt)\right\}^{-1}\|_\infty = \mathcal{O}(1)$. 

Similarly, the condition $0 < \|\mathbf{A}\|_2 < 1$ in Assumption \ref{assum:D1} guarantees that $\mathbf{a}(\bt)$ is a contraction map with respect to the norm $\|\cdot\|_2$, and that $\|(\mathbf{I}+\mathbf{A})^{-1}\|_2 = \mathcal{O}(1)$. However, since $p$ is allowed to diverge, it does not necessarily imply that $\|(\mathbf{I}+\mathbf{A})^{-1}\|_\infty = \mathcal{O}(1)$ as required in Assumption~\ref{assum:D1}.

Moreover, Assumption \ref{assum:D} can be further relaxed, for example, by replacing
$$\underset{i=1,\ldots,p}{\max} \sum_{j=1}^p \sum_{k=1}^p |G_{ijk}(\bt)| = \mathcal{O}(1),$$
with $\underset{i=1,\ldots,p}{\max} \|\mathbf{G}_i(\bt)\|_\infty = \underset{i=1,\ldots,p}{\max} \left\{\underset{j=1,\ldots,p}{\max} \sum_{k=1}^p \left|G_{ijk}(\bt)\right|\right\} = \mathcal{O}(1)$. In this case, the second result in Theorem \ref{thm:bias:asymp} will change to $\|\mathbb{E}[\widehat{\bt}] - \bt_0\|_2 = \mathcal{O}\left(p^{1/2} \max\left\{pn^{-2\alpha}, n^{-(\alpha+\beta)}\right\}\right).$ This will alter the restriction on $p/n$ in terms of $\alpha$. Indeed, instead of $p=o(n^{2\alpha})$ as stated in Assumption~\ref{assum:B}, we will need stronger condition that $p=o(n^{4/3\alpha})$. This suggests that there is a trade-off between the magnitude of the asymptotic bias and the number of parameters~$p$.

\section{Proof of Theorem \ref{thm:consist:asympnorm}}
\label{app:proof:consist:asympnorm}

\subsection{Proof of Consistency}
\label{app:proof:consist:asympnorm:1}

First, we recall that $\bpi(\bt, n) = \mathbb{E}\left[\widehat{\bpi}(\bt, n)\right] = \bt + \mathbf{a}(\bt) + \mathbf{b}(\bt, n)$ and also that $\bpi(\bt)  = \bt + \mathbf{a}(\bt)$. Next, we define the functions $Q(\bt)$, $Q(\bt,n)$ and $\widehat{Q}(\bt,n)$ as 
\begin{align*}
    & Q(\bt) \vcentcolon = \left\|\bpi(\bt_0) - \bpi(\bt)\right\|_2, \\
    & Q(\bt, n) \vcentcolon = \left\|\bpi(\bt_0, n) - \bpi(\bt, n)\right\|_2, \\
    & \widehat{Q}(\bt, n) \vcentcolon = \left\|\widehat{\bpi}(\bt_0, n) - \frac{1}{H}\sum_{h=1}^H \widehat{\bpi}_h (\bt, n) \right\|_2.
\end{align*}
Then this proof is directly obtained by verifying the conditions of Theorem 2.1 of \cite{newey1994large} on the functions $Q(\bt)$ and $\widehat{Q}(\bt,n)$. Reformulating the requirements of this theorem to our setting, we want to show that (i) $\bT$ is compact, (ii) $Q(\bt)$ is continuous, (iii) $Q(\bt)$ is uniquely minimized at $\bt_0$, and (iv) $\widehat{Q}(\bt,n)$ converges uniformly in probability to $Q(\bt)$.

On one hand, Assumption \ref{assum:A} ensures the compactness of $\bT$. On the other hand, Assumption~\ref{assum:D} ensures that $Q(\bt)$ is continuous and uniquely minimized at $\bt_0$, since $\bpi(\bt)$ is continuous and injective. What remains to show is that $\widehat{Q}(\bt,n)$ converges uniformly in probability to $Q(\bt)$, which is equivalent to show: for all $\epsilon >0$ and for all $\delta > 0$, there exists a sample size $N \in \mathbb{N}$ such that for all $n \geq N$, we have 
\begin{equation*}
    P\left\{\underset{\bt \in \bT}{\sup} \left|\widehat{Q}(\bt, n) - Q(\bt)\right| \geq \epsilon\right\} \leq \delta.
\end{equation*}

Fix $\epsilon > 0$ and $\delta > 0$. Using the above definitions, we have
\begin{equation}
\label{eq:appD:1}
    \underset{\bt \in \bT}{\sup} \left|\widehat{Q}(\bt, n) - Q(\bt)\right| \leq \underset{\bt \in \bT}{\sup} \left\{\left|\widehat{Q}(\bt, n) - Q(\bt, n)\right| + \left|Q(\bt, n) - Q(\bt)\right| \right\}.
\end{equation}
Considering the first term on the right hand side of \eqref{eq:appD:1}, we have
\begin{equation}
\begin{aligned}
\label{eq:appD:2}
    \left|\widehat{Q}(\bt, n) - Q(\bt, n)\right| &\leq \left\|\widehat{\bpi}(\bt_0, n) - \frac{1}{H}\sum_{h=1}^H \widehat{\bpi}_h (\bt, n) - \bpi(\bt_0, n) + \bpi(\bt, n)\right\|_2 \\
    &\leq \left\|\widehat{\bpi}(\bt_0, n) - \bpi(\bt_0, n)\right\|_2 + \left\|\bpi(\bt, n) - \frac{1}{H}\sum_{h=1}^H \widehat{\bpi}_h (\bt, n)\right\|_2 \\
    &= \left\|\mathbf{v}(\bt_0,n)\right\|_2 + \left\|\frac{1}{H} \sum_{h=1}^H \mathbf{v}_h(\bt, n)\right\|_2 \\
    &= \mathcal{O}_{\rm p}(p^{1/2} n^{-\alpha}) + \mathcal{O}_{\rm p}(H^{-1/2} p^{1/2} n^{-\alpha}) \\
    &= \mathcal{O}_{\rm p}(p^{1/2} n^{-\alpha}).
\end{aligned}
\end{equation}
Similarly, the second term of \eqref{eq:appD:1} can be computed as
\begin{equation}
\begin{aligned}
\label{eq:appD:3}
    \left|Q(\bt, n) - Q(\bt)\right| &\leq \left\| \bpi(\bt_0, n) - \bpi(\bt, n) - \bpi(\bt_0) + \bpi(\bt) \right\|_2 \\
    &\leq \left\| \bpi(\bt_0, n) - \bpi(\bt_0) \right\|_2 + \left\| \bpi(\bt, n) -  \bpi(\bt) \right\|_2 \\
    &= \left\| \mathbf{b}(\bt_0,n) \right\|_2 + \left\| \mathbf{b}(\bt,n) \right\|_2 \\
    &= \mathcal{O}(p^{1/2} n^{-\beta}).
\end{aligned}
\end{equation}
Combining the results in \eqref{eq:appD:2} and \eqref{eq:appD:3} into \eqref{eq:appD:1}, we obtain
\begin{equation*}
    \underset{\bt \in \bT}{\sup} \left|\widehat{Q}(\bt, n) - Q(\bt)\right| = \mathcal{O}_{\rm p}(p^{1/2}n^{-\alpha}) + \mathcal{O}(p^{1/2}n^{-\beta}),
\end{equation*}
which, by Assumptions \ref{assum:B} and \ref{assum:C}, completes the proof. 
\hfill $\square$ 

\subsection{Proof of Asymptotic Normality}
\label{app:proof:consist:asympnorm:2}

Recall from \eqref{eq:appC:1} that 
\begin{equation*}
    \widehat{\bt} - \bt_0 = \mathbf{a}(\bt_0) - \mathbf{a}(\widehat{\bt}) + \mathbf{b}(\bt_0, n) + \mathbf{v}(\bt_0,n) -  \mathbf{b}(\widehat{\bt},n) - \frac{1}{H}\sum_{h=1}^H \mathbf{v}_h(\widehat{\bt},n),
\end{equation*}
and from \eqref{eq:appC:4}, \eqref{eq:appC:5} that 
\begin{equation*}
    \mathbf{a}(\widehat{\bt}) - \mathbf{a}(\bt_0) = \mathbf{A}(\bt_0)(\widehat{\bt} - \bt_0) + \mathbf{u},
\end{equation*}
with $\mathbf{u} = [u_1 \ldots u_p]\tt$ where
\begin{equation*}
    u_i \vcentcolon = \sum_{j=1}^p \sum_{k=1}^p g_{ijk}(\widehat{\bt})\Delta_j\Delta_k. 
\end{equation*}
So we have
\begin{equation*}
    \widehat{\bt} - \bt_0 = -\mathbf{A}(\bt_0)(\widehat{\bt} - \bt_0) - \mathbf{u} + \mathbf{b}(\bt_0, n) + \mathbf{v}(\bt_0,n) -  \mathbf{b}(\widehat{\bt},n) - \frac{1}{H}\sum_{h=1}^H \mathbf{v}_h(\widehat{\bt},n).
\end{equation*}
By Assumption~\ref{assum:D} and Lemma~\ref{lem:1}, we know $\mathbf{I}+\mathbf{A}(
\bt_0)$ is nonsingular. For simplicity, we write $\bm{\mathcal{A}} \vcentcolon = \left\{\mathbf{I}+\mathbf{A}(
\bt_0)\right\}^{-1}$. Then, we can also write
\begin{equation*}
    \widehat{\bt} - \bt_0 = \bm{\mathcal{A}}\left\{\mathbf{b}(\bt_0, n) + \mathbf{v}(\bt_0,n)\right\} - \bm{\mathcal{A}} \left\{\mathbf{b}(\widehat{\bt},n) + \frac{1}{H}\sum_{h=1}^H \mathbf{v}_h(\widehat{\bt},n)\right\} - \bm{\mathcal{A}} \mathbf{u},
\end{equation*}
and therefore,
\begin{equation}
\begin{aligned}
\label{eq:appD:4}
        \sqrt{n}\mathbf{s}\tt (\widehat{\bt} - \bt_0) &=  \sqrt{n}\mathbf{s}\tt\bm{\mathcal{A}}\left\{\mathbf{b}(\bt_0, n) + \mathbf{v}(\bt_0,n)\right\}
        \\
        &  \hspace{0.5cm}- \sqrt{n}\mathbf{s}\tt \bm{\mathcal{A}} \left\{\mathbf{b}(\widehat{\bt},n) + \frac{1}{H}\sum_{h=1}^H \mathbf{v}_h(\widehat{\bt},n)\right\} 
          - \sqrt{n}\mathbf{s}\tt \bm{\mathcal{A}} \mathbf{u}.
\end{aligned}
\end{equation}
Next we want to evaluate the three terms on the right hand side of \eqref{eq:appD:4} respectively. By Assumption \ref{assum:E}, we have 
\begin{equation*}
    \sqrt{n} \mathbf{s}\tt \mathbf{\Sigma}(\bt_0)^{-1/2} \left\{\widehat{\bpi}(\bt_0,n) - \bpi(\bt_0)\right\} = \sqrt{n} \mathbf{s}\tt \mathbf{\Sigma}(\bt_0)^{-1/2} \left\{\mathbf{b}(\bt_0,n) + \mathbf{v}(\bt_0,n)\right\} \overset{d}{\to} \mathcal{N}(0,1),
\end{equation*}
which implies that 
\begin{equation*}
    \sqrt{n}\mathbf{s}\tt\bm{\mathcal{A}} \left\{\mathbf{b}(\bt_0,n) + \mathbf{v}(\bt_0,n)\right\} \overset{d}{\to} \mathcal{N}\left(0, W_s(\bt_0)\right),
\end{equation*}
where $W_s(\bt_0) \vcentcolon = \mathbf{s}\tt \bm{\mathcal{A}} \mathbf{\Sigma}(\bt_0)\bm{\mathcal{A}}\tt \mathbf{s}$. So we have
\begin{equation}
\label{eq:appD:5}
    \sqrt{n}\mathbf{s}\tt\bm{\mathcal{A}} \left\{\mathbf{b}(\bt_0,n) + \mathbf{v}(\bt_0,n)\right\} \overset{d}{=} W_s(\bt_0)^{1/2}z_0 + o_{\rm p}(1),
\end{equation}
where $z_0 \sim \mathcal{N}(0,1)$. Similarly, for any $h=1,\ldots,H$, we have
\begin{equation}
\label{eq:appD:6}
    \sqrt{n}\mathbf{s}\tt\bm{\mathcal{A}} \left\{\mathbf{b}(\widehat{\bt},n) + \mathbf{v}_h(\widehat{\bt},n)\right\} \overset{d}{=} W_s(\widehat{\bt})^{1/2}z_h + o_{\rm p}(1),
\end{equation}
where $z_h$ are independent $\mathcal{N}(0,1)$ random variables. Now we want to further evaluate $W_s(\widehat{\bt})^{1/2}z_h$. By Taylor's theorem, we have 
\begin{equation}
\label{eq:appD:7}
    W_s(\widehat{\bt})^{1/2} = W_s(\bt_0)^{1/2} + \frac{\partial W_s(\bt_0)^{1/2}}{\partial \bt\tt} (\widehat{\bt} - \bt_0) + o_{\rm p}(1).
\end{equation}
Moreover, 
\begin{equation*}
\begin{aligned}
    \frac{\partial W_s(\bt_0)^{1/2}}{\partial \bt\tt} &= \frac{1}{2} W_s(\bt_0)^{-1/2} \frac{\partial W_s(\bt_0)}{\partial \bt\tt} = \frac{1}{2} W_s(\bt_0)^{-1/2} \frac{\partial \left\{\mathbf{s}\tt \bm{\mathcal{A}} \mathbf{\Sigma}(\bt_0)\bm{\mathcal{A}}\tt \mathbf{s}\right\}}{\partial \bt\tt} \\
    &= \frac{1}{2} W_s(\bt_0)^{-1/2} \tr \left[\frac{\partial \left\{\mathbf{s}\tt \bm{\mathcal{A}} \mathbf{\Sigma}(\bt_0)\bm{\mathcal{A}}\tt \mathbf{s}\right\}}{\partial \left\{\bm{\mathcal{A}} \mathbf{\Sigma}(\bt)\bm{\mathcal{A}}\tt\right\}} \cdot \frac{\partial \left\{\bm{\mathcal{A}} \mathbf{\Sigma}(\bt_0)\bm{\mathcal{A}}\tt\right\}}{\partial \bt\tt}\right]\\
    &= \frac{1}{2} W_s(\bt_0)^{-1/2} \left[\tr \left\{\mathbf{s}\mathbf{s}\tt \bm{\mathcal{A}} \frac{\partial \mathbf{\Sigma}(\bt_0)}{\partial \theta_i} \bm{\mathcal{A}}\tt \right\}\right]_{i=1,\ldots,p} \\
    &\leq \frac{1}{2} W_s(\bt_0)^{-1/2} \left[\lambda_{\max}\left\{ \bm{\mathcal{A}} \frac{\partial \mathbf{\Sigma}(\bt_0)}{\partial \theta_i} \bm{\mathcal{A}}\tt \right\}\right]_{i=1,\ldots,p} \\
    &= \frac{1}{2} W_s(\bt_0)^{-1/2} \mathbf{w}(\bt_0),
\end{aligned}    
\end{equation*}
where the last inequality comes from the fact that for any $i=1,\ldots,p$, we have
\begin{equation*}
\begin{aligned}
    \tr \left\{\mathbf{s}\mathbf{s}\tt \bm{\mathcal{A}} \frac{\partial \mathbf{\Sigma}(\bt_0)}{\partial \theta_i} \bm{\mathcal{A}}\tt \right\} &\leq \tr\left(\mathbf{s}\mathbf{s}\tt\right) \lambda_{\max}\left\{\bm{\mathcal{A}} \frac{\partial \mathbf{\Sigma}(\bt_0)}{\partial \theta_i} \bm{\mathcal{A}}\tt \right\} = \|\mathbf{s}\|_2^2 \cdot \lambda_{\max}\left\{\bm{\mathcal{A}} \frac{\partial \mathbf{\Sigma}(\bt_0)}{\partial \theta_i} \bm{\mathcal{A}}\tt \right\} \\
    &= \lambda_{\max}\left\{\bm{\mathcal{A}} \frac{\partial \mathbf{\Sigma}(\bt_0)}{\partial \theta_i} \bm{\mathcal{A}}\tt \right\},
\end{aligned}
\end{equation*}
based on the inequality from Theorem 1 in \cite{fang1994inequalities} since $\mathbf{s}\mathbf{s}\tt$ is positive semi-definite and $\bm{\mathcal{A}}\left\{\partial \mathbf{\Sigma}(\bt_0) / \partial \theta_i\right\} \bm{\mathcal{A}}\tt$ is symmetric. Therefore, using Assumption \ref{assum:E} and the consistency of $\widehat{\bt}$, we have
\begin{equation*}
\begin{aligned}
    \left| \frac{\partial W_s(\bt_0)^{1/2}}{\partial \bt\tt} (\widehat{\bt} - \bt_0) \right| &\leq \left\|\frac{\partial W_s(\bt_0)^{1/2}}{\partial \bt\tt}\right\|_2 \left\|\widehat{\bt} - \bt_0\right\|_2 \\
    &\leq \frac{1}{2}W_s(\bt_0)^{-1/2} \left\|\mathbf{w}(\bt_0)\right\|_2 \left\|\widehat{\bt} - \bt_0\right\|_2 \\
    &= \mathcal{O}(1) o_{\rm p}(1) = o_{\rm p}(1).
\end{aligned}
\end{equation*}
So \eqref{eq:appD:7} is equivalent to
\begin{equation*}
    W_s(\widehat{\bt})^{1/2} = W_s(\bt_0)^{1/2} + o_{\rm p}(1),
\end{equation*}
and thus, \eqref{eq:appD:6} can be re-evaluated as
\begin{equation}
\label{eq:appD:8}
    \sqrt{n}\mathbf{s}\tt\bm{\mathcal{A}} \left\{\mathbf{b}(\widehat{\bt},n) + \mathbf{v}_h(\widehat{\bt},n)\right\} \overset{d}{=} W_s(\bt_0)^{1/2}z_h + o_{\rm p}(1).
\end{equation}
Lastly, we want to evaluate the third term on the right hand side of \eqref{eq:appD:4}, $\sqrt{n}\mathbf{s}\tt \bm{\mathcal{A}} \mathbf{u}$. Assumption \ref{assum:E} implies that $\alpha=1/2$. Recall that from \eqref{eq:appC:3}, we have $\Delta_j = \mathcal{O}_{\rm p}(n^{-\alpha}) = \mathcal{O}_{\rm p}(n^{-1/2})$ for all $j=1,\ldots,p$, and that from \eqref{eq:appC:6}, we have $|g_{ijk}(\bt)| \leq 1/2 \underset{\bt\in\bT}{\max}|G_{ijk}(\bt)|$ for all $\bt \in \bT$. So under Assumption~\ref{assum:D}, we have
\begin{equation*}
\begin{aligned}
    \|\mathbf{u}\|_\infty &\leq \underset{i=1,\ldots,p}{\max} \sum_{j=1}^p \sum_{k=1}^p \left|g_{ijk}(\widehat{\bt})\right| \cdot \left|\Delta_j\Delta_k \right| \\
    &\leq \frac{1}{2} \underset{i=1,\ldots,p}{\max} \sum_{j=1}^p \sum_{k=1}^p \left\{\underset{\bt\in\bT}{\max} |G_{ijk}(\bt)| \right\} \cdot \left|\Delta_j\Delta_k \right| \\
    &=\mathcal{O}_{\rm p}(n^{-2\alpha}) = \mathcal{O}_{\rm p}(n^{-1}). 
\end{aligned}
\end{equation*}
Moreover, $\|\mathbf{s}\|_1 \leq p^{1/2} \|\mathbf{s}\|_2 = p^{1/2}$. Under Assumption \ref{assum:D}, $\left\|\mathbf{A}(\bt_0)\right\|_\infty <1$, so by Lemma~\ref{lem:1} in Appendix \ref{app:proof:JINI:bias:finite}, we have $\left\|\bm{\mathcal{A}}\right\|_\infty = \mathcal{O}(1)$. Thus, we obtain 
\begin{equation}
\label{eq:appD:9}
    \left|\sqrt{n}\mathbf{s}\tt \bm{\mathcal{A}} \mathbf{u}\right| \leq \sqrt{n} \|\mathbf{s}\|_1 \left\|\bm{\mathcal{A}}\mathbf{u}\right\|_\infty \leq \sqrt{n} \|\mathbf{s}\|_1 \left\|\bm{\mathcal{A}}\right\|_\infty \left\|\mathbf{u}\right\|_\infty = \mathcal{O}_{\rm p}(p^{1/2}n^{-1/2}) = o_{\rm p}(1),
\end{equation}
as by Assumption \ref{assum:B} we have $p=o(n^{2\alpha}) = o(n)$. Using the results of \eqref{eq:appD:5}, \eqref{eq:appD:8} and \eqref{eq:appD:9} in \eqref{eq:appD:4}, we have
\begin{equation*}
    \sqrt{n}\mathbf{s}\tt(\widehat{\bt}-\bt_0) \overset{d}{=} W_s(\bt_0) \left(z_0 + \frac{1}{H}\sum_{h=1}^H z_h\right) + o_{\rm p}(1) \overset{d}{=} \left(1+\frac{1}{H}\right)^{1/2} W_s(\bt_0)^{1/2} z + o_{\rm p}(1),
\end{equation*}
where $z \sim \mathcal{N}(0,1)$, and therefore,
\begin{equation*}
    \left(1+\frac{1}{H}\right)^{-1/2} \sqrt{n} \mathbf{s}\tt \left\{\bm{\mathcal{A}} \mathbf{\Sigma}(\bt_0) \bm{\mathcal{A}}\tt\right\}^{-1/2} (\widehat{\bt} - \bt_0) \overset{d}{\to} \mathcal{N}(0,1).
\end{equation*}
$\hfill\square$

\section{Proof of Theorem \ref{thm:ib}}
\label{app:proof:ib}
First, recall in \eqref{eqn:def:function:T} that 
\begin{equation}
\begin{aligned}
\label{eq:appE:1}
    \bm{T}(\bt, n) &= \bt + \widehat{\bpi}(\bt_0, n) - \frac{1}{H} \sum_{h=1}^H \widehat{\bpi}_h(\bt, n) \\
    &= \bt + \widehat{\bpi}(\bt_0, n) - \left\{\bt + \mathbf{a}(\bt) + \mathbf{b}(\bt,n) + \frac{1}{H}\sum_{h=1}^H \mathbf{v}_h(\bt, n) \right\} \\
    &= \widehat{\bpi}(\bt_0, n) - \left\{ \mathbf{a}(\bt) + \mathbf{b}(\bt,n) + \frac{1}{H}\sum_{h=1}^H \mathbf{v}_h(\bt, n) \right\}.
\end{aligned}
\end{equation}
This function is defined on $\bT$ with target space $\real^p$ and is only a function with respect to $\bt$ and $n$, as the seeds used to compute $\widehat{\bpi}_h(\bt, n)$ are fixed (see Appendix~\ref{app:imple:ib} for more details). In this proof, we first show that $\bm{T}(\bt,n)$ is a contraction map for sufficiently large $n$. It enables us to apply Kirszbraun theorem (see \citealp{federer2014geometric}) and Banach fixed-point theorem to show that $\bm{T}(\bt,n)$ admits a unique fixed point. 

Let us consider $\bt_1, \bt_2 \in \bT$. With \eqref{eq:appE:1}, we obtain
\begin{equation*}
    \bm{T}(\bt_2, n) - \bm{T}(\bt_1, n)
    =  \mathbf{a}(\bt_1) - \mathbf{a}(\bt_2) + \mathbf{b}(\bt_1, n) - \mathbf{b}(\bt_2, n) 
     + \frac{1}{H} \sum_{h=1}^H \left\{\mathbf{v}_h(\bt_1,n) - \mathbf{v}_h(\bt_2,n) \right\}.
\end{equation*}
Therefore, we have 
\begin{equation}
\begin{aligned}
\label{eq:appE:2}
    \|\bm{T}(\bt_2, n) - \bm{T}(\bt_1, n)\|_2 
    &\leq \left\| \mathbf{a}(\bt_2) - \mathbf{a}(\bt_1)\right\|_2 + \left\| \mathbf{b}(\bt_2, n) - \mathbf{b}(\bt_1, n)\right\|_2 
    \\ 
    & \hspace{0.5cm} + \frac{1}{H}\sum_{h=1}^H \left\|\mathbf{v}_h(\bt_2, n) - \mathbf{v}_h(\bt_1, n)\right\|_2.
\end{aligned}
\end{equation}
Considering the first term on the right hand side of \eqref{eq:appE:2}, by Assumption \ref{assum:D} we have 
\begin{equation}
\label{eq:appE:3}
    \|\mathbf{a}(\bt_2) - \mathbf{a}(\bt_1)\|_2 \leq M \| \bt_2 - \bt_1\|_2,
\end{equation}
where $0 < M < 1$. For the second term on the right hand side of \eqref{eq:appE:2}, by Assumption~\ref{assum:C} and using the mean value theorem, we have
\begin{equation}
\label{eq:appE:4}
    \|\mathbf{b}(\bt_2,n) - \mathbf{b}(\bt_1,n)\|_2 \leq \|\mathbf{B}(\bt^*, n)\|_2 \| \bt_2 - \bt_1\|_2 \leq \|\mathbf{B}(\bt^*, n)\|_F \| \bt_2 - \bt_1\|_2,
\end{equation}
where $\bt^*$ is on the line connecting $\bt_1$ and $\bt_2$ and $\|\cdot\|_F$ denotes the Frobenius norm. Moreover, by Assumption~\ref{assum:C} we have
\begin{equation}
\begin{aligned}
\label{eq:appE:5}
    \|\mathbf{B}(\bt^*, n)\|_F &= \sqrt{\sum_{i=1}^p \sum_{j=1}^p B(\bt^*, n)_{ij}^2} \leq \sqrt{p\cdot \underset{i=1,\ldots,p}{\max} \sum_{j=1}^p B(\bt^*, n)_{ij}^2} 
    \\
    &= p^{1/2} \underset{i=1,\ldots,p}{\max} \sqrt{\sum_{j=1}^p B(\bt^*, n)_{ij}^2} %
    \leq p^{1/2} \underset{i=1,\ldots,p}{\max} \sum_{j=1}^p |B(\bt^*, n)_{ij}| 
    \\    
    &= p^{1/2} \|\mathbf{B}(\bt^*, n)\|_\infty = \mathcal{O}(p^{1/2}n^{-\beta}).
\end{aligned}   
\end{equation}
For the third term on the right hand side of \eqref{eq:appE:2}, by Assumption \ref{assum:B} we have 
\begin{equation}
\label{eq:appE:6}
    \frac{1}{H} \sum_{h=1}^H \|\mathbf{v}_h(\bt_2, n) - \mathbf{v}_h(\bt_1,n)\|_2 = \mathcal{O}_{\rm p}(H^{-1/2}p^{1/2}n^{-\alpha}).
\end{equation}
So plugging in the results of \eqref{eq:appE:3}-\eqref{eq:appE:6} to \eqref{eq:appE:2}, we obtain
\begin{equation*}
\begin{aligned}
    \|\bm{T}(\bt_2, n) - \bm{T}(\bt_1, n)\|_2 &\leq M \left\|\bt_2 - \bt_1\right\|_2 + \|\mathbf{B}(\bt^*, n)\|_F \left\|\bt_2 - \bt_1\right\|_2 
    \\ 
    & \hspace{0.5cm} + \frac{1}{H}\sum_{h=1}^H \left\|\mathbf{v}_h(\bt_2, n) - \mathbf{v}_h(\bt_1, n)\right\|_2 
    \\
    &= \left\{M + \mathcal{O}(p^{1/2}n^{-\beta})\right\} \left\|\bt_2 - \bt_1\right\|_2 + \mathcal{O}_{\rm p}(H^{-1/2}p^{1/2}n^{-\alpha}).
\end{aligned}
\end{equation*}
Thus, for sufficiently large $n$ and $H$, there exists some $0 \leq \epsilon <1$ such that for all $\bt_1, \bt_2 \in \bT$, 
\begin{equation*}
    \|\bm{T}(\bt_2, n) - \bm{T}(\bt_1, n)\|_2 \leq \epsilon \|\bt_2 - \bt_1\|_2.
\end{equation*}
Using Kirszbraun theorem, we can extend $\bm{T}(\bt, n)$ to a contraction map from $\real^p$ to itself. Therefore, applying Banach fixed-point theorem, there exists a unique fixed point $\widehat{\bt}\in\real^p$, and $\widehat{\bt}\in\bT$ for sufficiently large $n$ and $H$.

It remains to show the convergence rate of the IB algorithm. First, we can observe that 
\begin{equation*}
    \left\|\bt^{(k)} - \widehat{\bt} \right\|_2 \leq \frac{\epsilon^k}{1-\epsilon} \left\|\bt^{(1)} - \bt^{(0)}\right\|_2,
\end{equation*}
and
\begin{equation*}
    \left\| \bt^{(1)} - \bt^{(0)} \right\|_2 \leq \left\| \bt^{(1)} - \widehat{\bt}\right\|_2 + \left\| \bt^{(0)} - \widehat{\bt} \right\|_2 \leq (\epsilon+1) \left\| \bt^{(0)} - \widehat{\bt}\right\|_2 = \mathcal{O}_{\rm p}\left(p^{1/2}\right).
\end{equation*}
Since $\bt^{(0)} = \widehat{\bpi}(\bt_0,n) \in \bT$, we have, 
under Assumption \ref{assum:A} and for all $k \in \mathbb{N}^{+}$, that
\begin{equation*}
    \left\|\bt^{(k)} - \widehat{\bt} \right\|_2 \leq \frac{\epsilon^k}{1-\epsilon} \left\|\bt^{(1)} - \bt^{(0)}\right\|_2 = \mathcal{O}_{\rm p}\left(p^{1/2}\epsilon^k\right) = \mathcal{O}_{\rm p}\left(p^{1/2}\exp(-ck)\right),
\end{equation*}
where $c$ is some positive real number. 
\hfill $\square$ 

\newpage
\section{Implementation of the IB Algorithm}
\label{app:imple:ib}

The IB algorithm can be implemented using the steps presented in Table \ref{tab:ib:algo}. We emphasize that it is necessary to use the same random number generator at each iteration in order to avoid introducing extra variation besides the current value of the parameter $\bt^{(k)}$, the only quantity that varies between iterations. For example, this can be achieved by generating $H$ samples $\mathbf{u}_h,h=1,\ldots,H$, of size $n$ from the uniform distribution and the simulated samples are obtained using $ F^{-1}_{\bt^{(k)}}\left(\mathbf{u}_h\right)$ at Step~3 in Table~\ref{tab:ib:algo}.
\begin{table}[!hbt]
\centering
\caption{IB algorithm}
\begin{tabular}{ll}
\toprule
Step 1 & Compute the initial estimator $\widehat{\bpi}(\bt_0,n)$ on the observed sample, say $\mathbf{X}(\bt_0)$. \\
Step 2 & Set $k = 0$, and let $\bt^{(k)} = \widehat{\bpi}(\bt_0,n)$. \\
Step 3 & Simulate $H$ samples from the model $F_{\bt^{(k)}}$ with fixed ``seed'' values, which\\
& are denoted by $\mathbf{X}_h^*\left(\bt^{(k)}\right)$ with $h=1,\ldots,H$. \\
Step 4 & Compute the value of the initial estimator on each of the $H$ simulated\\
& samples $\mathbf{X}_h^*\left(\bt^{(k)}\right)$ in Step 3, which are denoted by $\widehat{\bpi}_h\left(\bt^{(k)}, n\right)$ with\\ & $h=1,\ldots,H$. \\
Step 5 & Compute $\bt^{(k+1)}$ using $\bt^{(k+1)} = \bt^{(k)} + \widehat{\bpi}(\bt_0, n) - 1/H \sum_{h=1}^H \widehat{\bpi}_h\left(\bt^{(k)}, n\right)$. \\
Step 6 & Set $k = k+1$. \\
Step 7 & Repeat Steps 3 to 6 until either $\|\bt^{(k+1)} - \bt^{(k)}\|_2 < \epsilon$, where $\epsilon$ denotes a \\  
& tolerance level, or a pre-defined maximum number of iterations is reached. \\
\bottomrule
\end{tabular}
\label{tab:ib:algo}
\end{table}
%


\clearpage
\appendix
{\centerline{\Large \sc Supplementary Materials}}

\titleformat{\section}{\large\scshape}{\Alph{section}.}{1em}{}

\section{Logistic Regression}
\label{supp:logistic}
Logistic regression, one of the most popular models in the class of Generalized Linear Models (GLM) \citep[see e.g.][]{NeWe:72,McCuNe:89}, can be used to model binary outcomes conditionally on a set of covariates. The MLE is a suitable estimator, in terms of finite sample properties, when $n$ is relatively large compared to $p$. More precisely, as is proposed for example in medical studies, the Events Per Variable (EPV) quantity, defined as the number of occurrences of the least frequent outcome over the number of covariates, should be large enough  to avoid finite sample bias of the MLE. The EPV can often be used to choose the maximal number of covariates that is recommended for a logistic regression (see e.g. \citealp{AuSt:17} and the references therein). In pratice, a rule of thumb is typically at least 10 EPV (see e.g. \citealp{PeCoKeHoFe:96}). However, 
with an EPV smaller than 10, a situation that is quite frequently encountered in practical situations, the MLE becomes significantly biased and the resulting inference becomes unreliable. In this section, we propose to use the JINI estimator with the MLE as initial estimator, and through a simulation study in two high dimensional settings, we compare its finite sample performance to the ones of the MLE, the BBC estimator (also based on the MLE) and a bias reduced estimator proposed by \citet{KoFi:09}. 

Specifically, we consider the logistic regression with (independent) responses $y_i \in \{0,1\}$, $i=1,\ldots,n$, with linear predictor $\x_i\tt\boldsymbol{\beta}$ where $\x_i\in\real^p$ is a vector of fixed covariates, and with logit link ${\mu}_i(\bm{\beta})\vcentcolon = \mathbb{E}[Y_i\vert \x_i]  =\exp(\x_i\tt\bm{\beta})/\left\{1+\exp(\x_i\tt\boldsymbol{\beta})\right\}$. The MLE for $\bm{\beta}$ 
%
%
is computed using the \texttt{glm} function in the \texttt{stats} R package. The bias reduced MLE of \citet{KoFi:09}, which we denote as BR-MLE, is computed using the \texttt{brglm} function (with default parameters) in the \texttt{brglm} R package \citep{kosmidis2017brglm}. 

We consider two simulation settings provided in Table~\ref{tab:sim-logistic}, that can possibly occur in practice, including  balanced (Setting I) and  unbalanced (Setting II) outcomes. We consider large models with $p=200$ and choose $n$ as to provide EPV of $5$ and $3.75$ respectively, both of which are below the usually recommended value of $10$. Moreover, as done in Section~\ref{sec:logistic-misscla}, the covariates are simulated independently from a $\mathcal{N}(0,4^2/n)$ for Setting I and from a $\mathcal{N}(0.6,4^2/n)$ for Setting II, in order to ensure that the size of the log-odds ratio $\mathbf{x}_i\bm{\beta}$ does not increase with $n$, so that $\mu_i(\bm{\beta})$ is not trivially equal to either $0$ or $1$. 

\begin{table}[!htb]
     \centering
     \caption{Simulation settings for the logistic regression}
     \begin{tabular}{lrr}
 \toprule
 Parameters &  Setting I & Setting II \\
 \midrule
 $n=$ & $2000$ & $3000$ \\
 $\sum_{i=1}^ny_i\approx$ & $1000$ & $750$ \\
 EPV $\approx$ & 5 & 3.75 \\
 $\beta_1=\beta_2=$ & $5$ & $5$ \\
 $\beta_3=\beta_4=$ & $-7$ & $-7$ \\
 $\beta_5=\ldots=\beta_{200}=$ & $0$ & $0$ \\
 $H=$ & $200$ & $200$ \\
 Number of simulations $=$ & $1000$ & $1000$ \\
 \bottomrule
     \end{tabular}
     \label{tab:sim-logistic}
\end{table}

The simulation results are presented in  Figure~\ref{fig:sim-logistic-smr} as the (absolute) bias and RMSE of the estimators. Boxplots of the finite sample distributions of the considered estimators as well as a simulation-based approximation of the bias function of the initial estimator (i.e. the MLE) are provided in Figures~\ref{fig:sim-logistic-bxp}~and~\ref{fig:logistic-bias}\footnote{See Supplementary Material~\ref{supp:bias} for a detailed explanation on how the graphs are built.} respectively. The bias function of the MLE for the logistic regression appears linear in the balanced setting and approximately locally linear in the unbalanced one, which is in line with the results of \cite{sur2019modern}. Figure~\ref{fig:sim-logistic-smr} (and Figure \ref{fig:sim-logistic-bxp}) supports favorably the theoretical findings of Section~\ref{sec:bias:consist}. Indeed, in both settings, we observe that the MLE is significantly biased (except for parameters with true value of zero). On the other hand, the BBC estimator, the BR-MLE and the JINI estimator have a reduced bias, with the BR-MLE and the JINI estimator achieving (apparent) unbiasedness, which, for the JINI estimator, is suggested by the first result of Theorem \ref{thm:bias:finite}. In addition, the bias corrected estimators have a reduced RMSE, which suggests, in particular, that the JINI estimator can significantly correct the bias of the MLE without loss of RMSE compared to other bias correction methods. 
\begin{figure}
    \centering
    \includegraphics[width=12cm]{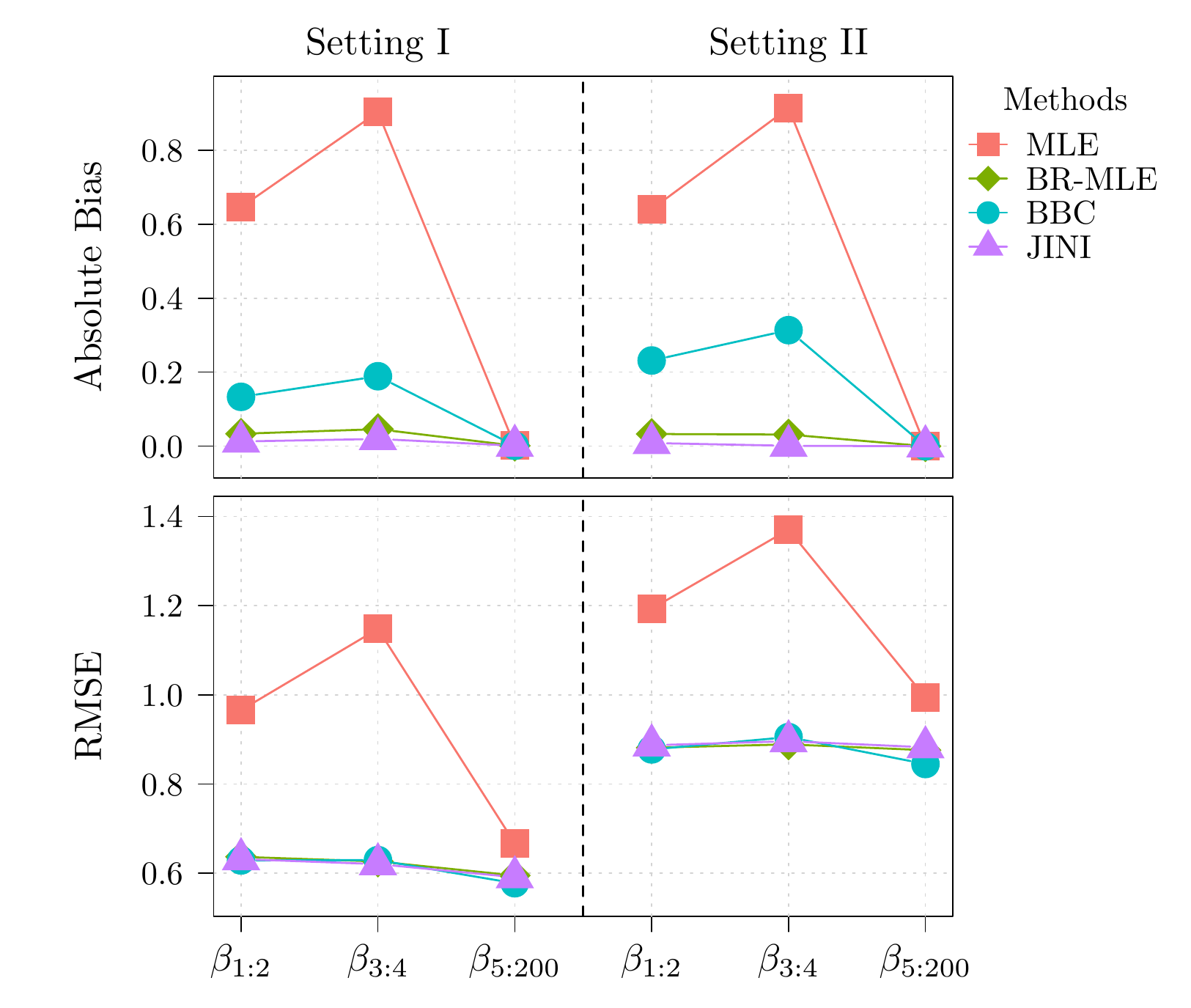}
    \caption{Finite sample absolute bias and RMSE of estimators for the logistic regression using the simulation settings presented in Table \ref{tab:sim-logistic}. The estimators are the MLE, the bias reduced MLE (BR-MLE), the BBC estimator and the JINI estimator based on the MLE. The estimators are grouped according to their parameter values.}
    \label{fig:sim-logistic-smr}
 \end{figure}
\begin{figure}
     \centering
     \includegraphics[width=12cm]{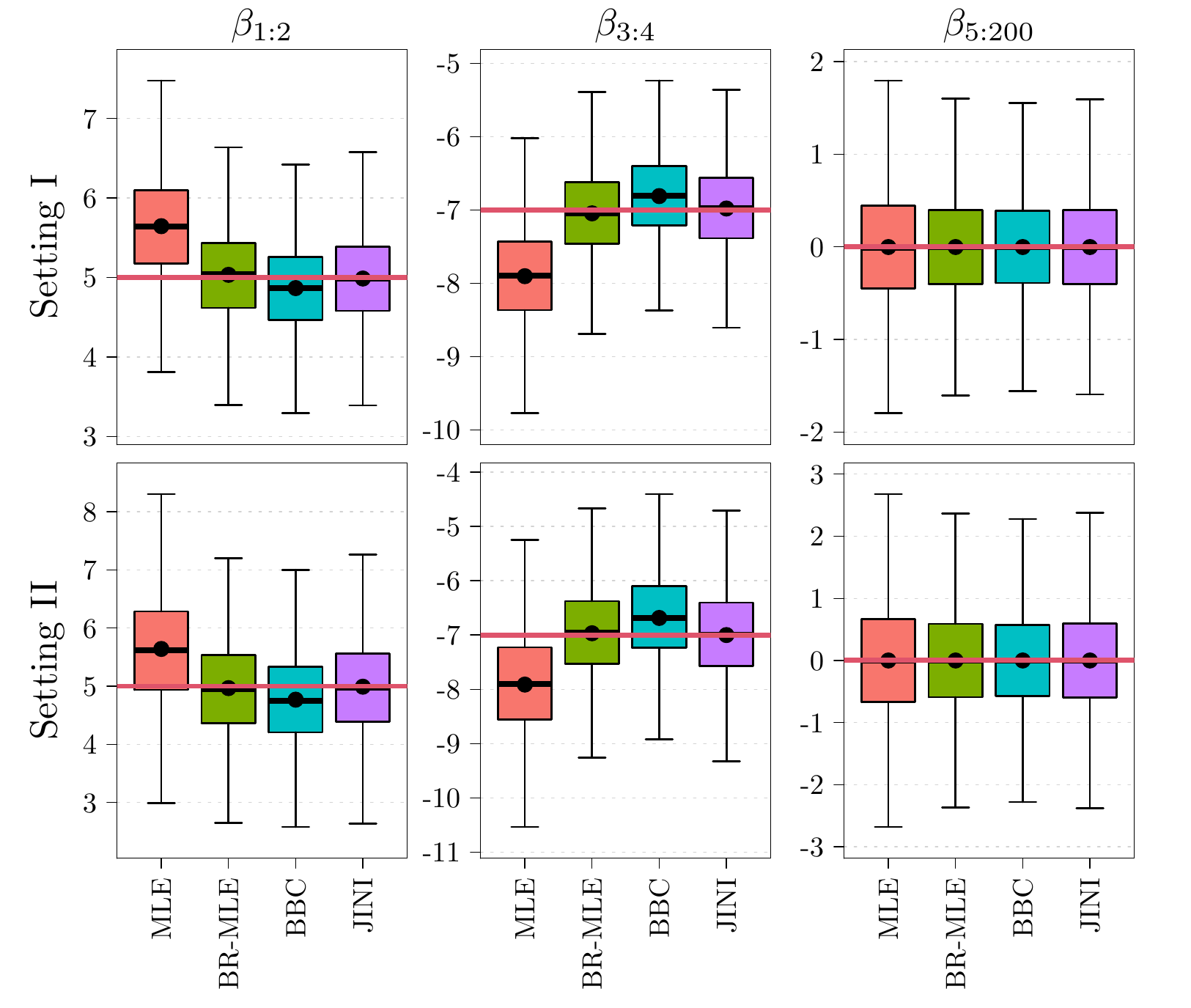}
     \caption{Finite sample distributions of estimators for the logistic regression using the simulation settings presented in Table \ref{tab:sim-logistic}. The estimators are the MLE, the bias reduced MLE (BR-MLE), the BBC and the JINI estimators based on the MLE. The  lines indicate the true parameter values and the dots locate the mean simulation values. The estimators are grouped according to their parameter values.}
     \label{fig:sim-logistic-bxp}
 \end{figure}
\begin{figure}
    \centering
     \includegraphics[width=12cm]{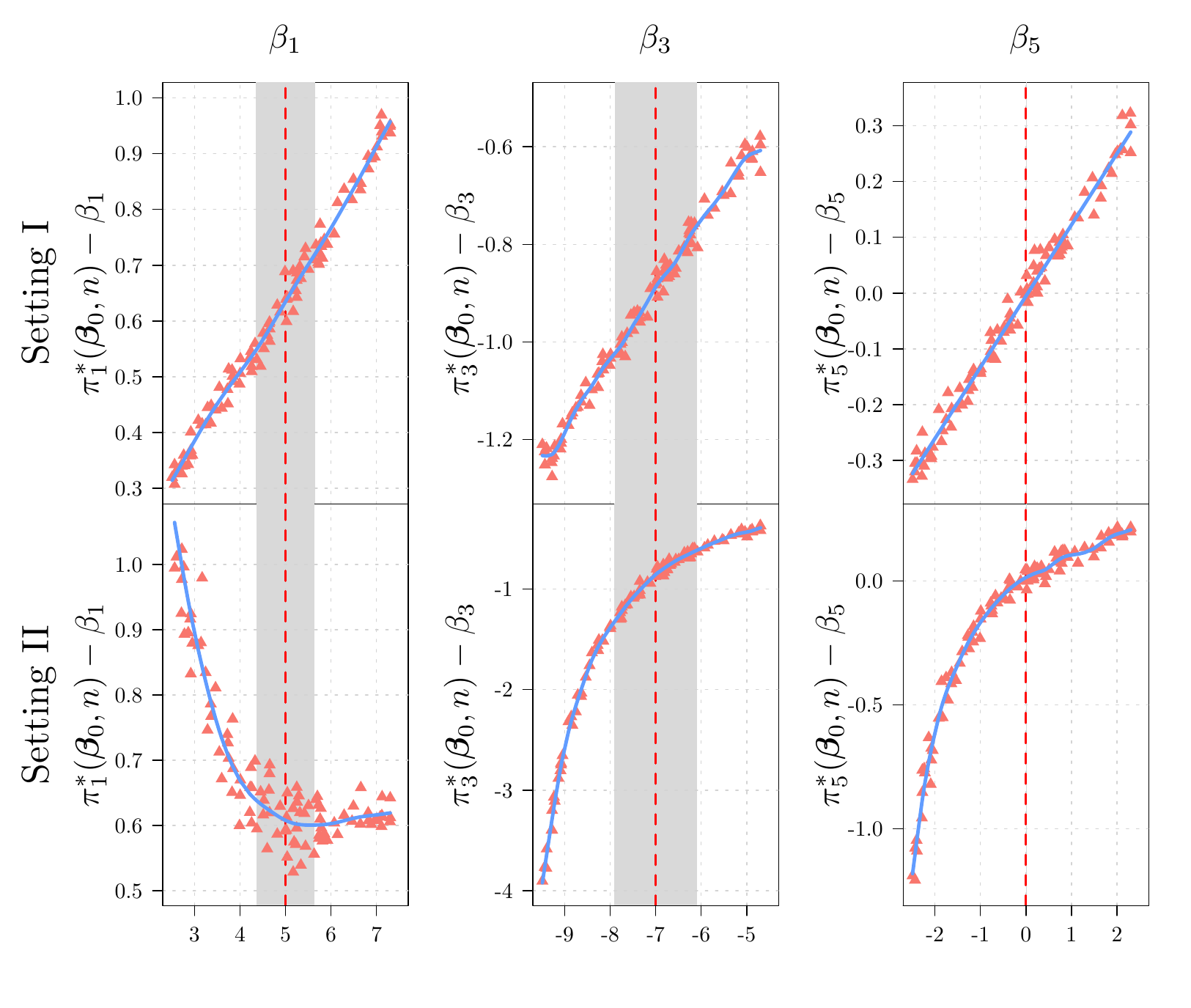}
     \caption{Simulation-based approximation of the bias function of the initial estimator (i.e. the MLE) of the logistic regression using the simulation settings presented in Table~\ref{tab:sim-logistic}. The dashed line corresponds to the true parameter value used in the simulation and the rectangle region represents a plausible neighborhood of the true parameter value. Notice that the rectangle regions of $\beta_5$ in both settings are barely visible as there is approximately no empirical (absolute) bias of the initial estimator.}
     \label{fig:logistic-bias}
\end{figure}

\newpage
\section{Logistic Regression with Random Intercept}
\label{supp:glmm}

A natural way of extending the logistic regression to account for the dependence structure between observed responses is to use the GLMM family \citep[see e.g.][and the references therein]{LeNe:01,McCuSe:01,JiangBook2007}. In this section, we consider a commonly used model in practice, namely the logistic regression with random intercept model. Specifically, we denote each binary response as $y_{ij}$, where $i=1,\ldots,m$ with $m \in \mathbb{N}^+$ as the number of clusters and $j=1,\ldots,n_i$ with $n_i \in \mathbb{N}^+$ as the number of observations for the $i^{th}$ cluster. Then, the (conditional) expected value of the response can be expressed as
\begin{equation*}
    \mu_{ij}(\bm{\beta}\vert U_i) \vcentcolon = \mathbb{E}\left[Y_{ij}\vert U_i\right] = \frac{\exp{\left( \mathbf{x}_{ij}\tt\bm{\beta}+U_i\right)}}{1+\exp{\left(\mathbf{x}_{ij}\tt\bm{\beta}+U_i\right)}},
\end{equation*}
where $\mathbf{x}_{ij}$ is the covariate vector of length $p$ associated to response $y_{ij}$, $\bm{\beta}$ is the vector of fixed effects parameters of length $p$, and $U_i$ is a normal random variable with zero mean and unknown variance $\sigma^2$ which represents the random effect associated to the $i^{th}$ cluster.

Since the random effects are not observed, the MLE for the fixed effects $\bm{\beta}$ is derived on the marginal likelihood function, where the random effects are integrated out. As these integrals have no closed-form solutions, approximations to the marginal likelihood function have been proposed, including Penalized Quasi-Likelihood (PQL) \citep[see e.g.][]{BrCl:93}, Laplace Approximations (LA) \citep[see e.g.][]{RaYaYo:00} and adaptive Gauss-Hermite Quadrature (GHQ) \citep[see e.g.][]{PiCh:06}. However, it is known that PQL leads to biased estimators, and both LA and GHQ are more accurate but are often numerically unstable especially when $p/n$ is relatively large \citep[see e.g.][]{BOLKER2009127,KiChEm:13}. In this section, we consider an approximation of the MLE defined through Penalized Iteratively Reweighted Least Squares (PIRLS) \citep[see e.g.][]{BaMaBo:10}. This approximated MLE  is relatively accurate while remaining computationally efficient. It is implemented in the \texttt{glmer} function  of the \texttt{lme4} R package (with argument \texttt{nAGQ} set to 0). We propose here to consider the BBC and the JINI estimators with the approximated MLE as initial estimator. 

By means of a simulation study with settings presented in Table \ref{tab:sim-glmm}, the finite sample performance in terms of (absolute) bias and RMSE of the approximated MLE, the BBC and the JINI estimators are compared. We consider two high dimensional simulation settings, with, in Setting I, a small number of clusters ($m=5$) and a large number of measurements per cluster ($n_i=50 \; \forall i$), a situation that is frequently encountered in cluster randomized trials \citep[see e.g.][and the references therein]{CRT-Huang-2016,CRT-Leyrat-2017}. On the other hand, in Setting II the number of clusters ($m=50$) is larger than the number of measurements within clusters ($n_i=5 \; \forall i$), which possibly reflects a more frequently encountered case. The covariates  $x_{ij,2}, \ldots, x_{ij,30}$ are simulated independently from a $\mathcal{N}(0, 2^2/n)$, and $x_{ij,1}$ is the intercept. 

\begin{table}
     \centering
     \caption{Simulation settings for the logistic regression with random intercept}
     \begin{tabular}{lrr}
 \toprule
 Parameters &  Setting I & Setting II \\
 \midrule
 $m=$ & $5$ & $50$ \\
 $\forall i, \;n_i=$ & $50$ & $5$\\
 $n=$ & $250$ & $250$ \\
 $\sum_{i=1}^m\sum_{j=1}^{n_i} y_{ij}\approx$ & $125$ & $125$   \\
 EPV $\approx$ & $4$ & $4$ \\
 $\beta_1=$ & $0$ & $0$ \\
 $\beta_2=\beta_3=$ & $5$ & $5$ \\
 $\beta_4=\beta_5=$ & $-7$ & $-7$ \\
 $\beta_6=\ldots=\beta_{30}=$ & $0$ & $0$ \\
 $\sigma^2=$ & $2$ & $2$ \\
 $H=$ & $200$ & $200$ \\
 Number of simulations $=$ & $1000$ & $1000$ \\
 \bottomrule 
     \end{tabular}
     \label{tab:sim-glmm}
\end{table}

The (absolute) bias and RMSE of the estimators are provided in Figure \ref{fig:sim-glmm-summary} while the boxplots of the finite sample distributions and the simulation-based approximation of the bias function of the initial estimator (i.e. the approximated MLE) are provided in Figures~\ref{fig:sim-glmm-boxplot}~and~\ref{fig:glmm-bias}\footnote{See Supplementary Material~\ref{supp:bias} for a detailed explanation on how the graphs are built.} respectively. The bias function of the MLE appears to be locally linear in both settings, and we observe, in Figure \ref{fig:sim-glmm-summary} (and Figure \ref{fig:sim-glmm-boxplot}) that the MLE is significantly biased (except for the fixed effects parameters with true value of zero). The BBC estimator does not correct the bias in general, and in Setting II, the BBC estimator introduces more bias than the MLE for non-zero fixed effects parameters. On the other hand, the JINI estimator drastically corrects the bias of the MLE in both settings, including for the random effect variance. In terms of RMSE, the BBC and the JINI estimators have approximately the same RMSE in both settings. Moreover, their RMSE are roughly the same as the MLE in Setting I, and are lower than the MLE in Setting II. Such bias correction outperformance of the JINI estimator without loss of RMSE brings important advantages when performing inference in practice, and/or when the parameter estimates, such as the random effect variance estimate, are used, for example, to evaluate the sample size needed in subsequent randomized trials. The JINI estimator can possibly be based on other initial estimators in order to further improve efficiency for example, however this study is left for further research.

\begin{figure}
    \centering
    \includegraphics[width=12cm]{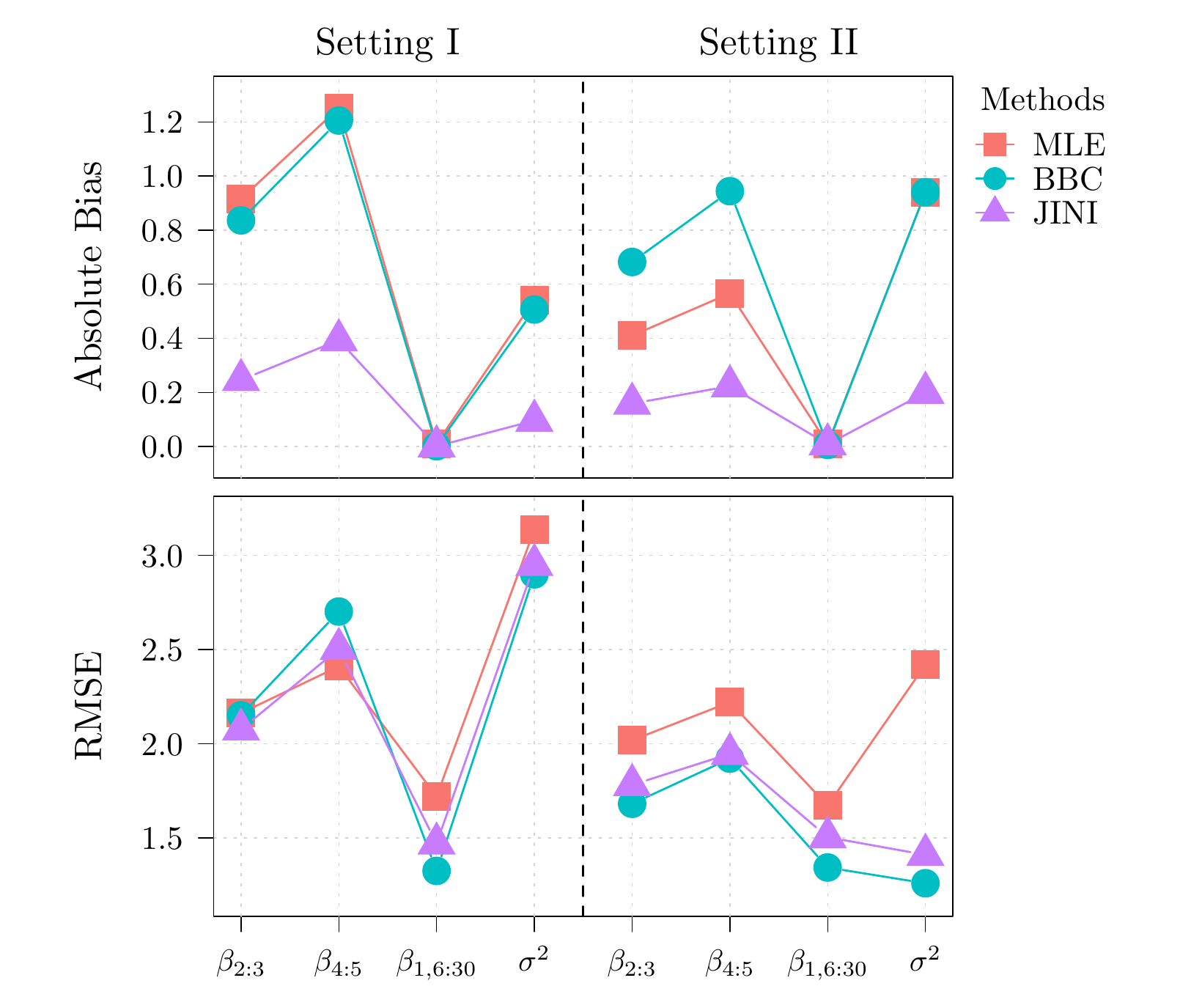}
   \caption{Finite sample absolute bias and RMSE of estimators for the logistic regression with random intercept, using the simulation settings presented in Table \ref{tab:sim-glmm}. The estimators are the approximated MLE, the BBC and the JINI estimators with the approximated MLE as initial estimator. The estimators are grouped according to their parameter values.}
    \label{fig:sim-glmm-summary}
\end{figure}
\begin{figure}
    \centering
    \includegraphics[width=12cm]{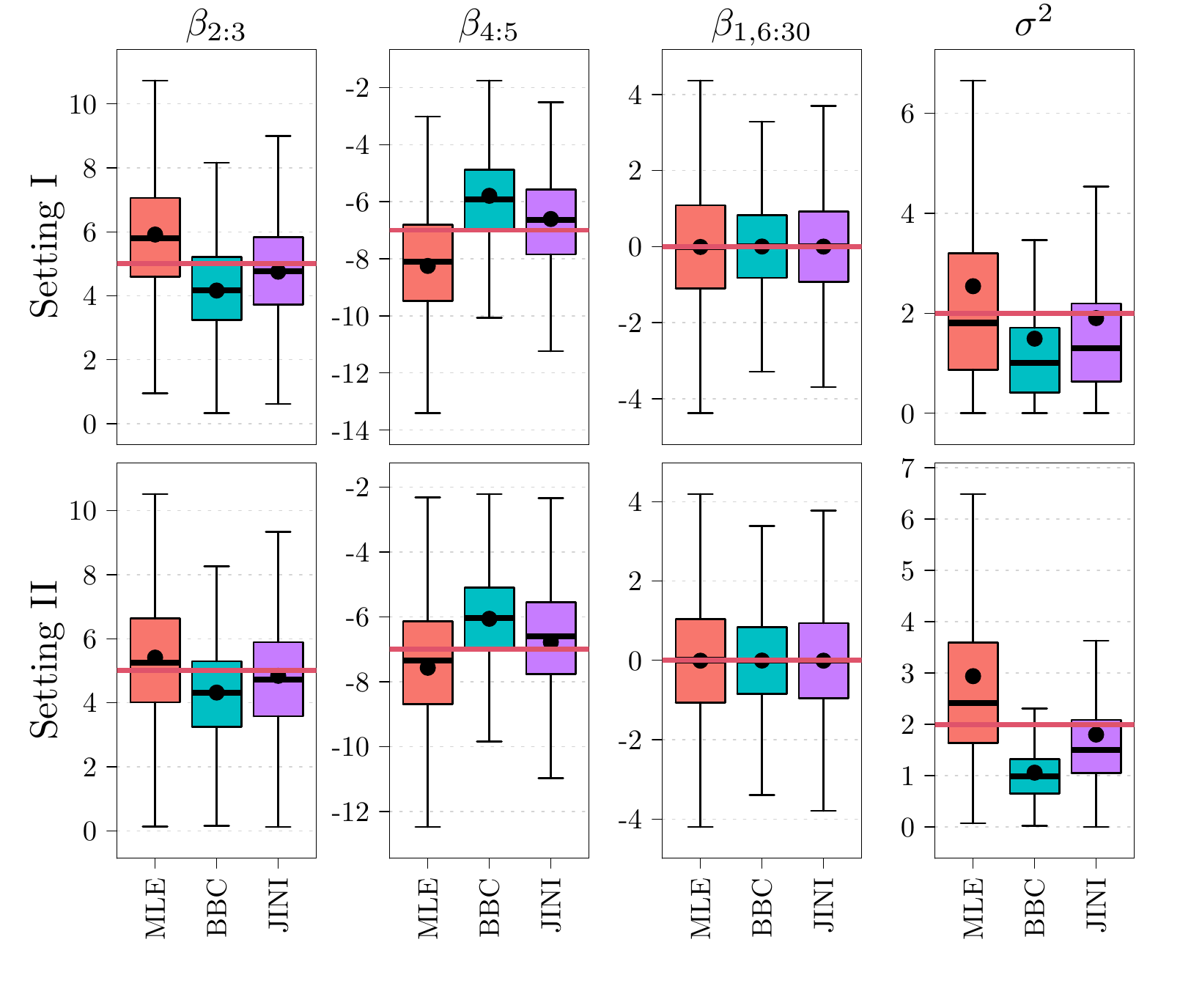}
   \caption{Finite sample distribution of estimators for the logistic regression with random intercept, using the simulation settings presented in Table \ref{tab:sim-glmm}. The estimators are the approximate MLE, the BBC and the JINI estimators with the approximate MLE as initial estimator. The lines indicate the true parameter values and the dots locate the mean simulation values. The estimators are grouped according to their parameter values.}
    \label{fig:sim-glmm-boxplot}
\end{figure}
\begin{figure}
    \centering
     \includegraphics[width=12cm]{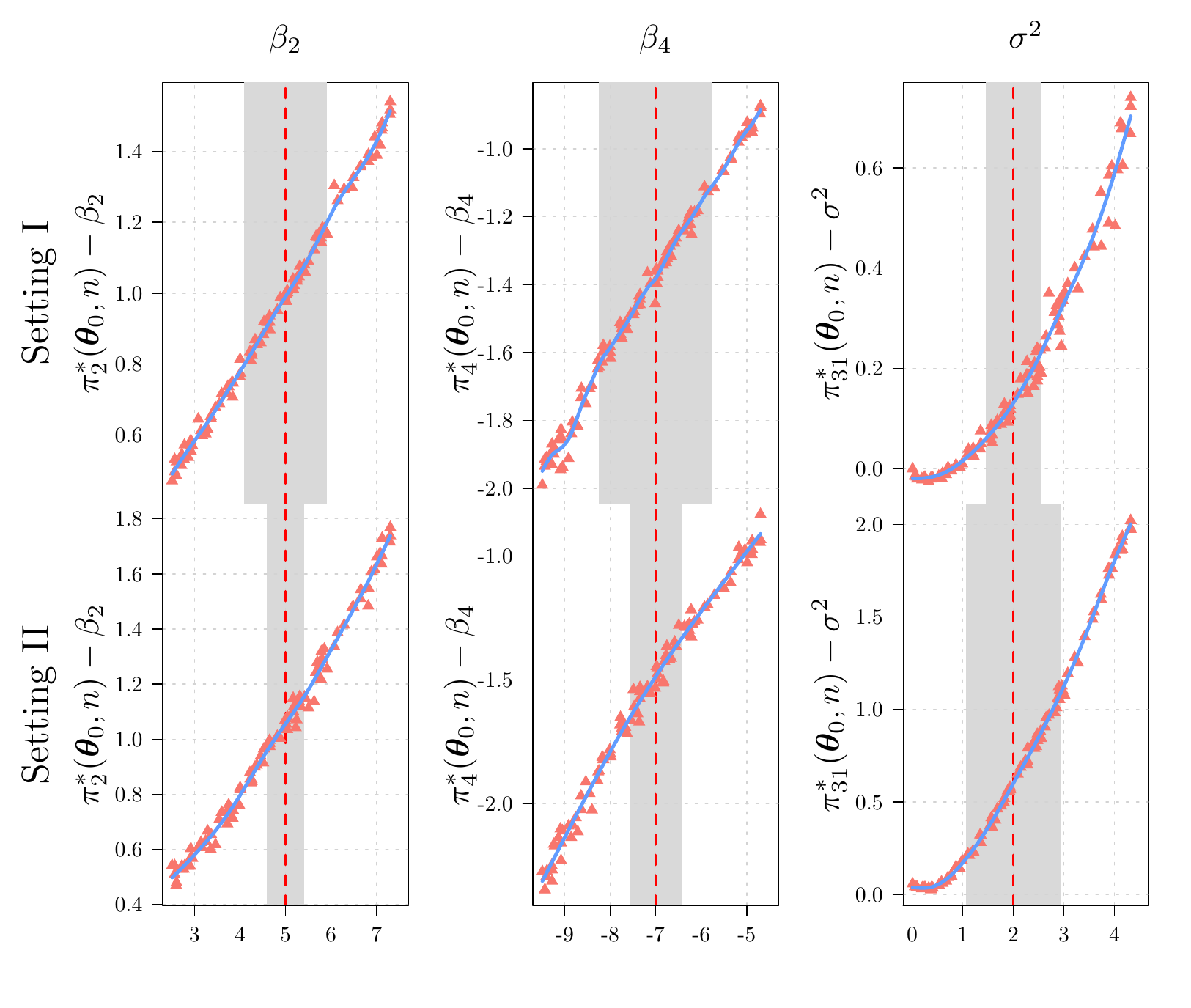}
     \caption{Simulation-based approximation of the bias function of the initial estimator (i.e. the approximated MLE) of the logistic regression with random intercept using the simulation settings presented in Table~\ref{tab:sim-glmm}. The dashed  line corresponds to the true parameter value used in the simulation and the rectangle region represents a plausible neighborhood of the true parameter value.}
     \label{fig:glmm-bias}
\end{figure}

\newpage
\section{Censored Poisson Regression}
\label{supp:cens-pois}

The analysis of count response data given multiple covariates can usually be done using the Poisson regression, which belongs to the class of GLM. In practice, however, the data is often exposed to some form of censoring. For example, the observed count response can be censored above a known threshold, say $C$, which is imposed by the survey design or is used to reflect relevant theoretical constraints \citep[see e.g.][]{terza1985tobit}. In this case, due to the existence of censoring, the MLE for the classical (uncensored) Poisson regression becomes biased and inconsistent. So adjustments of the likelihood function are necessary to take into account the censoring mechanism \citep[see e.g.][]{Bran:92}. Alternatively, some researchers have also proposed the use of multiple imputation \citep{Rubin:1987} to deal with censored responses \citep[see e.g.][and the references therein]{VanBuur:06,VanBurr:18}. In general, the derivation and/or computation of the MLE based on a modified likelihood function can be challenging, especially when $p/n$ is relatively large and/or when a significant proportion of data is censored. 

Specifically, with a sample of $n$ independent subjects, we denote $y_i^*$ as the actual (unobserved) count data of the $i^{th}$ subject and define the censored (observed) count data as $y_i \vcentcolon = \min(y_i^*, C)$. 
The response $y_i^*$ is linked to the vector of (fixed) covariates $\mathbf{x}_i \in \real^p$, with $\bm{\beta}$ as the regression coefficients, through the exponential link function $\mu_i \vcentcolon = \mathbb{E}[Y^*_i \vert \mathbf{x}_i] = \exp(\mathbf{x}_i\tt \bm{\beta})$.
%
%
The MLE with censored Poisson responses can be computed using the \texttt{vglm} function (with the argument \texttt{family = cens.poisson}) in the \texttt{VGAM} R package \citep[see e.g.][]{yee2015vector}. As an alternative to the MLE for censored Poisson regression, we propose to construct the JINI estimator with the MLE for the uncensored Poisson regression (which is inconsistent in this case) as the initial estimator. This MLE is implemented in the \texttt{glm} function (with the argument \texttt{family = poisson}) in the \texttt{stats} R package. Then, we compare the performance of the JINI estimator to the MLE for the censored Poisson regression as a benchmark.  In this example, we simulate the covariates in the same way as in Section~\ref{sec:app:consist}: $x_{i1}$ is the intercept, $x_{i2}$ is from a $\mathcal{N}(0,1)$, $x_{i3}$ is such that the first half of values being zeros and the remaining being ones, and $x_{i4}, \ldots, x_{ip}$ are simulated from a $\mathcal{N}(0,4^2/n)$ independently. More details regarding the simulation setting can be found in Table \ref{tab:sim-cens-pois}.

Figure \ref{fig:cens-pois-smr} shows the (absolute) bias and RMSE of the MLE of the censored Poisson regression and the JINI estimator  based on the classical MLE of the Poisson regression that ignores the censoring mechanism. The boxplots of the finite sample distributions and the simulation-based approximation of the bias function of the initial estimator (i.e. the classical MLE that ignores the censoring mechanism) can be found in Figures~\ref{fig:cens-pois-bxp}~and~\ref{fig:cens-pois-bias}\footnote{See Supplementary Material~\ref{supp:bias} for a detailed explanation on how the graphs are built.} receptively. We observe in Figure~\ref{fig:cens-pois-smr} that the benchmark MLE, despite consistent, shows significant bias  (except for parameters with true value of zero) because of the high dimensional setting. On the other hand, although the JINI estimator is based on an inconsistent initial estimator, it shows almost no bias for all parameters. This is due to the rather quadratic bias function of the classical MLE that ignores the censoring mechanism as observed in Figure \ref{fig:cens-pois-bias}, as suggested by Theorem~\ref{thm:bias:asymp} in Section~\ref{sec:bias:inconsist}. In terms of RMSE, the JINI estimator is also comparable to (or even better than) the MLE for all parameters. This example highlights the bias correction advantage of the JINI estimator without loss of RMSE. Moreover, since it is based on an initial inconsistent estimator that is readily available, in cases where the MLE for the correct model is not numerically implemented, the JINI estimator can be used so that analytical and/or numerical challenges are avoided. Such an example is proposed in Section~\ref{sec:logistic-misscla}.

\begin{table}[!tb]
     \centering
     \caption{Simulation setting for the censored Poisson regression}
     \begin{tabular}{lr}
 \toprule
 Parameters &  Values \\
 \midrule
 $n=$ & $200$\\
 $\beta_1=$ & $0.5$ \\
 $\beta_2=$ & $0.8$ \\
 $\beta_3=$ & $-0.4$ \\
 $\beta_4= \ldots = \beta_{50} = $ & $0$ \\
 $C=$ & $5$\\
 Censoring percentage $\approx$ & $7\%$\\
 $H=$ & $200$ \\
 Number of simulations $=$ & $1000$ \\
 \bottomrule 
     \end{tabular}
     \label{tab:sim-cens-pois}
\end{table}

\begin{figure}
    \centering
     \includegraphics[width=12cm]{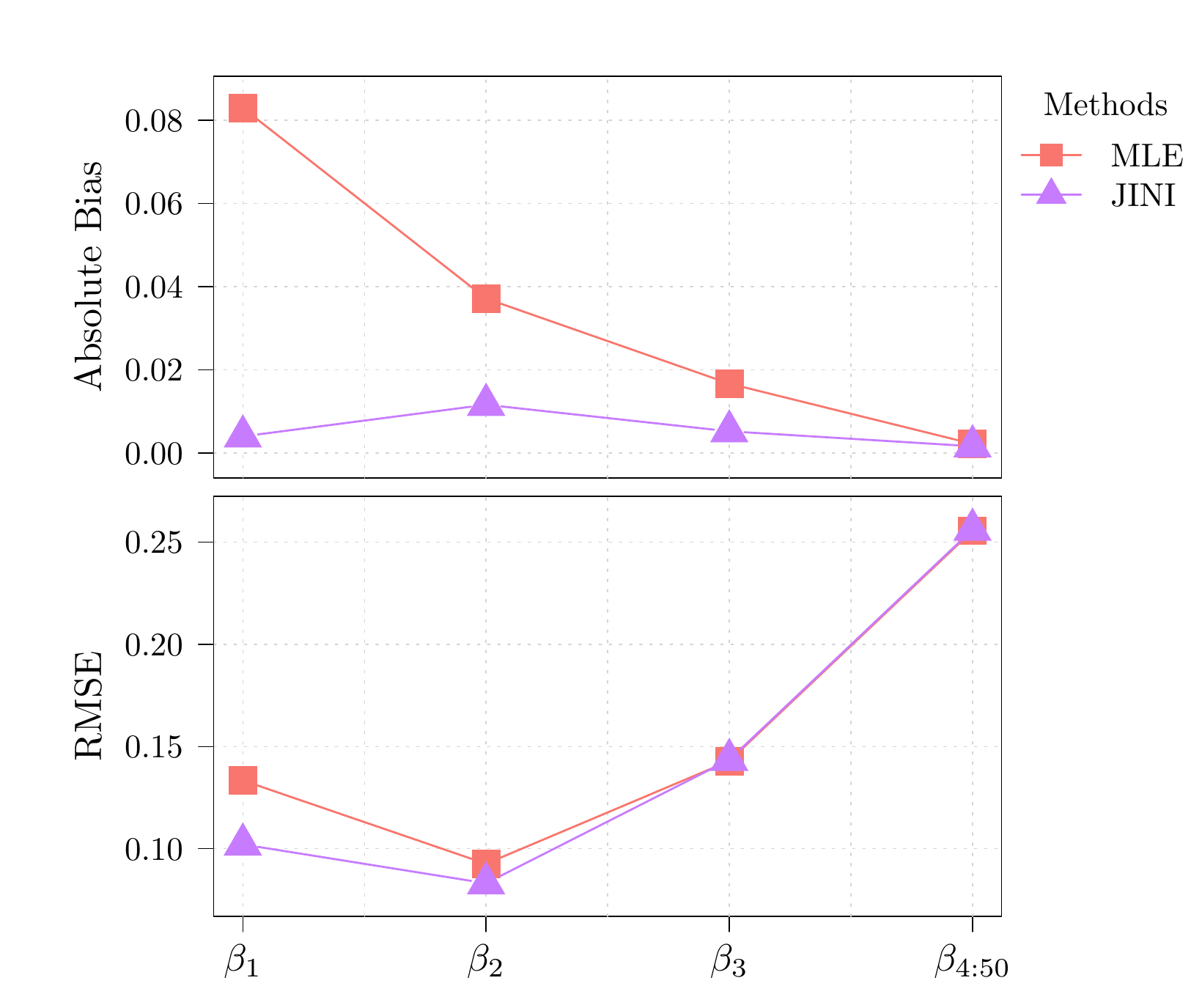}
     \caption{Finite sample absolute bias and RMSE of estimators for the censored Poisson regression  using the simulation setting presented in Table \ref{tab:sim-cens-pois}. The estimators are the MLE for the censored Poisson regression and the JINI estimator based on the classical MLE for the Poisson regression that ignores the censoring mechanism. The estimators are grouped according to their parameter values.}
     \label{fig:cens-pois-smr}
\end{figure}

\begin{figure}
    \centering
     \includegraphics[width=12cm]{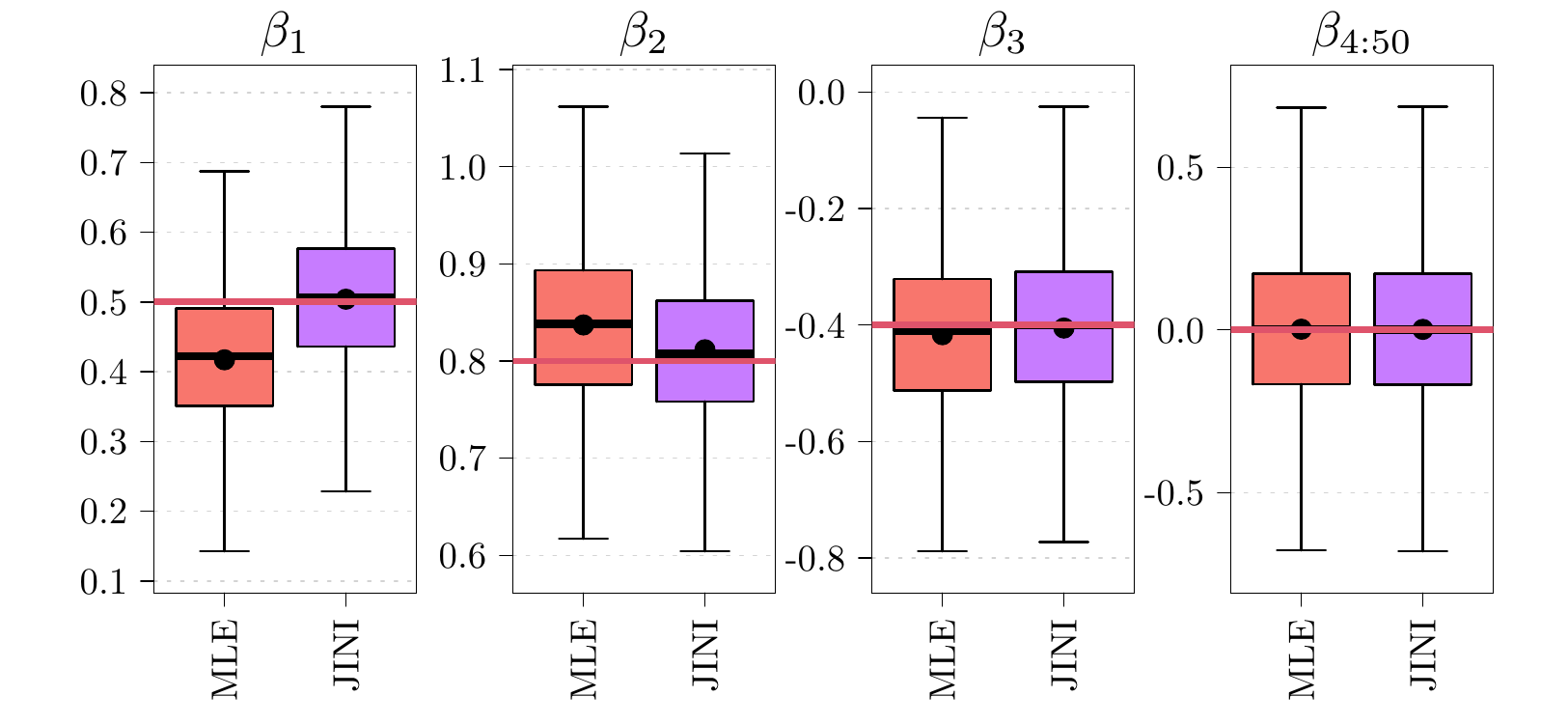}
     \caption{Finite sample distributions of estimators for the censored Poisson regression using the simulation setting presented in Table \ref{tab:sim-cens-pois}. The estimators are the MLE of the censored Poisson regression and the JINI estimator based on the classical MLE of Poisson regression that ignores the censoring mechanism. The lines indicate the true parameter values and the dots locate the mean simulation values. The estimators are grouped according to their parameter values.}
     \label{fig:cens-pois-bxp}
\end{figure}

\begin{figure}
    \centering
     \includegraphics[width=12cm]{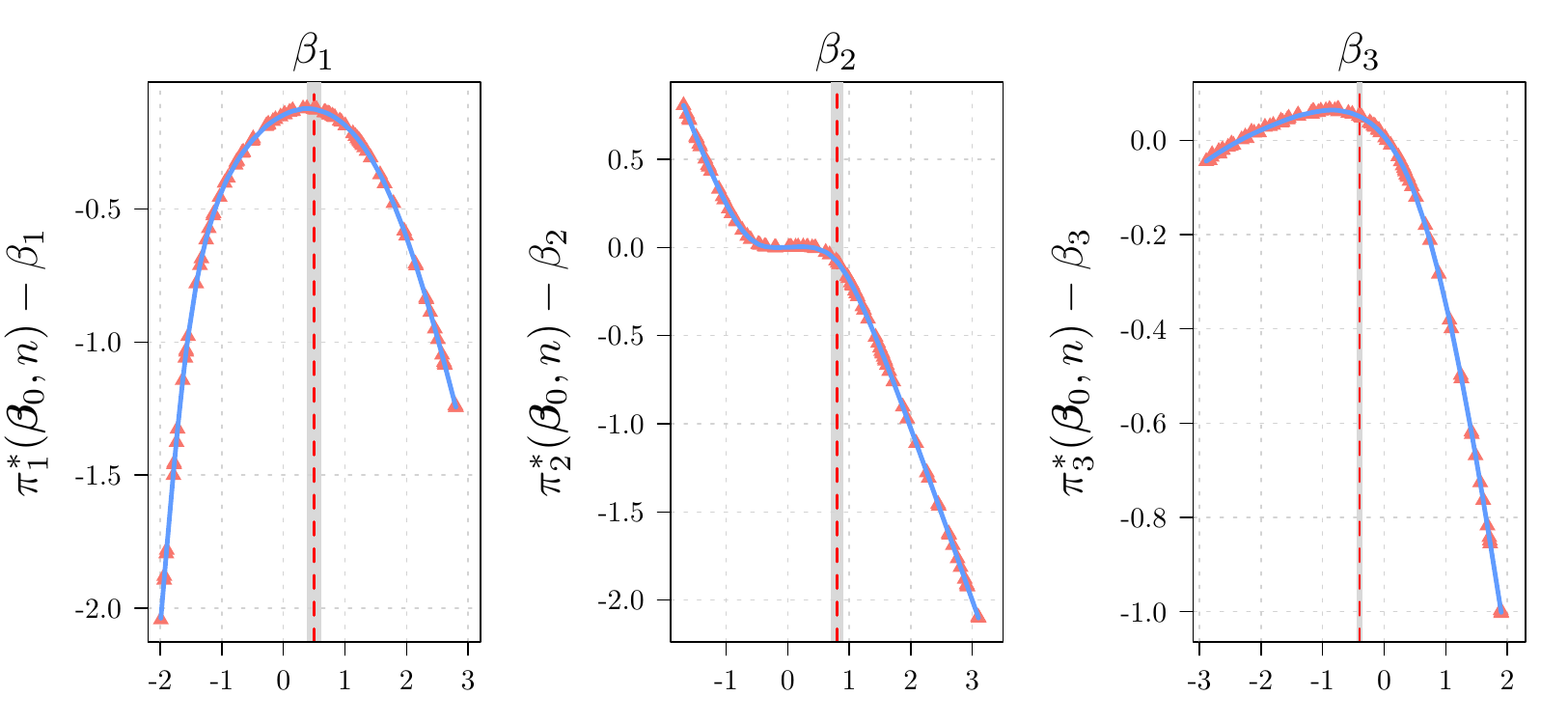}
     \caption{Simulation-based approximation of the bias function of the initial estimator (i.e. the classical MLE that ignores the censoring mechanism) of the censored Poisson regression using the simulation setting presented in Table \ref{tab:sim-cens-pois}. The dashed line corresponds to the true parameter value used in the simulation and the rectangle region represents a  plausible neighborhood of the true parameter value.}
     \label{fig:cens-pois-bias}
\end{figure}

\newpage
\section{Number of Iterations of the IB Algorithm}
\label{supp:ib-num-iter}

In this section, we present a figure that summarizes the number of iterations needed for the IB algorithm to construct the corresponding JINI estimators considered in Sections~\ref{sec:app:consist} and \ref{sec:app:inconsist} as well as in Supplementary Material~\ref{supp:logistic}. Figure~\ref{fig:ib-num-iter} suggests that the IB algorithm takes more iterations to converge when considering an inconsistent initial estimator. However, in all the situations we consider, the number of iterations needed for the IB algorithm appears to be relatively small and the computations typically converge in less than 15 iterations. 

\begin{figure}[ht]
    \centering
    \includegraphics[width=12cm]{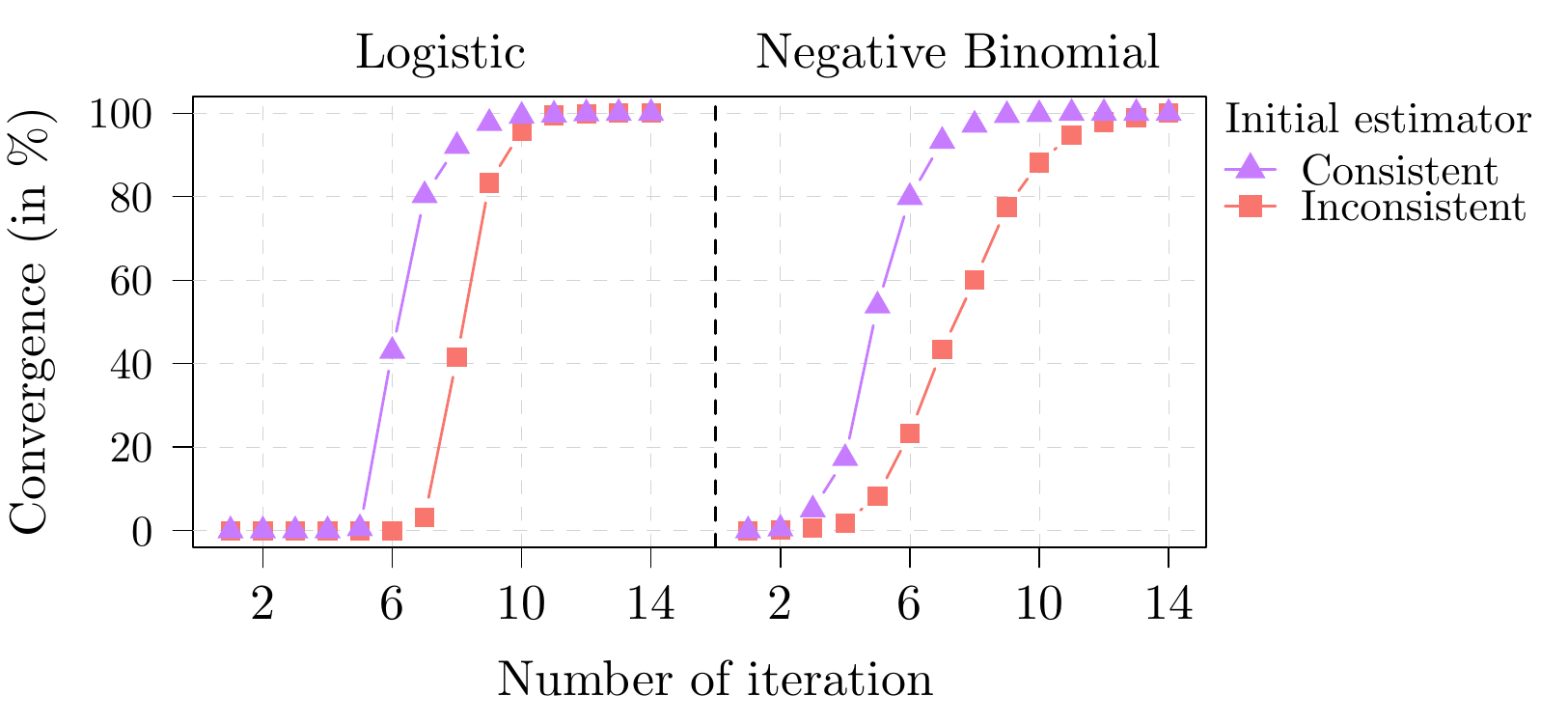}
    \caption{Summary of the number of iterations needed for the IB algorithm to construct the corresponding JINI estimators considered in Sections \ref{sec:app:consist} and \ref{sec:app:inconsist}, as well as in Supplementary Material~\ref{supp:logistic}. The graph on the left concerns the cases of the logistic regression. In this graph, the triangle dotted lines correspond to the case of Setting I of the logistic regression (Supplementary Material~\ref{supp:logistic}, with a consistent initial estimator), while the square dotted lines correspond to the case of Setting I of the logistic regression with misclassified responses (Section~\ref{sec:logistic-misscla}, with an inconsistent initial estimator). The graph on the right concerns the cases of the negative binomial regression. In this graph, the triangle dotted lines correspond to the case of negative binomial regression (Section~\ref{sec:app:consist}, with a consistent initial estimator), while the square dotted lines correspond to the case of negative binomial regression with censored responses (Section~\ref{sec:NB-censoring}, with an inconsistent initial estimator). 
    }
    \label{fig:ib-num-iter}
 \end{figure}

\newpage
\section{Boxplots of Finite Sample Distributions of Estimators}
\label{supp:boxplots}

In this section, we present the boxplots of the finite sample distributions of estimators considered in Sections \ref{sec:app:consist} and \ref{sec:app:inconsist}. More precisely, Figure~\ref{fig:sim-neg_bin-bxp} corresponds to the negative binomial regression (Section~\ref{sec:app:consist}), and Figure~\ref{fig:NB-censoring-bxp} corresponds to the censored negative binomial regression (Section~\ref{sec:NB-censoring}).

\begin{figure}[h]
     \centering
     \includegraphics[width=12cm]{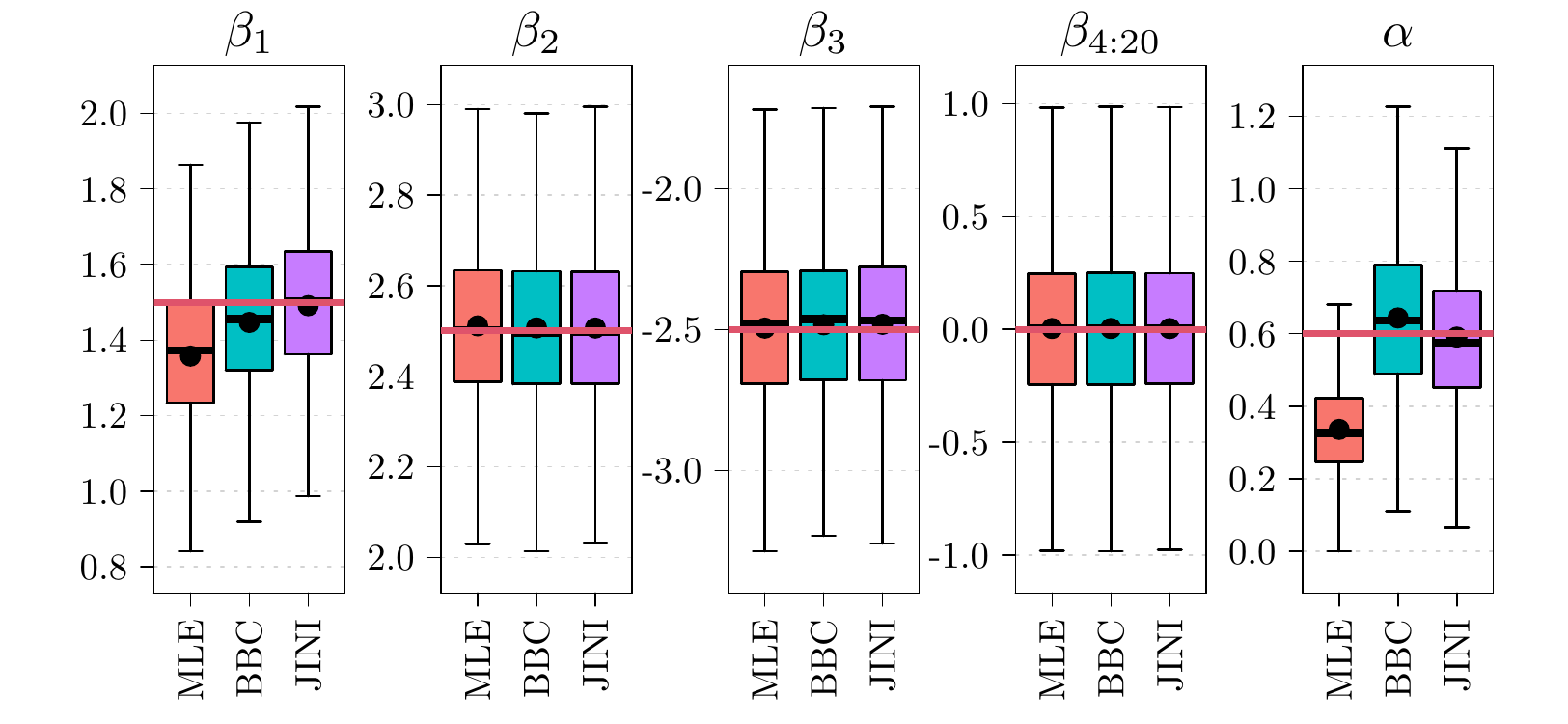}
     \caption{Finite sample distributions of estimators for the negative binomial regression using the simulation setting presented in Table \ref{tab:sim-negbin}. The estimators are the MLE, the BBC and the JINI estimators based on the MLE. The lines indicate the true parameter values and the  dots locate the mean simulation values. The estimators are grouped according to their parameter values.}
     \label{fig:sim-neg_bin-bxp}
\end{figure}

\begin{figure}
    \centering
     \includegraphics[width=12cm]{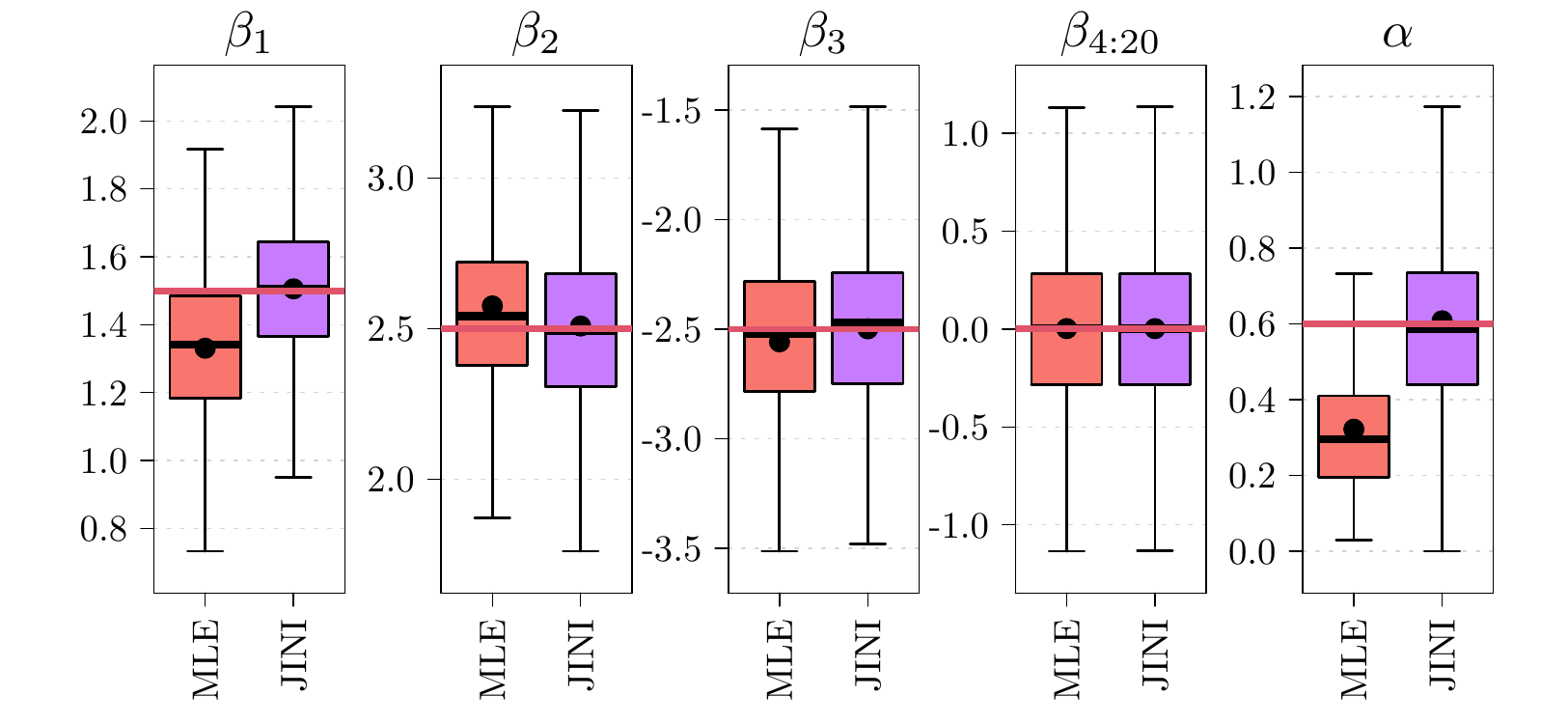}
     \caption{Finite sample distributions of estimators for the censored negative binomial regression  using the simulation settings presented in Table \ref{tab:sim-NB-censoring}. The estimators are the the MLE for the censored negative binomial regression model, and the JINI estimator with the classical MLE for the negative binomial regression model that ignores the censoring as initial estimator. The lines indicate the true parameter values and the dots locate the mean simulation values. The estimators are grouped according to their parameter values.}
     \label{fig:NB-censoring-bxp}
\end{figure}

\clearpage
\section{Simulation-Based Approximations of Bias Functions}
\label{supp:bias}

In this section, we present the simulation-based approximations of the bias functions of the initial estimators $\widehat{\bpi}(\bt_0, n)$ considered in Sections \ref{sec:app:consist} and \ref{sec:app:inconsist}, based on which we construct the JINI estimator (and the BBC estimator). We recall in Section~\ref{sec:existing:methods} that $\mathbf{d}^{*}(\bt, n) = \bpi^*(\bt, n) - \bt$, where $\bpi^*(\bt, n)$ is an approximation for $\bpi(\bt,n)$ given in \eqref{eqn:pi-star}. Therefore, to approximate the bias function,  
we use 
\begin{equation*}
    \mathbf{d}^{*}(\bt_i, n) = \bpi^*(\bt_i, n) - \bt_i,
\end{equation*}
where $\bt_i=[\theta_j]_{j=1,\ldots,p}$ with $\theta_j=\theta_{0j}\;\forall j\neq i$, while $\theta_j, j=i$, varies around $\theta_{0i}$. This is done similarly to \cite{sur2019modern} for graphical representations of the bias functions. 
%
%
In addition, in the computation of $\bpi^*(\bt,n)$ to obtain an unbiased approximation for $\bpi(\bt,n)$, we use $H=10^3$ when the sample size is relatively large (i.e. for the logistic regression in Figure \ref{fig:logistic-bias} and for the logistic regression with misclassified responses in Figure \ref{fig:sero-bias}), and $H=10^4$ for the other simulation settings.

Figures \ref{fig:nb-bias1} to \ref{fig:sero-bias} present the simulation-based approximations of the bias functions of the initial estimators considered in Sections \ref{sec:app:consist} and \ref{sec:app:inconsist}. In each graph, we only report the simulation-based approximations of the bias functions regarding three parameters with different values.

Moreover, we also include a  rectangle region which represents a plausible neighborhood of the true parameter. This neighborhood is obtained using the true parameter in the simulation settings plus and minus the empirical absolute bias of the initial estimators. One shall mainly refer to the form of the bias function inside this  rectangle region when validating the assumptions discussed in Section \ref{sec:bias}.

Overall, we observe that the bias functions of the initial estimators we consider appear to be relatively smooth, even linear, at least over a plausible neighborhood of the true parameter values. This highlights the fact that the assumptions considered in Section \ref{sec:bias} appear to be reasonable approximations to the true bias functions, and possibly provide a reasonable explanation to the outperformance of the JINI estimator in terms of bias correction that is in line with the theoretical findings discussed in Section~\ref{sec:bias}.

\begin{figure}
    \centering
     \includegraphics[width=12cm]{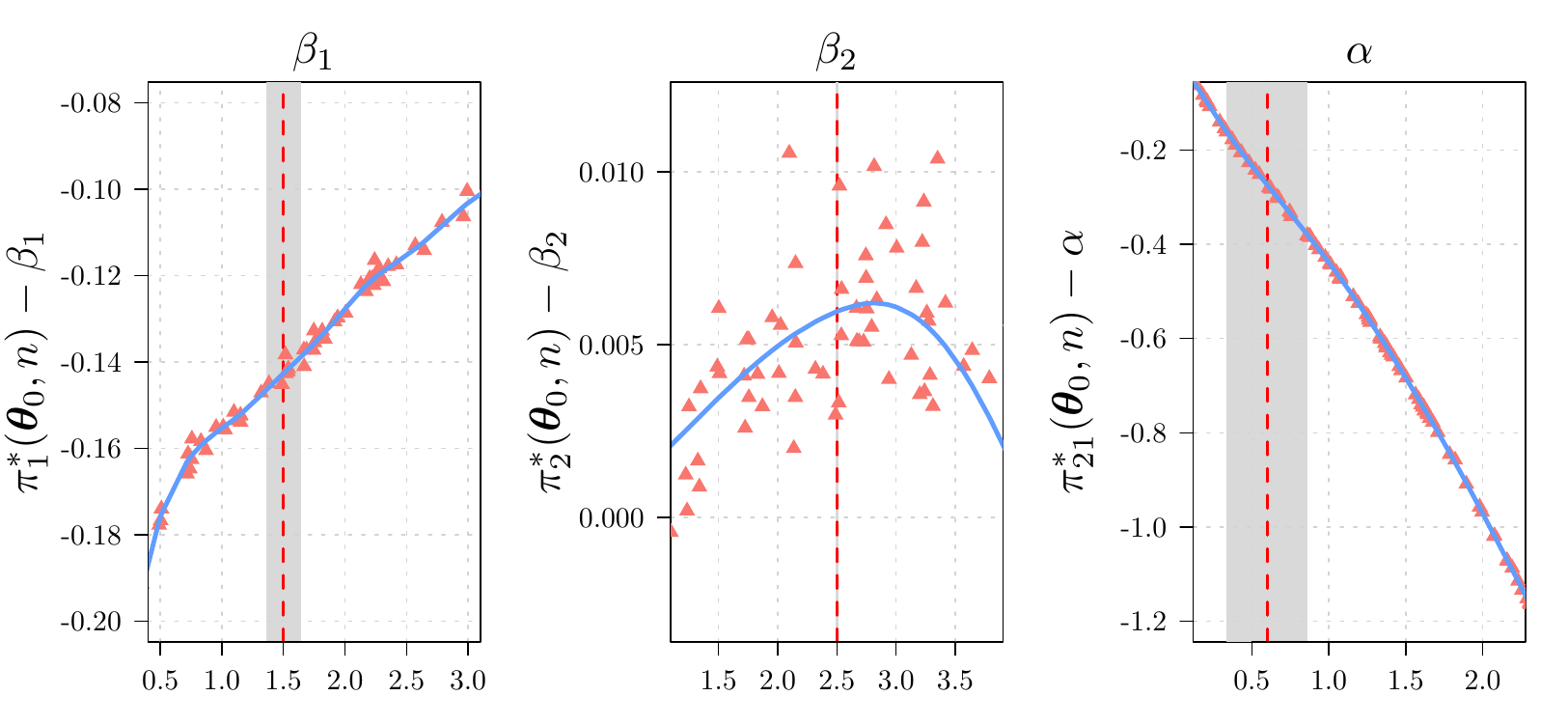}
     \caption{Simulation-based approximation of the bias function of the initial estimator (i.e. the MLE) of the negative binomial regression using the simulation setting presented in Table~\ref{tab:sim-negbin}. The dashed line corresponds to the true parameter value used in the simulation and the rectangle region represents a plausible neighborhood of the true parameter value. Notice that the rectangle region of $\beta_2$ is barely visible as there is approximately no empirical (absolute) bias of the initial estimator.}
     \label{fig:nb-bias1}
\end{figure}

\begin{figure}
    \centering
     \includegraphics[width=12cm]{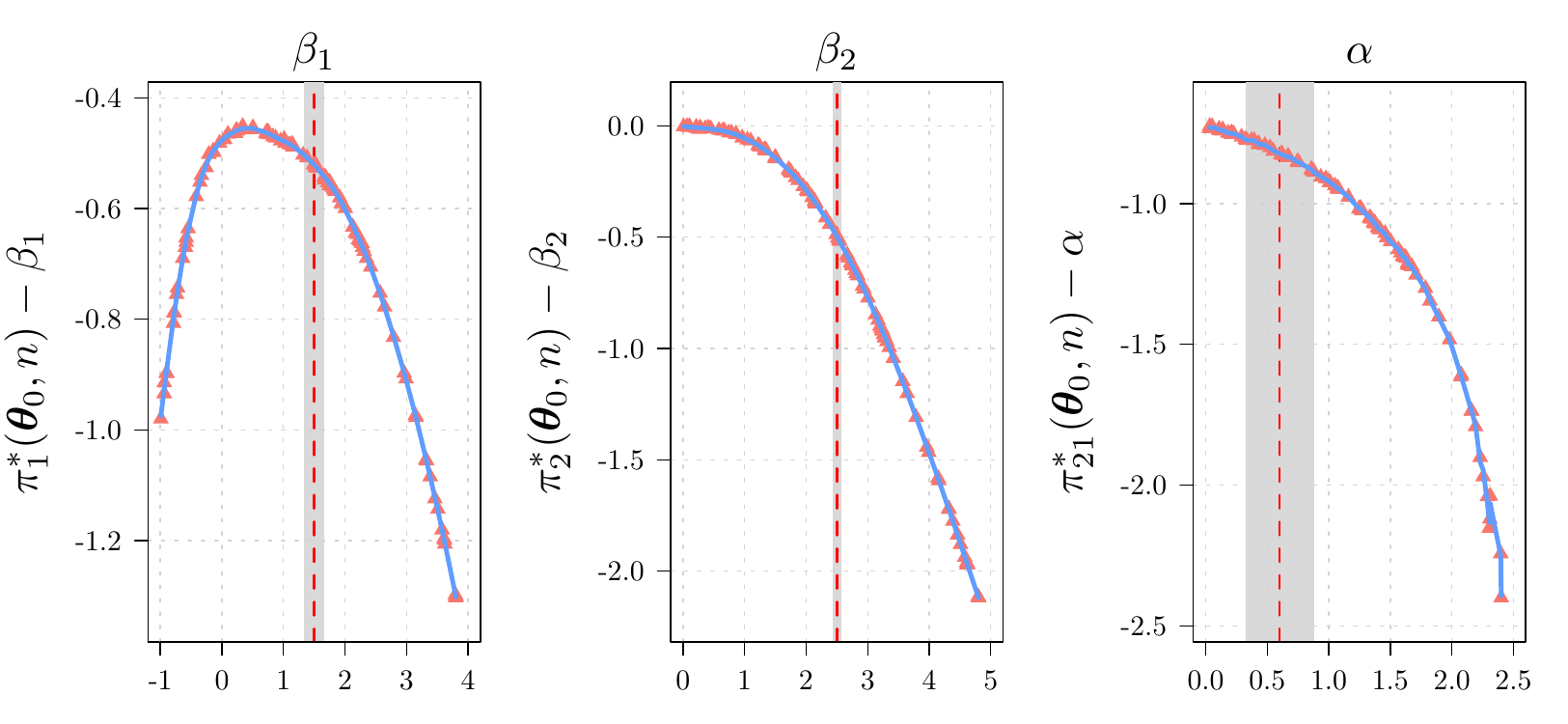}
     \caption{Simulation-based approximation of the bias function of the initial estimator (i.e. the classical MLE that ignores the censoring mechanism) of the censored negative binomial regression using the simulation settings presented in Table \ref{tab:sim-NB-censoring}. The dashed line corresponds to the true parameter value used in the simulation and the rectangle region represents a plausible neighborhood of the true parameter value. }
     \label{fig:nb-bias2}
\end{figure}

\begin{figure}
    \centering
     \includegraphics[width=12cm]{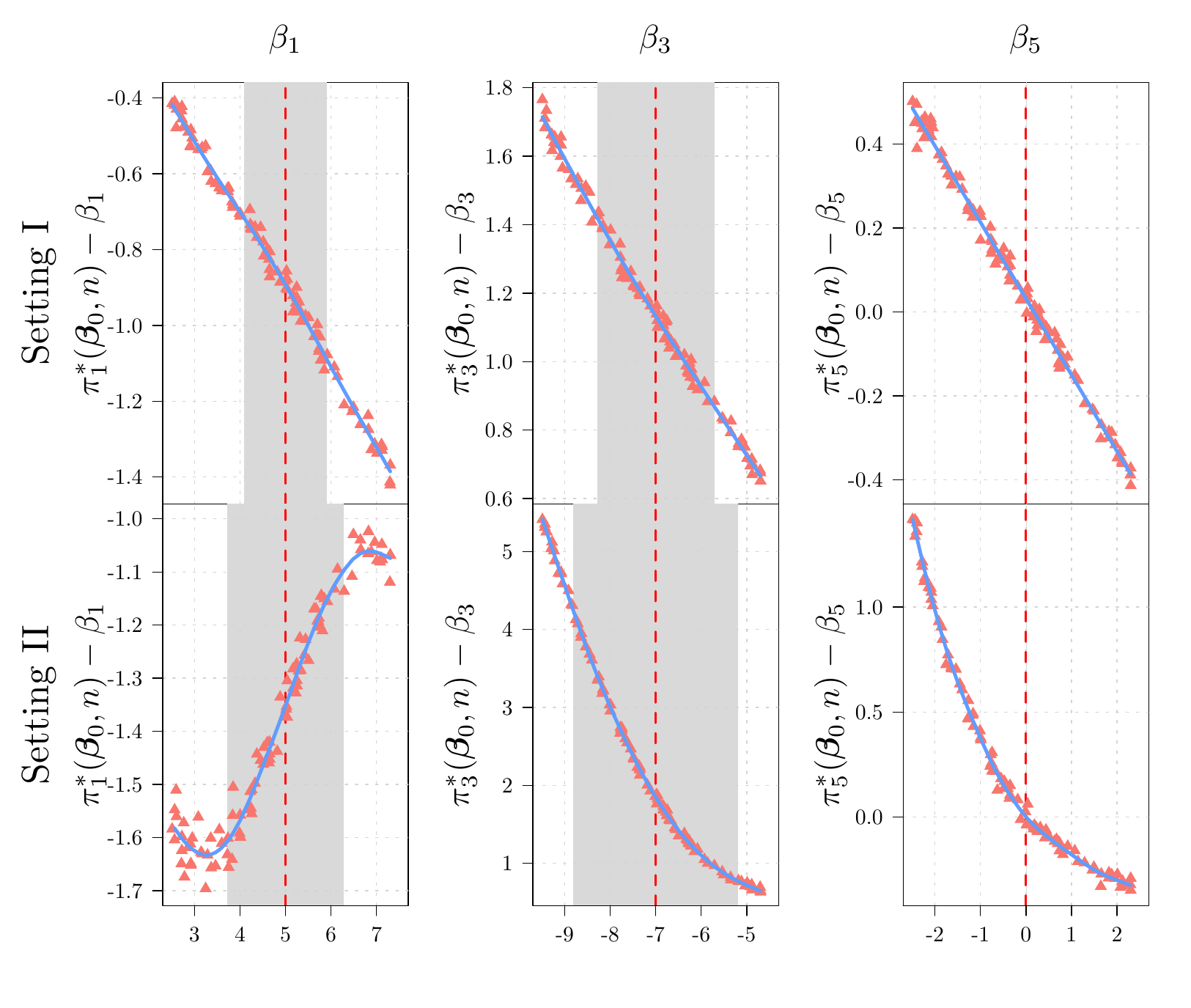}
     \caption{Simulation-based approximation of the bias function of the initial estimator (i.e. the classical MLE that ignores the misclassification) of the logistic regression with misclassified responses using the simulation settings presented in Table \ref{tab:sim-logistic-misscla}. The dashed line corresponds to the true parameter value used in the simulation and the rectangle region represents a plausible neighborhood of the true parameter value. Notice that the rectangle regions of $\beta_5$ in both settings are barely visible as there is approximately no empirical (absolute) bias of the initial estimator.}
     \label{fig:sero-bias}
\end{figure}

\clearpage
\section{ROC Curve Representation for the Simulation Section~\ref{sec:logistic-misscla}}
\label{supp:ROC}

\begin{figure}[hbt!]
    \centering
     \includegraphics[width=12cm, angle=0]{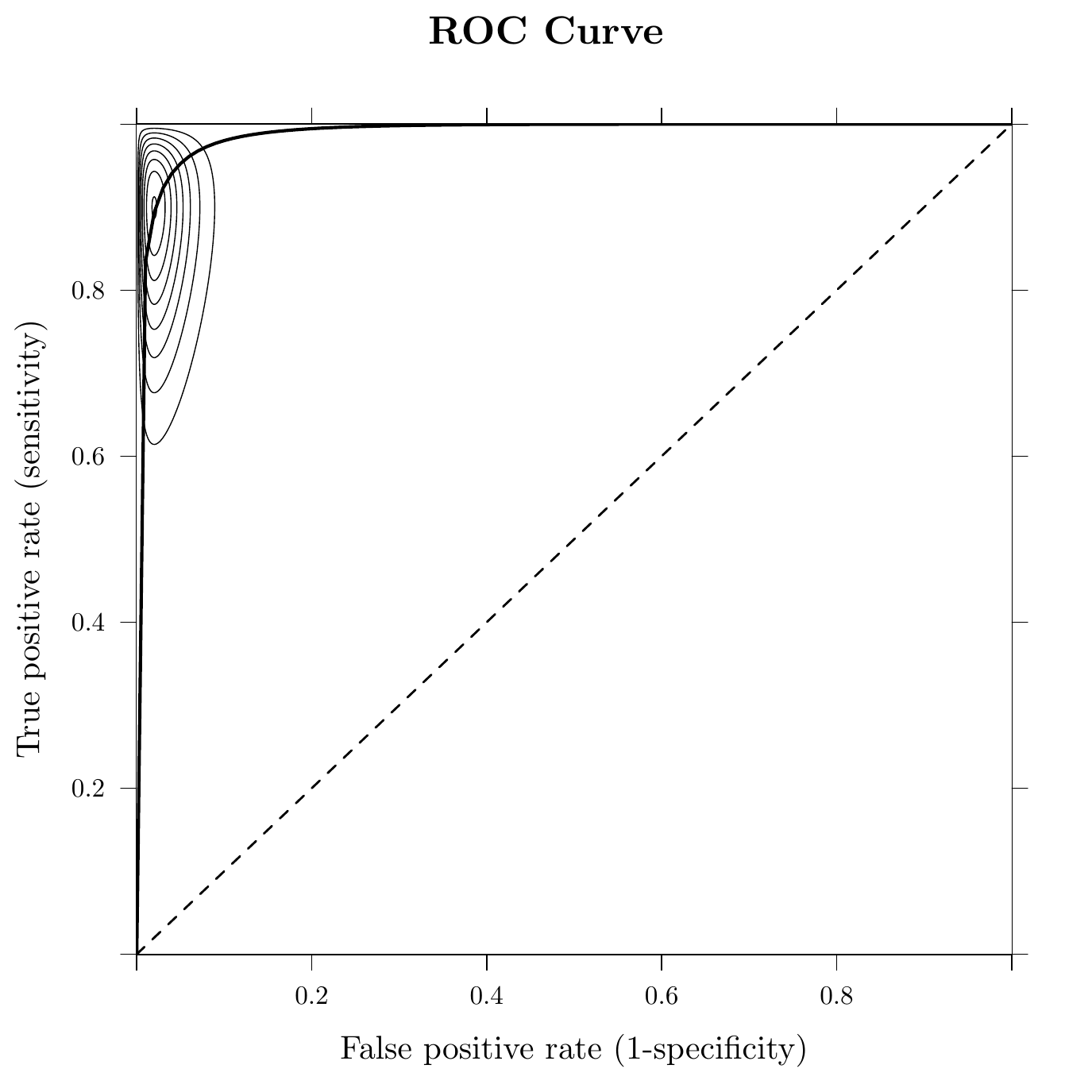}
     \caption{ROC curve for the logistic regression with misclassified responses using the simulation settings presented in Table~\ref{tab:sim-logistic-misscla}. The contour lines are the product of the densities of two beta distributions representing the false positive and true positive rates. The area under the curve is $99\%$.}
     \label{fig:roc}
\end{figure}

%





\end{document}